\documentclass[11pt,a4paper]{amsart}
\usepackage{amssymb, amstext, amscd, amsmath}

\usepackage{url}
\usepackage{amsfonts}
\usepackage{lscape}
\usepackage{amsthm}
\usepackage{amsfonts}
\usepackage{array}
\usepackage{eucal}
\usepackage{latexsym}
\usepackage{mathrsfs}
\usepackage{textcomp}
\usepackage{verbatim}
\usepackage{tikz}

\newtheorem{thm}{Theorem}[section]

\newtheorem{claim}[thm]{Claim}

\newtheorem{con}[thm]{Conjecture}

\newtheorem{defn}[thm]{Definition}
\newtheorem{example}[thm]{Example}
\newtheorem{lem}[thm]{Lemma}
\newtheorem{prop}[thm]{Proposition}

\numberwithin{equation}{section}

\newcommand{\bN}{{\mathbb{N}}}
\newcommand{\bQ}{{\mathbb{Q}}}
\newcommand{\bR}{{\mathbb{R}}}

  \newcommand{\C}{{\mathcal{C}}}
  \newcommand{\D}{{\mathcal{D}}}
  
  \newcommand{\F}{{S}}

  \newcommand{\I}{{\mathcal{I}}}

  \newcommand{\M}{{\mathcal{M}}}
  \newcommand{\N}{{\mathcal{N}}}

\renewcommand{\S}{{\mathcal{S}}}

  \newcommand{\X}{{\mathcal{X}}}
  \newcommand{\Y}{{\mathcal{Y}}}

\newcommand{\rank}{\operatorname{rank}}

\newcommand{\td}{\mathrm{td}\,}

\begin{document}
\title[Symmetry-forced rigidity of frameworks on surfaces]{Symmetry-forced rigidity of frameworks on surfaces}
\author[Anthony Nixon]{Anthony Nixon}
\address{Department of Mathematics and Statistics\\ Lancaster University\\ Lancaster\\
LA1 4YF \\ U.K. }
\email{a.nixon@lancaster.ac.uk}
\author[Bernd Schulze]{Bernd Schulze}
\address{Department of Mathematics and Statistics\\ Lancaster University\\ Lancaster\\
LA1 4YF \\ U.K. }
\email{b.schulze@lancaster.ac.uk}
\date{\today}

\maketitle

\begin{abstract}
A fundamental theorem of Laman characterises when a bar-joint framework realised
generically in the Euclidean plane admits a non-trivial continuous deformation of its vertices. This
has recently been extended in two ways. Firstly to frameworks that are symmetric with respect to
some point group but are otherwise generic, and secondly to frameworks in Euclidean 3-space that
are constrained to lie on 2-dimensional algebraic varieties.

We combine these two settings and consider the rigidity of symmetric frameworks realised on
such surfaces. First we establish necessary conditions for a framework to be symmetry-forced rigid
for any group and any surface by setting up a symmetry-adapted rigidity matrix for such frameworks
 and by extending the methods in \cite{jkt} to this new context. This
gives rise to several new symmetry-adapted rigidity matroids on group-labelled quotient graphs.
In the cases when the surface is a sphere, a cylinder or a cone we then also provide combinatorial
characterisations of generic symmetry-forced rigid frameworks for a number of symmetry groups,
including rotation, reflection, inversion and dihedral symmetry. The proofs of these results are based
on some new Henneberg-type inductive constructions on the group-labelled quotient graphs that correspond
to the bases of the matroids in question.
 For the remaining symmetry groups in 3-space - as well as for other types of surfaces - we
provide some observations and conjectures.
\end{abstract}

\section{Introduction}

A finite simple graph embedded into Euclidean space $\bR^d$ with vertices interpreted as universal joints and edges as stiff bars is known as a \emph{bar-joint framework}. We are interested in establishing from the combinatorics of the graph when it is possible to deform such frameworks. A framework is \emph{rigid} if there is no edge-length preserving continuous motion of the vertices which changes the distance between a pair of unconnected joints \cite{asiroth,GSS,W1}. Deciding the rigidity of a framework is typically an NP-hard problem \cite{Ab}. One way around this is to restrict attention to generic frameworks; that is, frameworks whose vertex coordinates form an algebraically independent set over $\bQ$.
A fundamental result in rigidity theory is Laman's theorem which gives a combinatorial characterisation of generic rigid frameworks in Euclidean 2-space \cite{Lamanbib}.
Finding combinatorial characterisations of generic rigid bar-joint frameworks in dimensions 3 and higher remains a key open problem in discrete geometry (see \cite{W1}, for example). However, very recently, Laman-type characterisations have been established for  generic rigid bar-joint frameworks in 3D whose joints are constrained to concentric spheres or cylinders or to surfaces which have a one-dimensional space of tangential motions (e.g., the torus or surfaces of revolution) \cite{nop1,no,nop} (see also Theorem \ref{thm:surf3}).

Over the last decade, a number of papers have studied when symmetry causes frameworks on a graph to become infinitesimally flexible, or stressed, and when it has no impact. These questions not only lead to many interesting and appealing mathematical results (see \cite{jkt,mt1,mt,owen,BSlaman,BSWWorbit,tan,LT}, for example) but they also have a number of important practical applications in biochemistry and engineering, since many natural structures such as molecules and proteins, as well as many human-built structures such as linkages and other mechanical machines, exhibit non-trivial symmetries (see \cite{FGsymmax,dimers,W1}, for example).
  
Of particular interest are symmetry-induced infinitesimal motions which are fully symmetric (in the sense that the velocity vectors are invariant under all symmtries of the framework), because for symmetry-generic configurations (i.e., configurations which are as generic as possible subject to the given symmetry constraints), the existence of a fully-symmetric infinitesimal motion guarantees the existence of a finite (i.e., continuous) motion which preserves the symmetry of the framework throughout the path \cite{KG1,BS6}. A symmetric framework which has no non-trivial fully symmetric motion is said to be \emph{symmetry-forced rigid} \cite{jkt,mt1,mt,tan}.

To detect fully symmetric infinitesimal motions in a symmetric framework,
 a symmetric analog of the rigidity matrix, called the orbit rigidity matrix, has
recently been constructed in \cite{BSWWorbit}.
The orbit rigidity matrix of a framework with symmetry group $S$ has one row for each edge orbit, and one set of columns for each vertex orbit under the  group action of $S$, and its entries can explicitly be derived in a very simple and transparent fashion (see \cite{BSWWorbit} for details). The key properties of the orbit rigidity matrix are that its kernel is isomorphic to the space of $S$-symmetric infinitesimal motions of the framework, and its co-kernel is isomorphic to the space of $S$-symmetric self-stresses of the framework. 
 Using  the orbit rigidity matrix, combinatorial characterisations of symmetry-forced rigid symmetry-generic frameworks have recently been established for a number of symmetry groups in the plane (under the assumption that the symmetry group acts freely on the framework joints) \cite{jkt,mt,mt1}.

In this paper, we extend these concepts and some of these combinatorial results to symmetric frameworks in 3D whose joints are constrained to surfaces. The type of a surface (see Definition \ref{d:surfacetype}) is the dimension of the space of tangential isometries.
The combinatorial descriptions, with or without symmetry, depend on this type.

In Section~\ref{sec:necesscond}, we first establish an orbit rigidity matrix for such frameworks.
We then adopt the methods recently described in \cite{jkt} and use this new matrix to  derive necessary
 conditions for symmetric frameworks on surfaces to be symmetry-forced rigid for any point group
 which is compatible with the given surface.

Furthermore, in Sections~\ref{sec:combchar}--\ref{sec:thms} we use the orbit-surface rigidity matrix to derive combinatorial characterisations of symmetry-forced rigid frameworks which are embedded generically with inversive or certain improper-rotational  (where an improper rotation is a rotation followed by a reflection in a plane perpendicular to the rotation axis) or dihedral symmetry on the sphere, with rotational, reflective or inversive symmetry on the cylinder or with rotational, reflective, inversive or certain improper-rotational symmetry on the cone. We prove the sufficiency of these combinatorial counts by first showing that a short list of Henneberg-type inductive operations is sufficient to recursively generate all of the appropriate classes of group-labeled quotient graphs (Section~\ref{sec:recchar}). Then we adapt results from \cite{nop,nop1,W2} to show that each of these operations  preserves the maximality of the rank of the orbit-surface rigidity matrix (Section~\ref{sec:geom}). A summary of the results is given in Section~\ref{sec:combchar}. (See also the tables in Section~\ref{sec:furtherwork}.)

We finish by providing a number of conjectures for some other groups and surfaces (Section~\ref{sec:furtherwork}). In particular, we briefly discuss an alternative 2-fold rotational symmetry on the cylinder: half turn symmetry with axis perpendicular to the cylinder. This situation induces a symmetry-preserving motion in a framework that counts to be minimally rigid without symmetry! For this case and some others we conjecture that the necessary counts we derived here are sufficient.


\section{Frameworks on surfaces} \label{sec:fwsurf}

In \cite{nop,nop1} frameworks supported on surfaces were considered. In particular, attention was paid to classical surfaces such as spheres, cylinders and cones.
Formally, let $\M$ be a 2-dimensional irreducible algebraic variety embedded in $\bR^3$.
We expect that with minor modifications our theorems and arguments can almost certainly be extended to certain reducible varieties. However these varieties must have the special property (parallel planes, concentric cylinders, etc.) that the dimension of the space of tangential isometries of $\M$ is the same as in each irreducible component.

A framework on a surface $\M\subseteq \mathbb{R}^3$ is a pair $(G, p)$, where $G$ is a finite simple graph
and $p:V(G)\to \M$ is a map such that $p(i)\neq p(j)$ for all $\{i,j\}\in E(G)$. We also say that $(G,p)$ is a  \emph{realisation} of the \emph{underlying graph} $G$ in $\mathbb{R}^3$ which is \emph{supported} on $\M$. For $i\in V(G)$, we say that $p(i)$ is the \emph{joint} of $(G,p)$ corresponding to $i$, and for $e=\{i,j\}\in E(G)$, we say that the line segment between $p(i)$ and $p(j)$ is the \emph{bar} of $(G,p)$ corresponding to $e$. For simplicity, we denote $p(i)$ by $p_i$ for $i\in V(G)$. 

An \emph{infinitesimal motion} of a framework $(G, p)$ on a surface $\M$ is a sequence $u$ of velocity vectors $u_1, \dots , u_{|V(G)|}$,
considered as acting at the framework joints, which  are tangential to the surface and
satisfy the infinitesimal flex requirement in $\bR^3$, $(u_i-u_j)\cdot(p_i-p_j) = 0$,
for each edge $\{i,j\}$.
It is elementary to show that $u$ is an infinitesimal motion if and only if $u$ lies in the
nullspace (kernel) of the rigidity matrix $R_\M(G,p)$ given in the following definition.
The submatrix of $R_{\M}(G,p)$ given by the first
$|E(G)|$ rows provides the usual rigidity matrix, $R_3(G,p)$ say, for the unrestricted framework $(G,p)$ (see \cite{W1}, for example).
The tangentiality condition corresponds to
$u$ lying in the nullspace of the matrix formed by the last $|V(G)|$ rows.

\begin{defn}\label{defn:surfacematrix}
The
\emph{rigidity matrix} $R_\M(G,p)$ of $(G,p)$ on $\M$ is an $|E(G)|+|V(G)|$ by $3|V(G)|$ matrix of the form \[\begin{bmatrix} R_3(G,p)\\ \N(p) \end{bmatrix}. \] Consecutive triples of columns in $R_\M(G,p)$ correspond to
framework joints. $R_3(G,p)$ is the usual rigidity matrix of $(G,p)$, that is,  the first $|E(G)|$ rows of $R_\M(G,p)$ correspond to the bars of $(G,p)$ and the entries in row $e=\{i,j\}$ are zero except possibly in the column triples for
$p_i$ and $p_j$, where the entries are the coordinates of $p_i-p_j$ and $p_j-p_i$ respectively. The final $|V(G)|$ rows of $R_\M(G,p)$ (i.e. the rows of $\N(p)$) correspond to the joints of $(G,p)$ and the entries
in the row for vertex $i$ are zero except in the columns for $i$ where the entries are the coordinates of a normal vector $N(p_i)$  to $\M$
at $p_i$.
\end{defn}

An infinitesimal motion of a framework $(G,p)$ on $\M$ is called \emph{trivial} if  it lies in the kernel of $R_\M(K_n,p)$, where $K_n$ is the complete graph on the vertex set of $G$. If every infinitesimal motion of $(G,p)$ is trivial, then $(G,p)$ is called \emph{infinitesimally rigid}. Otherwise $(G,p)$ is called \emph{infinitesimally flexible}. 

Let $\bQ(p)$ denote the field extension of $\bQ$ formed by adjoining the coordinates of $p$.
A framework $(G,p)$ on $\M$ is said to be \emph{generic} for $\M$ if $\td[\bQ(p):\bQ]=2|V(G)|$. 
This implies, \cite[Corollary $3.2$]{JMN}, that
any rational polynomial $h(x)$ in $3|V(G)|$ variables that satisfies $h(p)=0$ satisfies $h(q)=0$ for all points $q \in \M^{3|V(G)|}$.

A framework $(G,p)$ supported on $\M$ is called \emph{regular} if $\rank R_\M(G,p)= max \{\rank R_\M(G,q): q \in \M^{|V(G)|}\}$. If a framework on $\M$ is generic,
then it is clearly also regular. Moreover, if some realisation of a graph $G$ on $\M$ is infinitesimally rigid, then the same is true for every regular (and hence every generic) realisation of $G$ on $\M$.

\begin{thm}[\cite{nop1}]\label{thm:finitemotion}
A regular framework $(G,p)$ on an algebraic surface $\M$ is infinitesimally rigid if and
only if it is continuously rigid on $\M$.
\end{thm}

Note that the complete graphs $K_2$ and $K_3$ provide curiosities when $\M$ is a cylinder in that they are continuously rigid but have non-trivial vectors in their nullspaces which are not tangential isometries. For graphs with $|V(G)|\geq 6-k$ on a surface of type $k$ (see Definition \ref{d:surfacetype}) such worries disappear and both possible definitions of infinitesimal rigidity are equivalent.

\begin{defn}\label{d:surfacetype}
A  surface $\M$  is said to be of type $k$ if $\dim \ker R_{\M}(K_n,p)\geq k$ for all complete graph frameworks $(K_n,p)$ on $\M$ and $k$ is the largest such number.
\end{defn}

In other words $k$ is the dimension of the group $\Gamma$ of Euclidean isometries supported by $\M$.

 
A framework on a surface $\M$ is called \emph{isostatic} if it is minimally infinitesimally rigid, that is, if it is infinitesimally rigid and the removal of any bar results in an infinitesimally flexible framework. The following three results concerning generic isostatic frameworks on surfaces were recently established in \cite{nop,nop1}. 
 
\begin{thm}\label{thm:surf1}
Let $(G, p)$ be an isostatic generic framework on the algebraic surface $\M$  of type $k, 0\leq k\leq 3$,
with $G$ not equal to $K_1, K_2, K_3$ or $K_4$.
Then $|E(G)|=2|V(G)|-k$ and for every subgraph $H$ of $G$ with at least one edge,
$|E(H)|\leq 2|V(H)|-k$.
\end{thm}

\begin{lem}\label{lem:surf2}
Let $(G,p)$ be a regular framework on a surface $\M$ of type $k$. Then $(G,p)$ is
isostatic on $\M$ if and only if
\begin{enumerate}
\item $\rank R_{\M}(G,p) = 3|V(G)|-k$ and
\item $ 2|V(G)|-|E(G)| =k.$
\end{enumerate}
\end{lem}

A graph $G$ is called \emph{$(2,k)$-sparse} if for every subgraph $H$ of $G$, with at least one edge, we have $|E(H)|\leq 2|V(H)|-k$. A $(2,k)$-sparse graph satisfying $|E(G)|= 2|V(G)|-k$ is called \emph{$(2,k)$-tight}.

\begin{thm}\label{thm:surf3}
Let $G$ be a simple graph, let $\M$ be an irreducible algebraic surface in $\bR^3$ of type $k\in \{1,2\}$ and let $(G,p)$ be a generic framework on $\M$.
Then $(G,p)$ is isostatic on $\M$ if and only if $G$ is $K_1, K_2, K_3, K_4$ or is $(2,k)$-tight.
\end{thm}

We remark that for type $k=3$, it was shown in \cite{salww,BSWWconing} that Laman's theorem \cite{Lamanbib} applies to the sphere. It is an open problem to characterise generic isostatic frameworks on surfaces of type $k=0$. In particular, the natural analogue of Theorem~\ref{thm:surf3} is known to be false. The graph formed from $K_5$ by adding a degree $2$ vertex gives an example of a $(2,0)$-tight simple graph that is flexible on any such surface.
Due to this complication, we will consider symmetric analogues of Laman's theorem for surfaces of type $k>0$ only in this paper.

We finish this section by defining a \emph{stress} for a framework on $\M$. Stresses and stress matrices have been used to some effect in a variety of aspects of rigidity theory (e.g. \cite{C05,C&W92,W2}). Very recently the analogous properties of stresses for frameworks on surfaces have been developed \cite{JNstress}. We record the definition here as it will be useful for us in what follows. 

\begin{defn}\label{defn:stress}
A \emph{(self)-stress} for $(G,p)$ on $\M$ is a pair $(\omega,\lambda)$ such that $(\omega,\lambda) \in coker R_{\M}(G,p)$, that is, a vector $(\omega, \lambda)\in \mathbb{R}^{|E(G)|+|V(G)|}$ such that $(\omega,\lambda)^T R_\M(G,p)=0$.

Equivalently, $\omega$ is a stress if for all $1\leq i \leq |V(G)|$
\[ \sum_{\{i,j\}\in E(G)} \omega_{ij}(p_i-p_j) + \lambda_i N(p_i)=0. \]
\end{defn}


\section{Symmetric graphs} \label{sec:symgraph}

In this section we review some basic properties of symmetric graphs. In particular, we introduce the notion of a `gain graph' which is a useful tool to describe the underlying combinatorics of symmetric frameworks (see also \cite{jkt,Ross,schtan}, for example).

\subsection{Quotient gain graphs}

Given a group $S$, an \emph{$S$-gain graph} is a pair $(H,\psi)$, where $H$ is a directed multi-graph (which may contain $|S|-1$ loops at each vertex and up to $|S|$ multiple edges between any pair of vertices) and $\psi:E(H)\to S$ is a map which assigns an element of $S$ to each edge of $H$. The map $\psi$ is also called the \emph{gain function} of $(H,\psi)$ (see Figures~\ref{c2gaingraphs} (b) and (d) for examples of $\mathcal{C}_2$-gain graphs). A gain graph is a directed graph, but its orientation is only used as a reference orientation, and may be changed, provided that we also modify $\psi$ so that if an edge has gain $x$ in one orientation, then it has gain $x^{-1}$ in the other direction. Note that if $S$ is a group of order $2$, then the orientation is irrelevant. For simplicity, we omit the labels of edges with identity gain in the figures.

 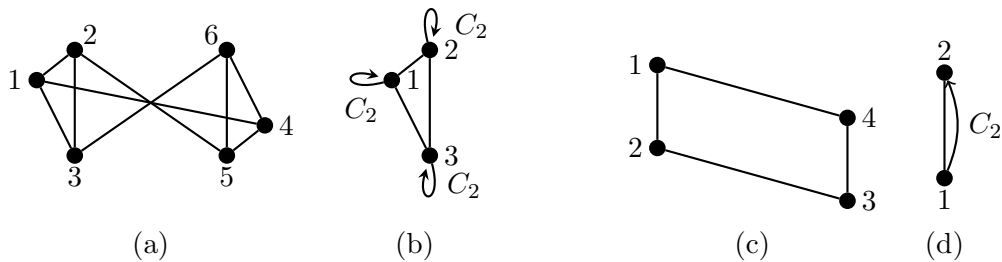
\begin{figure}[htp]
\begin{center}
\begin{tikzpicture}[very thick,scale=1]
       \tikzstyle{every node}=[circle, draw=black, fill=black, inner sep=0pt, minimum width=5pt];
   \path (0,-1.2) node (p3) [label = below: $3$] {} ;
    \path (2,-1.2) node (p5) [label = below: $5$] {} ;
    \path (2,0.2) node (p6) [label = above left: $6$] {} ;
   \path (0,0.2) node (p2) [label = above right: $2$] {} ;
   \path (-0.5,-0.2) node (p1) [label = left: $1$] {} ;
   \path (2.5,-0.8) node (p4) [label = right: $4$] {} ;
   \draw[thick] (p2) -- (p3);
     \draw[thick] (p1) -- (p2);
     \draw[thick] (p1) -- (p3);
     \draw[thick] (p1) -- (p4);
     \draw[thick] (p2) -- (p5);
     \draw[thick] (p6) -- (p3);
     \draw[thick] (p6) -- (p4);
     \draw[thick] (p6) -- (p5);
     \draw[thick] (p5) -- (p4);
       \node [rectangle, draw=white, fill=white] (b) at (0,-2.2) {$\quad$};
    \node [rectangle, draw=white, fill=white] (b) at (1,-2.4) {(a)};
        \end{tikzpicture}
          \hspace{0.3cm}
     \begin{tikzpicture}[very thick,scale=1]
\tikzstyle{every node}=[circle, draw=black, fill=black, inner sep=0pt, minimum width=5pt];
   \path (0,-1.2) node (p3) [label = right: $3$] {} ;
   \path (0,0.2) node (p2) [label = right: $2$] {} ;
   \path (-0.5,-0.2) node (p1) [label = right: $1$] {} ;
\path[thick]
(p1) edge (p3);
\path[thick]
(p2) edge (p3);
\path[thick]
(p2) edge (p1);
\path[thick]
(p1) edge [loop left,->, >=stealth,shorten >=2pt,looseness=26] (p1);
\path[thick]
(p2) edge [loop above,->, >=stealth,shorten >=2pt,looseness=26] (p2);
\path[thick]
(p3) edge [loop below,->, >=stealth,shorten >=2pt,looseness=26] (p3);
\node [rectangle, draw=white, fill=white] (b) at (-0.9,-0.6) {$C_2$};
\node [rectangle, draw=white, fill=white] (b) at (0.55,0.52) {$C_2$};
\node [rectangle, draw=white, fill=white] (b) at (0.45,-1.6) {$C_2$};
\node [rectangle, draw=white, fill=white] (b) at (-0.2,-2.4) {(b)};
\end{tikzpicture}
       \hspace{1.5cm}
    \begin{tikzpicture}[very thick,scale=1]
\tikzstyle{every node}=[circle, draw=black, fill=black, inner sep=0pt, minimum width=5pt];
        \path (0,0) node (p1) [label = left: $1$] {} ;
       \path (0,-1.1) node (p2) [label = left: $2$] {} ;
   \path (2.5,-1.8) node (p3) [label = right: $3$] {} ;
   \path (2.5,-0.7) node (p4) [label = right: $4$] {} ;
   \draw[thick] (p1) -- (p4);
      \draw[thick] (p3) -- (p4);
     \draw[thick] (p2) -- (p3);
      \draw[thick] (p2) -- (p1);
\node [rectangle, draw=white, fill=white] (a) at (1.25,-0.7) {};
\node [rectangle, draw=white, fill=white] (b) at (1.25,-2.4) {(c)};
         \end{tikzpicture}
               \hspace{0.3cm}
\begin{tikzpicture}[very thick,scale=1]
   \tikzstyle{every node}=[circle, draw=black, fill=black, inner sep=0pt, minimum width=5pt];
   \path (0,-1.5) node (p1) [label = below: $1$] {} ;
   \path (0,-0.1) node (p2) [label = above: $2$] {} ;
  \draw[thick] (p1) -- (p2);
\path[thick]
(p1) edge [->,bend right=22] (p2);
\node [rectangle, draw=white, fill=white] (b) at (0.54,-0.8) {$C_2$};
\node [rectangle, draw=white, fill=white] (b) at (0,-2.4) {(d)};
\end{tikzpicture}\end{center}
\vspace{-0.3cm}
\caption{$\mathcal{C}_2$-symmetric graphs ((a), (c)) and their quotient gain graphs ((b), (d)), where $\mathcal{C}_2$ denotes half-turn symmetry.}
\label{c2gaingraphs}
\end{figure}

Let $G$ be a finite simple graph. An \emph{automorphism} of $G$ is a permutation $\pi$ of the vertex set $V(G)$ of $G$ such that $\{i,j\}\in E(G)$ if and only if $\{\pi(i),\pi(j)\}\in E(G)$. The set of all automorphisms of $G$ forms a group, called the automorphism group $\textrm{Aut}(G)$ of $G$. An \emph{action} of a group $S$ on $G$ is a group homomorphism $\theta:S\to \textrm{Aut}(G)$. If $\theta(x)(i)\neq i$ for all $i\in V(G)$ and all non-trivial elements $x$ of the group $S$, then the action $\theta$ is called \emph{free}. If $S$ acts on $G$ by $\theta$, then we say that the graph $G$ is \emph{$S$-symmetric} (with respect to $\theta$). Throughout this paper, we only consider free actions, and we will omit to specify the action $\theta$ if it is clear from the context. In that case, we write $x v$ instead of $\theta(x)(v)$.

For an $S$-symmetric graph $G$, the \emph{quotient graph} $G/ S$ is the multi-graph which has the set $V(G)/ S$ of vertex orbits as its vertex set and the set $E(G)/ S$ of edge orbits as its edge set. Note that an edge orbit may be represented by a loop in $G/ S$.

While several distinct graphs may have the same quotient graph, a gain labeling makes the relation one-to-one, up to equivalence for the gain function (see Section~\ref{sec:balgaingr}), provided that the underlying action is free. To see this, choose an arbitrary representative vertex $i$ for each vertex orbit, so that each vertex orbit has the form $Si=\{x i| x\in S\}$. If the action is free, an edge orbit connecting $Si$ and $Sj$ can be written as $\{\{x i, x x' j\}| x\in S\}$ for a unique $x'$ in $S$. We then orient the edge orbit from $Si$ to $Sj$ in $G/ S$ and assign it the gain $x'$. This gives the \emph{quotient $S$-gain graph} $(G/ S, \psi)$. 

Conversely, let $(H,\psi)$ be an $S$-gain graph. For $x\in S$ and $i\in V(H)$, we denote the pair $(x,i)$ by $x i$. The \emph{covering graph} (or \emph{lifted graph}) of $(H,\psi)$ is the simple graph with vertex set $S\times V(H)=\{x i| x\in S, i\in V(H)\}$ and edge set $\{\{x i, x\psi(e) j\}| x\in S, e=(i,j)\in V(H)\}$. Clearly, $S$ acts freely on the covering graph with the action $\theta$ defined by $\theta(x): i\mapsto x i$ for $x\in S$ under which the quotient comes back to $(H,\psi)$. 

The map $c:G\to H$ defined by $c(x i)=i$ and $c(\{x i, x \psi(e) j\})=(i,j)$ is called a \emph{covering map}. The fiber $c^{-1}(i)$ of a vertex $i\in V(H)$ and the fiber $c^{-1}(e)$ of an edge $e\in E(H)$ coincide with a vertex orbit and edge orbit of $G$, respectively.


\subsection{Balanced gain graphs and the switching operation} \label{sec:balgaingr}

Let $(H,\psi)$ be an $S$-gain graph, and let $W=e_1, e_2, \ldots , e_k, e_1$ be a closed walk in $(H,\psi)$, where $e_i\in E(H)$ for all $i$. We define the \emph{gain} of $W$ as $\psi(W)=\psi(e_1)\cdot \psi(e_2)\cdot \cdots \psi(e_k)$ if each edge is oriented in the forward direction, and if an edge $e_i$ is directed backwards, then we replace $\psi(e_i)$ by $\psi(e_i)^{-1}$ in the product. (If $S$ is an additive group, then we replace the product by the sum.)  Note that for abelian groups, the gain of a closed walk is independent of the choice of the starting vertex. However, this is of course not the case for non-abelian groups. 

  For $v\in(H,\psi)$, we denote by $\mathcal{W}(H,v)$ the set of closed walks starting at $v$. Similarly, if $F\subseteq E(H)$ and $v\in V(H)$, then $\mathcal{W}(F,v)$ denotes the set of closed walks starting at $v$ and using only edges of $F$, where $\mathcal{W}(F,v)=\emptyset$ if $v$ is not incident to an edge of $F$.

For $F\subseteq E(H)$, the subgroup induced by $F$ relative to $v$ is defined as $\langle F\rangle_{\psi,v}=\{\psi(W)|\, W\in \mathcal{W}(F,v)\}$. We will sometimes omit the subscript $\psi$ of  $\langle F\rangle_{\psi,v}$ if it is clear from the context.

For any connected $F\subseteq E(H)$, we say that $F$ is \emph{balanced} if $\langle F\rangle_v=\{\textrm{id}\}$ for some $v\in V(F)$, and \emph{unbalanced} otherwise. By \cite[Proposition 2.1]{jkt}, this property is invariant under the choice of the base vertex $v\in V(F)$, and $F$ is unbalanced if and only if $F$ contains an unbalanced cycle. So we may extend this notion to any (possibly disconnected) $F\subseteq E(H)$, and call $F$ \emph{unbalanced} if $F$ contains an unbalanced cycle.


If $S$ is of order $2$, then we will think of $S$ as the group $\mathbb{Z}_2=\{0,1\}$ with addition as the group operation. So a subgraph of a $\mathbb{Z}_2$-gain graph $(H,\psi)$ will be unbalanced if and only if it contains a cycle with gain $1$. Note that a $\mathbb{Z}_2$-gain graph is commonly known as a \emph{signed graph} in the literature \cite{zas1,Zas}.
Let $v$ be a vertex in an $S$-gain graph $(H,\psi)$. To \emph{switch} $v$ with $x\in S$ means to change the gain function  $\psi$ on $E(H)$ as follows:
$$\psi'(e)=\left\{\begin{array}{ll} x\cdot \psi(e)\cdot x^{-1} & \textrm{ if $e$ is a loop incident with $v$;}\\x\cdot \psi(e) & \textrm{ if $e$ is a non-loop incident from $v$;}\\ \psi(e) \cdot x^{-1}& \textrm{ if $e$ is a non-loop incident to $v$;}\\\psi(e) & \textrm{ otherwise.} \end{array}\right.$$
In particular, if we switch a vertex $v$ in a  $\mathbb{Z}_2$-gain graph $(H,\psi)$ with $0$, then the gain function $\psi$ remains unchanged, and if we switch $v$ with $1$, then the gain of every non-loop edge that is incident with $v$ changes its gain from $0$ to $1$ or vice versa, and the gains of all other edges remain the same.

We say that a gain function $\psi'$  is \emph{equivalent} to another gain function $\psi$ on the same edge set if $\psi'$ can be obtained from $\psi$ by a sequence of switching operations.

In the following, we summarise some key properties of the switching operation. Detailed proofs of these results for an arbitrary discrete symmetry group $S$ can be found in \cite{jkt}. For the special case of signed graphs, these theorems were first proved by Zaslavsky in the 1980s \cite{Zas}.

\begin{prop}\label{prop:lem1}(\cite[Prop. 2.2]{jkt})  Switching a vertex of an $S$-gain graph $(H,\psi)$ does not alter the  balance of $(H,\psi)$.
\end{prop}

\begin{prop} \label{prop:lem2} (\cite[Prop. 2.3 and Lemma 2.4]{jkt}) An $S$-gain graph  $(H,\psi)$ is balanced if and only if the vertices in $V(H)$ can be switched so that every edge in the resulting $S$-gain graph $(H,\psi')$ has the identity element of $S$ as its gain.
\end{prop}

\begin{lem}\label{lem:balance}(\cite[Lemma 2.5]{jkt}) Let $(G,\psi)$ be an $S$-gain graph, and let $U\subseteq V(G)$ and $W\subseteq V(G)$ be subsets of $V(G)$. Further, let $H$ be the signed subgraph of $(G,\psi)$ induced by $U$, and let $K$ be the signed subgraph of $(G,\psi)$ induced by $W$, and suppose that $H$, $K$ and $H\cap K$ is connected. If $H$ and $K$ are balanced, then $H\cup K$ is also balanced.
\end{lem}


\section{Symmetric frameworks on surfaces} \label{sec:symfw}

Let $\M\subseteq \mathbb{R}^3$ be a surface, let $G$ be a finite simple graph, and let $p:V(G)\to \M$. A \emph{symmetry operation} of the framework $(G,p)$ on $\M$ is an isometry $x$ of $\mathbb{R}^3$ which maps $\M$ onto itself (i.e., $x$ is a symmetry of $\M$) such that for some $\alpha_x\in \textrm{Aut}(G)$, we have
\begin{equation} \label{eq:symop} x(p_i)=p_{\alpha_x(i)}\quad \textrm{for all } i\in V(G)\textrm{. }\nonumber\end{equation}  The set of all symmetry operations of a framework $(G,p)$ on $\M$ forms a group under composition, called the \emph{point group} of $(G,p)$. Clearly, we may assume w.l.o.g. that the point group of a framework is always a \emph{symmetry group}, i.e., a subgroup of the orthogonal group $O(\mathbb{R}^{d})$.

We use the Schoenflies notation for the symmetry operations and symmetry
groups considered in this paper, as this is one of the standard notations in
the literature about symmetric structures (see \cite{bishop,FGsymmax,KG1,KG2,BS6,schtan}, for example). In particular, $\mathcal{C}_s$ is a group of order $2$ generated by a single reflection $s$, and $\mathcal{C}_m$, $m\geq 1$, is a cyclic group generated by a rotation $C_m$ about an axis through the origin by an angle of $\frac{2\pi}{m}$. (See also Section~\ref{sec:combchar}.)

Given a surface $\M$, a symmetry group $S$ and a graph $G$, we let $\mathscr{R}^{\M}_{(G,S)}$ denote the set of all realisations of $G$ on $\M$ whose point group is either equal to $S$ or contains $S$ as a subgroup \cite{BS2,BS1}. In other words, the set $\mathscr{R}^{\M}_{(G,S)}$ consists of all realisations $(G,p)$ of $G$ in $\mathbb{R}^{3}$ which are supported on $\M$ and for which there exists an action $\theta:S\to \textrm{Aut}(G)$ so that
\begin{equation}\label{class} x\big(p(i)\big)=p(\theta(x)(i))\textrm{ for all } i\in V(G)\textrm{ and all } x\in S\textrm{.}\end{equation}
A framework $(G,p)\in \mathscr{R}^{\M}_{(G,S)}$ satisfying the equations in (\ref{class}) for the map $\theta:S\to \textrm{Aut}(G)$ is said to be \emph{of type $\theta$}, and the set of all realisations in $\mathscr{R}^{\M}_{(G,S)}$ which are of type $\theta$ is denoted by $\mathscr{R}^{\M}_{(G,S,\theta)}$ (see again \cite{BS2,BS1} and Figure~\ref{K22types}). It is shown in \cite{BS1} that if $p$ is injective, then
$(G,p)$ is of a unique type $\theta$ and $\theta$ is necessarily also a homomorphism.

\begin{figure}[htp]
\begin{center}
           \begin{tikzpicture}[very thick,scale=1]
\tikzstyle{every node}=[circle, draw=black, fill=black, inner sep=0pt, minimum width=5pt];
    \path (-0.7,0.8) node (p1) [label = above left: $p_1$] {} ;
    \path (0.7,0.8) node (p4) [label = above right: $p_4$] {} ;
    \path (-1.6,-0.5) node (p2) [label = below left: $p_2$] {} ;
     \path (1.6,-0.5) node (p3) [label = below right: $p_3$] {} ;
      \draw[thick] (p1) -- (p4);
    \draw[thick] (p1) -- (p2);
    \draw[thick] (p3) -- (p4);
    \draw[thick] (p2) -- (p3);
     \draw [dashed, thin] (0,-1) -- (0,1.3);
     \node [draw=white, fill=white] (z) at (-0.2,-1) {$s$};
      \node [draw=white, fill=white] (b) at (0,-1.6) {(a)};
        \end{tikzpicture}
        \hspace{2cm}
      \begin{tikzpicture}[very thick,scale=1]
\tikzstyle{every node}=[circle, draw=black, fill=black, inner sep=0pt, minimum width=5pt];
       \tikzstyle{every node}=[circle, draw=black, fill=black, inner sep=0pt, minimum width=5pt];
    \path (-0.7,0.8) node (p1) [label = above left: $p_1$] {} ;
    \path (0.7,0.8) node (p4) [label = above right: $p_3$] {} ;
    \path (-1.6,-0.5) node (p2) [label = below left: $p_2$] {} ;
     \path (1.6,-0.5) node (p3) [label = below right: $p_4$] {} ;
      \draw[thick] (p1) -- (p3);
    \draw[thick] (p1) -- (p2);
    \draw[thick] (p3) -- (p4);
    \draw[thick] (p2) -- (p4);
     \draw [dashed, thin] (0,-1) -- (0,1.3);
     \node [draw=white, fill=white] (z) at (-0.2,-1) {$s$};
      \node [draw=white, fill=white] (b) at (0,-1.6) {(b)};
        \end{tikzpicture}
\end{center}
\vspace{-0.3cm}
\caption{Realisations of the cycle graph $C_{4}$ in $\mathscr{R}^{\M}_{(C_4,\mathcal{C}_s)}$ of different types, where $\M$ is the Euclidean plane and $\mathcal{C}_s=\{id,s\}$ is the reflection group. The framework in (a) is of type
$\theta_{a}$, where $\theta_{a}: \mathcal{C}_{s} \to \textrm{Aut}(C_4)$ is the homomorphism defined by $\theta_{a}(s)=
(1 \, 4)(2 \, 3)$, and the framework in (b) is of type
$\theta_{b}$, where $\theta_{b}: \mathcal{C}_{s} \to \textrm{Aut}(C_4)$ is the homomorphism defined by $\theta_{b}(s)=
(1 \, 3)(2\, 4)$.}
\label{K22types}
\end{figure}
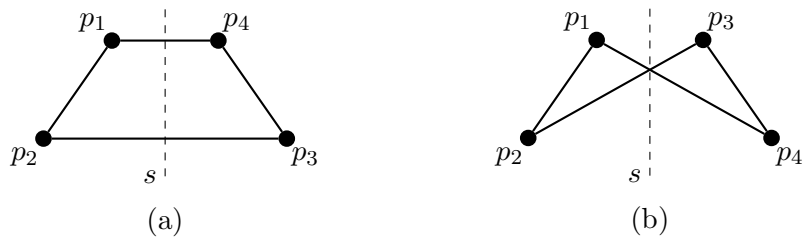

Let $S$ be an abstract group, and   $G$ be a  $S$-symmetric graph with respect to
a free action $\theta:S\rightarrow \textrm{Aut}(G)$.
Suppose also that $S$ acts on $\mathbb{R}^d$
via a homomorphism $\tau:S\rightarrow O(\mathbb{R}^d)$.
Then we say that a framework $(G,p)$ on a surface $\M$ is {\em $S$-symmetric} (with respect to $\theta$ and $\tau$) if
$(G,p)\in \mathscr{R}^{\M}_{(G,\tau(S),\theta)}$, that is, if 
\begin{equation*} 
\tau(x) (p(i))=p(\theta(x) i) \qquad \text{for all } x\in S \text{ and all } i\in V(G).
\end{equation*}

\begin{figure}
\begin{center}
\begin{tikzpicture}
\draw[black,thick] (0,4.1) ellipse (1.5 and .75);
\draw[black,thick] (0,0.1) ellipse (1.5 and .75);
 \draw[black,thick]
  (-1.5,0.1) -- (-1.5,4.1);
\draw[black,thick]
(1.5,0.1) -- (1.5,4.1);

\draw[black,thick] (6,4.1) ellipse (1.5 and 0.75);
\draw[black,thick] (6,0.1) ellipse (1.5 and 0.75);
 \draw[black,thick]
  (4.5,0.1) -- (4.5,4.1);
\draw[black,thick]
(7.5,0.1) -- (7.5,4.1);

\draw[dashed] (6,0) -- (6,4.1);

\draw[dashed] (-1.8,2.4) -- (-1.1,1.5) -- (1.6,1.5) -- (.9,2.4) -- (-1.8,2.4);

\filldraw[gray] (1,.3) circle (3pt);
\filldraw[gray] (-.5,.3) circle (3pt);
\filldraw (.5,1) circle (3pt);
\filldraw (-1,1) circle (3pt);
\filldraw[gray] (1,3.1) circle (3pt);
\filldraw[gray] (-.5,3.1) circle (3pt);
\filldraw (-1,2.5) circle (3pt);
\filldraw (.5,2.5) circle (3pt);

\filldraw (5,.8) circle (3pt);
\filldraw (7,.8) circle (3pt);
\filldraw[gray] (6.2,1.2) circle (3pt);
\filldraw (5.8,.4) circle (3pt);
\filldraw (5,2.8) circle (3pt);
\filldraw (7,2.8) circle (3pt);
\filldraw[gray] (6.2,3.2) circle (3pt);
\filldraw (5.8,2.4) circle (3pt);

\draw[black,thick]
(1,.3) -- (-.5,.3) -- (-1,1) -- (.5,1) -- (1,.3) -- (1,3.1);
\draw[black,thick]
(1,3.1) -- (-.5,3.1) -- (-1,2.5) -- (.5,2.5) -- (1,3.1);
\draw[black,thick]
(-.5,.3) -- (-.5,3.1);
\draw[black,thick]
(-1,1) -- (-1,2.5);
\draw[black,thick]
(.5,1) -- (.5,2.5);

\draw[black,thick]
(5,.8) -- (6.2,1.2) -- (7,.8) -- (5.8,.4) -- (5,.8) -- (5,2.8);
\draw[black,thick]
(6.2,3.2) -- (7,2.8) -- (5.8,2.4) -- (5,2.8) -- (6.2,3.2) -- (6.2,1.2);
\draw[black,thick]
(7,.8) -- (7,2.8);
\draw[black,thick]
(5.8,.4) -- (5.8,2.4);

\node [draw=white, fill=white] (b) at (0,-1.2) {(a)};
\node [draw=white, fill=white] (b) at (6,-1.2) {(b)};
\end{tikzpicture}
\end{center}
\vspace{-0.3cm}
\caption{Realisations of the cube graph $Q_3$ in $\mathscr{R}^{\M}_{(Q_3,S)}$ where $\M$ is a cylinder and $S$ is (a) $\C_s$ with mirror orthogonal to the axis of the cylinder and (b) $\C_4$ with the $4$-fold rotation around the axis of the cylinder. The grey joints are at the `back' of the cylinder.}
\label{fig:cylinderexamples}
\end{figure}
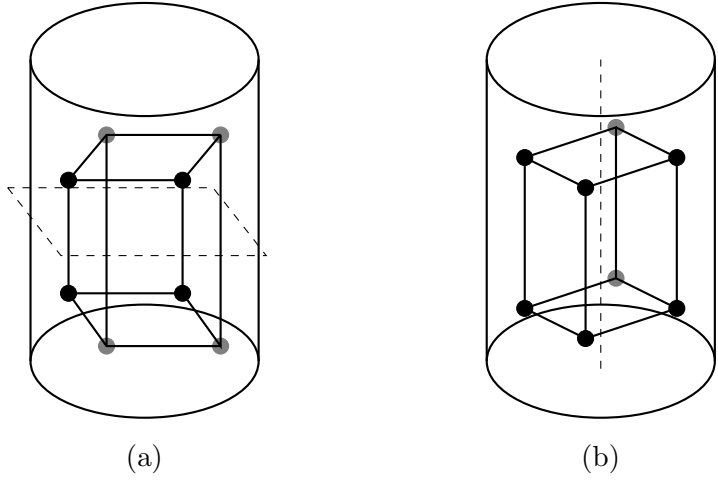

For simplicity, we will assume throughout this paper that a framework $(G,p)\in \mathscr{R}^{\M}_{(G,S,\theta)}$ has no joint that is `fixed' by a non-trivial symmetry operation in $S$ (i.e., $(G,p)$ has no joint $p_i$ with $x(p_i)=p_i$ for some $x\in S$, $x\neq id$).
In particular, this will simplify the construction of the orbit-surface rigidity matrix in the next section, since in this case this matrix has a set of $3$ columns for each orbit of vertices under the action $\theta$.

Let $\bQ_S$ denote the field extension of $\bQ$ formed by adjoining the entries of all the matrices in $S$ to $\bQ$. 
We say that a framework $(G,p)$ in $\mathscr{R}^{\M}_{(G,S,\theta)}$ with quotient $S$-gain graph $(G_0,\psi)$ is \emph{$S$-generic} if $\td[\bQ_S(p):\bQ_S]=2|V(G_0)|.$
This implies that the only polynomial equations in $3|V(G)|$ variables that evaluate to zero at $p$ are those that define $S$ or $\M$.
This is the natural extension of the definitions of generic seen in the literature \cite{JMN,jkt,nop1}.


\subsection{Symmetry-forced rigidity and the orbit-surface rigidity matrix}

Given an $S$-symmetric framework $(G,p)$ on a surface $\M$, we are interested in non-trivial motions of $(G,p)$ on $\M$ which preserve the symmetry group $S$ of $(G,p)$ throughout the path.
Infinitesimal motions corresponding to such symmetry-preserving continuous motions are `$S$-symmetri infinitesimal motions' (see also \cite{jkt,schtan,BSWWorbit}):

An infinitesimal motion $u$ of a framework $(G,p)$ in $\mathscr{R}^{\M}_{(G,S,\theta)}$  is \emph{$S$-symmetric} if \begin{equation*}\label{fulsymmot} x\big(u_i\big)=u_{\theta(x)(i)}\textrm{ for all } i\in V(G)\textrm{ and all } x\in S\textrm{,}\end{equation*} i.e., if $u$ is unchanged under all symmetry operations in $S$. Note that all the velocity
vectors $u_i$, considered as acting at the framework joints, are of course tangential to
the surface $\M$.

We say that $(G,p)\in \mathscr{R}^{\M}_{(G,S,\theta)}$ is \emph{$S$-symmetric infinitesimally rigid} if every $S$-symmetric infinitesimal motion is trivial. Note that the dimension of the space of trivial $S$-symmetric infinitesimal motions, denoted by $k_S$, can easily be read off from the character table for $S$ (see \cite{bishop}, for example).

Recall that the type $k$ of a surface $\M$ is the dimension of the group of isometries of $\bR^3$ acting tangentially to $\M$. Analogously, we call the dimension of the space of trivial $S$-symmetric infinitesimal motions, $k_S$, the \emph{symmetric type} of $\M$.


A self-stress $(\omega,\lambda)\in \mathbb{R}^{|E(G)|+|V(G)|}$ of $(G,p)$ is \emph{$S$-symmetric} if $\omega_e=\omega_f$
whenever $e$ and $f$ belong to the same edge orbit $S e=\{x e|\,x\in S\}$ of $G$, and $\lambda_i=\lambda_j$
whenever $i$ and $j$ belong to the same vertex orbit $S i=\{x i|\,x\in S\}$ of $G$.

\begin{figure}[htp]
\begin{center}
\begin{tikzpicture}[rotate=90, very thick,scale=1]
\tikzstyle{every node}=[circle, draw=black, fill=black, inner sep=0pt, minimum width=5pt];
    \path (0.1,1.2) node (p1) [label = left: $p_{4}$] {} ;
    \path (2.1,0) node (p4) [label = above right: $p_{1}$]{} ;
    \path (2.3,-1.3) node (p3) [label = right: $p_{2}$] {} ;
     \path (0.3,0) node (p2) [label = below: $p_{3}$] {} ;
        \draw[thick] (p1) -- (p4);
      \draw[thick] (p3) -- (p4);
     \draw[thick] (p2) -- (p3);
      \draw[thick] (p2) -- (p1);
            \draw [ultra thick, ->, black!40!white](p1) -- (0.52,0.74);
      \draw [ultra thick, ->, black!40!white](p3) -- (1.88,-0.84);
      \draw [ultra thick, ->, black!40!white](p2) -- (-0.5,-0.65);
      \draw [ultra thick, ->, black!40!white](p4) -- (2.9,0.55);
      \filldraw[fill=black, draw=black]
    (1.2,-0.05) circle (0.004cm);
     \node [draw=white, fill=white] (b) at (-1,-0) {(a)};
              \end{tikzpicture}
 \hspace{1cm}
            \begin{tikzpicture}[very thick,scale=1]
\tikzstyle{every node}=[circle, draw=black, fill=black, inner sep=0pt, minimum width=5pt];
    \path (-0.7,0.8) node (p1) [label = left: $p_{1}$] {} ;
    \path (0.7,0.8) node (p4) [label = right: $p_{4}$]{} ;
    \path (-1.6,-0.8) node (p2) [label = below: $p_{2}$] {} ;
     \path (1.6,-0.8) node (p3) [label = below: $p_{3}$] {} ;
      \draw[thick] (p1) -- (p4);
    \draw[thick] (p1) -- (p2);
    \draw[thick] (p3) -- (p4);
    \draw[thick] (p2) -- (p3);
     \draw [dashed, thin] (0,-1.6) -- (0,1.6);
     \draw [ultra thick, ->, black!40!white] (p1) -- (-0.7,1.3);
      \draw [ultra thick, ->, black!40!white] (p4) -- (0.7,1.3);
      \draw [ultra thick, ->, black!40!white] (p2) -- (-1.6,-0.3);
      \draw [ultra thick, ->, black!40!white] (p3) -- (1.6,-0.3);
      \node [draw=white, fill=white] (b) at (0,-2.1) {(b)};
        \end{tikzpicture}
                \hspace{1cm}
                \begin{tikzpicture}[very thick,scale=1]
\tikzstyle{every node}=[circle, draw=black, fill=black, inner sep=0pt, minimum width=5pt];
    \path (-0.7,0.8) node (p1) [label = left: $p_{1}$]  {} ;
    \path (0.7,0.8) node (p4)[label = right: $p_{4}$] {} ;
    \path (-1.6,-0.8) node [label = below: $p_{2}$](p2)  {} ;
     \path (1.6,-0.8) node [label = below: $p_{3}$](p3)  {} ;
      \draw[thick] (p1) -- (p4);
    \draw[thick] (p1) -- (p2);
    \draw[thick] (p3) -- (p4);
    \draw[thick] (p2) -- (p3);
     \draw [dashed, thin] (0,-1.6) -- (0,1.6);
     \draw [ultra thick, ->, black!40!white] (p1) -- (-0.5,0.3);
      \draw [ultra thick, ->, black!40!white] (p4) -- (0.9,1.3);
      \draw [ultra thick, ->, black!40!white] (p2) -- (-2.2,-1);
      \draw [ultra thick, ->, black!40!white] (p3) -- (1,-0.6);
      \node [draw=white, fill=white] (b) at (0,-2.1) {(c)};
        \end{tikzpicture}
       \end{center}
\vspace{-0.3cm}
\caption{Infinitesimal motions of frameworks in the Euclidean plane: (a) a $\mathcal{C}_2$-symmetric non-trivial infinitesimal motion; (b) a $\mathcal{C}_s$-symmetric trivial infinitesimal motion; (c) a non-trivial infinitesimal motion which is not $\mathcal{C}_s$-symmetric.}
\label{fulsym}
\end{figure}
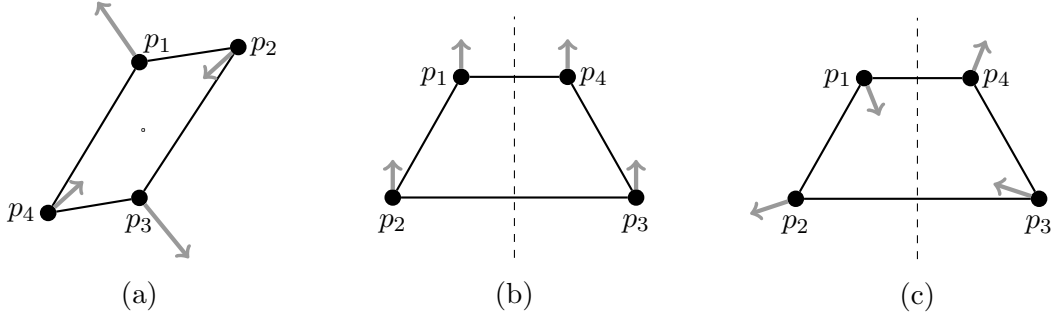

In Euclidean space, a key tool to study symmetric infinitesimal motions is the orbit rigidity matrix. This matrix is defined as follows (see also \cite{BSWWorbit}):

\begin{defn}\label{orbitmatrixdef}
Let $(G,p)$ be an $S$-symmetric framework (with respect to $\theta$ and $\tau$) in Euclidean $3$-space which has no joint that is `fixed' by a non-trivial symmetry operation in $S$. Further, let $(G_0,\psi)$ be the quotient $S$-gain graph of $(G,p)$. For each edge $e\in E(G_0)$, the \emph{orbit rigidity matrix} $O(G,p, S)$ of $(G,p)$ has the following corresponding ($3|V(G_0)|$-dimensional) row vector:
\begin{description}
\item[Case 1] Suppose $e=(i,j)$, where $i\neq j$. 
Then the corresponding row  in $O(G,p,S)$ is:
        \begin{displaymath}\renewcommand{\arraystretch}{0.8}
        \bordermatrix{  & &  i &  &  j &  \cr & 0 \ldots 0 & \big(p_i-\tau(\psi(e))(p_j)\big) & 0  \ldots  0 & \big(p_j-\tau(\psi(e))^{-1}(p_i)\big) & 0  \ldots  0}\textrm{.}
    \end{displaymath}
\item[Case 2] Suppose $e=(i,i)$ is a loop in $(G_0,\psi)$. Then $\psi(e)\neq id$ and  
the corresponding row  in $O(G,p,S)$ is:
\begin{displaymath}
\bordermatrix{ & &  i & \cr & 0  \ldots  0 & \big(2p_i-\tau(\psi(e))(p_i)-\tau(\psi(e))^{-1}(p_i)\big) & 0  \ldots  0}\textrm{.}
    \end{displaymath}
\end{description}
\end{defn}

Using the above definition of the orbit rigidity matrix for frameworks in Euclidean 3-space, we can easily set up the orbit-surface rigidity matrix as follows:

\begin{defn}
Let $(G,p)$ be a framework in $\mathscr{R}^{\M}_{(G,S,\theta)}$ with quotient $S$-gain graph $(G_0,\psi)$. The \emph{orbit-surface rigidity matrix} $O_{\M}(G,p,S)$ of  $(G,p)$ is the $(|E(G_0)|+|V(G_0)|)\times 3|V(G_0)|$ block matrix
\[\begin{bmatrix} O(G,p,S)\\ \N_0(p_0) \end{bmatrix} \]
where $O(G,p,S)$ is the standard orbit rigidity matrix for the  framework  and symmetry group considered in $\bR^3$ (see Definition \ref{orbitmatrixdef}) and $\N_0(p_0)$ represents the block-diagonalised matrix of surface normals to the framework joints corresponding to the vertices of $G_0$.
\end{defn}

A framework $(G,p)\in \mathscr{R}^{\M}_{(G,S,\theta)}$ is \emph{$S$-regular} if $O_\M(G,p,S)$ has maximal rank among all realisations in $\mathscr{R}^{\M}_{(G,S,\theta)}$. 
Note that if a framework on $\M$ is $S$-generic, then it is clearly also $S$-regular, and if some $S$-symmetric realisation of a graph $G$ is $S$-symmetric infinitesimally rigid, then the same is true for every $S$-regular realisation of $G$.

An $S$-gain graph $(G_0,\psi)$ is \emph{$S$-independent} if $O_{\M}(G,p,S)$ has linearly independent rows and \emph{$S$-dependent} otherwise. Clearly, if $(G,p)$ is \emph{$S$-isostatic} (i.e., minimally $S$-symmetric infinitesimally rigid) then $(G_0,\psi)$ is $S$-independent. The following lemma is an easy exercise.

\begin{lem}\label{lem:nullspace}
Let $\M$ be a surface with symmetric type $k_S$ with respect to a symmetry group $S$.
Let $N_\M(G,p,S)$ be the nullspace of $O_{\M}(G,p,S)$. Then $\dim N_\M(G,p,S)\geq k_S$.
\end{lem}

The following result summarises the key properties of the orbit-surface rigidity matrix:

\begin{thm} \label{orbitmatrixthm}
Let $(G,p)$ be a framework in $\mathscr{R}^{\M}_{(G,S,\theta)}$. Then  the solutions to $O_{\M}(G,p,S)u = 0$ are isomorphic to the space of $S$-symmetric infinitesimal motions of $(G,p)$. Moreover, the solutions to $(\omega,\lambda)^TO_{\M}(G,p,S) = 0$ are isomorphic to the space of  $S$-symmetric self-stresses of  $(G,p)$.
\end{thm}

\begin{proof}
This follows immediately from the corresponding result for the orbit rigidity matrix in Euclidean $3$-space \cite[Theorem 6.1 and Theorem 8.3]{BSWWorbit}.

For example, it was shown in \cite[Theorem 6.1]{BSWWorbit} that a vector $u$ lies in the kernel of the orbit rigidity matrix $O(G,p,S)$ if and only if $u$ is the restriction of an $S$-symmetric infinitesimal motion of $(G,p)$ to the joints corresponding to the vertices of the quotient $S$-gain graph of $G$. Therefore, a vector $u$ lies in the kernel of the orbit-surface rigidity matrix $O_{\M}(G,p,S)$ if and only if $u$ is in the kernel of $O(G,p,S)$ and the velocity vectors of $u$ are tangential to the surface $\M$. Moreover this holds if and only if $u$ is the restriction of an $S$-symmetric infinitesimal motion of $(G,p)$ on $\M$ to the joints corresponding to the vertices of the quotient $S$-gain graph of $G$.

Similarly, the proof of \cite[Theorem 8.3]{BSWWorbit} can easily be adapted to show that the solutions to $(\omega,\lambda)^TO_{\M}(G,p,S) = 0$ are isomorphic to the space of  $S$-symmetric self-stresses of  $(G,p)$.
\end{proof}

An $S$-symmetric framework $(G,p)$ is \emph{$S$-symmetric rigid} if every $S$-symmetric continuous motion is a rigid motion of $\M$.

It is an easy extension of \cite{asiroth} (see also \cite{nop1,BS6}) to show that for $S$-regular realisations, a symmetric infinitesimal motion implies a continuous symmetry preserving motion. Thus we have the following.

\begin{thm}
Let $(G,p)$ be an $S$-regular framework in $\mathscr{R}^{\M}_{(G,S,\theta)}$. Then $(G,p)$ is $S$-symmetric infinitesimally rigid on $\M$ if and only if $(G,p)$ is $S$-symmetric rigid on $\M$.
\end{thm}


\section{Necessary conditions for symmetry-forced rigidity}
\label{sec:necesscond}

We now establish analogues of Maxwell's theorem, showing combinatorial counts that must be satisfied by any $S$-regular $S$-isostatic  framework in $\mathscr{R}^{\M}_{(G,S,\theta)}$. We need the following definition (see also \cite{jkt,schtan}).

\begin{defn} \label{def:gainsparsity}
Let $(H,\psi)$ be an $S$-gain graph and let $k, \ell, m$ be nonnegative integers with $m\leq \ell$. Then
$(H,\psi)$ is called \emph{$(k,\ell,m)$-gain-sparse} if
\begin{itemize}
\item $|F|\leq k|V(F)|-\ell$ for any nonempty balanced $F\subseteq E(H)$;
\item $|F|\leq k|V(F)|-m$ for any nonempty $F\subseteq E(H)$.
\end{itemize}
A $(k,\ell,m)$-gain-sparse graph $(H,\psi)$ satisfying $|F|= k|V(F)|-m$ is called \emph{$(k,\ell,m)$-gain-tight}. 
Similarly, an edge set $E$ is called $(k,\ell,m)$-gain-sparse ($(k,\ell,m)$-gain-tight) if it induces a $(k,\ell,m)$-gain-sparse ($(k,\ell,m)$-gain-tight) graph.
\end{defn}

We first establish a necessary condition for a framework in $\mathscr{R}^{\M}_{(G,S,\theta)}$ to be $S$-isostatic.

\begin{thm}\label{thm:maxwell}
Let $\M$ be an irreducible algebraic surface of type $k$.
Let $S$ be a symmetry group of $\bR^3$ acting on $\M$ such that under $S$, $\M$ has type $k_S$. Let $(G,p)$ be an $S$-isostatic   framework in $\mathscr{R}^{\M}_{(G,S,\theta)}$ with quotient $S$-gain graph $(G_0,\psi)$. Then $(G_0,\psi)$ is $(2,k,k_S)$-gain-tight.
\end{thm}

\begin{proof}
First observe that if there exists an unbalanced subgraph $(H,\psi)$ of $(G_0,\psi)$ with $|E(H)|>2|V(H)|-k_S$, then Lemma \ref{lem:nullspace} implies that there is a row dependence in the orbit-surface matrix for that subgraph. Similarly, we must have $|E(G_0)|=2|V(G_0)|-k_S$.  So it remains to check that if $F\subseteq E(G_0)$ is balanced, then $|F|\leq 2|V(F)|-k$.

\begin{claim}\label{claim:switch}
Switching a vertex does not change the rank of the orbit-surface matrix.
\end{claim}

\begin{proof}
Let $v$ be the vertex we will switch with gain $\alpha$,
let $e=(u,v)$ be an edge, oriented into $v$, with gain $\psi(e)$ and let $f=(v,v)$ be a loop on $v$ with gain $\psi(f)$. Define $p_i'=p_i$ for all $v\in V(G_0)- v$ and $p'_{v}=\tau(\alpha) p_{v}$, and let $\psi'$ be the gain function obtained from $\psi$ by switching $v$ with $\alpha$. 

To prove the claim it suffices to show that the submatrix corresponding to the rows for $e,f$ and the row for $v$ has the same rank in $O_\M(G,p',S)$ and in $O_\M(G,p,S)$.

In $O_\M(G,p',S)$ we have

\setcounter{MaxMatrixCols}{20}
\[\begin{bmatrix}0&\dots&0 & p'_u-\tau(\psi'(e))p'_{v} & 0 & \dots & 0 & p'_{v}-\tau(\psi'(e))^{-1}p'_u & 0 & \dots & 0\\
0&\dots&&&&&0 & 2p'_{v}-\tau(\psi'(f))p'_{v}-\tau(\psi'(f))^{-1}p'_{v} & 0 & \dots & 0
\\ 0 & \dots & &&&&0& N_0(p'_v) & 0 & \dots & 0\end{bmatrix}.\]
Since $\psi'(e)=\psi(e)\alpha^{-1}$ and $\psi'(f)=\alpha\psi(f)\alpha^{-1}$ this is 
{\tiny
\setcounter{MaxMatrixCols}{20}
\[\begin{bmatrix}0&\dots&0 & p_u-\tau(\psi(e)\alpha^{-1}) \tau(\alpha) p_{v} & 0 & \dots & 0 & \tau(\alpha) p_{v}-\tau(\psi(e)\alpha^{-1})^{-1}p_u & 0 & \dots & 0\\ 
0&\dots&&&&&0 & 2\tau(\alpha) p_{v}-\tau(\alpha\psi(f)\alpha^{-1})\tau(\alpha) p_{v}-\tau(\alpha\psi(f)\alpha^{-1})^{-1}\tau(\alpha) p_{v} & 0 & \dots & 0\\
0 & \dots & &&&&0& N_0(\tau(\alpha)(p_v)) & 0 & \dots & 0\end{bmatrix},\]
}
which  simplifies (after scaling the third row if necessary) to 
\setcounter{MaxMatrixCols}{20}
\[\begin{bmatrix}0&\dots&0 & p_u-\tau(\psi(e))p_{v} & 0 & \dots & 0 & \tau(\alpha)(p_{v}-\tau(\psi(e))^{-1}p_u) & 0 & \dots & 0\\ 
0&\dots&&&&&0 & \tau(\alpha)(2p_{v}-\tau(\psi(f))p_{v}-\tau(\psi(f))^{-1}p_{v}) & 0 & \dots & 0\\
0 & \dots & &&&&0& \tau(\alpha)(N_0(p_v)) & 0 & \dots & 0\end{bmatrix},\]
since $N_0(\tau(\alpha)(p_v))$ is equal to a scalar multiple of $\tau(\alpha)(N_0(p_v)) $. To see this, recall that
$N_0(p_v)$ is a normal vector to $\M$ at the point $p_v$. Since $\tau(\alpha)$ is an isometry of $\mathbb{R}^3$, and $\tau(\alpha)$ maps $\M$ onto itself,  it maps the tangent plane to $\M$ at $p_v$ to the tangent plane to $\tau(\alpha)(\M)=\M$ at $\tau(\alpha)(p_v)$. 
   Thus, $\tau(\alpha)(N_0(p_v))$ is again a normal vector to $\M$ (with a possibly different magnitude), namely at the point $\tau(\alpha)(p_v)$.

By applying column operations to the triple for $p_v$ we turn this matrix into the submatrix of $O_\M(G,p,S)$ and hence the ranks are indeed the same. 
%
%
%
\end{proof}

Now suppose there exists a balanced edge set $F\subseteq E(G_0)$ with $|F|> 2|V(F)|-k$. Then, by the above Claim, we may switch the vertices of the  graph induced by $F$ so that every edge gain in this subgraph is the identity element of $S$. The submatrix of $O_{\M}(G,p,S)$ consisting of all those rows which correspond to the edges and vertices of the subgraph induced by $F$ is a standard surface rigidity matrix. Since $|F|> 2|V(F)|-k$, it follows from Lemma~\ref{lem:surf2} that this matrix has a row dependence, a contradiction.
\end{proof}

We remark that  the sparsity condition for unbalanced subgraphs is simpler than the reader may have anticipated. 
For most non-cyclic groups we can derive stronger necessary conditions by taking greater care to deal with the different possible
subgroups   $\langle F\rangle_{\psi,v}$ induced by $F$  (recall Section~\ref{sec:balgaingr}). 


For example, if $\M$ is the unit sphere and $S$ is a dihedral group $\D_m$, then it is possible that for some subset of edges $F$ of the $\D_m$-gain graph, the group $\langle F\rangle_v$ is neither trivial nor the entire group $\D_m$, but the cyclic subgroup $\C_m$ of $\D_m$. In that case, we need to adjust the number $k_{\D_m}=0$ to $k_{\langle F\rangle_v}=1$ in the sparsity count for $F$, where $k_{\langle F\rangle_v}$ is the dimension of the space of isometries of $\mathbb{R}^3$ which act tangentially on the unit sphere and are symmetric with respect to the group $\langle F\rangle_v=\C_m$.
 With this in mind, the following is proved similarly to Theorem \ref{thm:maxwell} (see, also, \cite[Lemma 5.2]{jkt}). 

\begin{thm}\label{thm:maxwellnon-abelian}
Let $\M$ be an irreducible algebraic surface of type $k$.
Let $S$ be a non-cyclic symmetry group of $\bR^3$ acting on $\M$ such that $\M$ has type $k_{S}$ under $S$. 
Let $(G,p)$ be an  $S$-isostatic framework in $\mathscr{R}^{\M}_{(G,S,\theta)}$ with quotient $S$-gain graph $(G_0,\psi)$. Then $(G_0,\psi)$ satisfies \begin{enumerate}
\item $|E(G_0)|=2|V(G_0)|-k_{S}$,
\item $|F|\leq 2|V(F)|-k_{\langle F\rangle_v}$ $\quad$ for all $F\subseteq E(G_0)$ and all $v\in V(F)$.
\end{enumerate}
\end{thm}


\section{Combinatorial characterisations of generic rigidity} \label{sec:combchar}

In the rest of this paper we will consider the more substantial problem of proving these counts are sufficient to guarantee that a symmetric framework supported on a surface is symmetry-forced isostatic. We will focus on three classical surfaces in $3$-space, namely the sphere, the cylinder and the cone. That is
\begin{itemize}
\item the unit sphere $\S$ centered at the origin, defined by the equation $x^2+y^2+z^2=1$;
\item the unit cylinder $\Y=S^1 \times \bR$ about the $z$-axis, defined by the equation $x^2+y^2=1$;
\item the unit cone $\C$ about the $z$-axis, defined by the  equation $x^2+y^2=z^2$.
\end{itemize}

In the Schoenflies notation, the relevant symmetry groups which are compatible with each of these surfaces are $\mathcal{C}_s$, $\mathcal{C}_m$, $\mathcal{C}_i$, $\mathcal{C}_{mv}$,  $\mathcal{C}_{mh}$,  $\mathcal{D}_{m}$, $\mathcal{D}_{mh}$, $\mathcal{D}_{md}$
 and $S_{2m}$.  As defined in Section~\ref{sec:symfw}, $\mathcal{C}_s$ is generated by a single reflection $s$, and $\mathcal{C}_m$, $m\geq 1$, is a group
generated by an $m$-fold rotation $C_m$. $\mathcal{C}_i$ is the group generated by an inversion, $\mathcal{C}_{mv}$ is a dihedral group that is generated by a rotation $C_m$ and
a reflection  whose reflectional plane contains the rotational axis of $C_m$, and 
 $\mathcal{C}_{mh}$ is generated by a rotation $C_m$ and the reflection  whose reflectional plane is perpendicular to the axis of $C_m$. Further, $\mathcal{D}_{m}$ denotes a symmetry group that
is generated by a rotation $C_m$ and another $2$-fold rotation $C_2$ whose rotational axis is perpendicular to the one of $C_m$. $\mathcal{D}_{mh}$ and $\mathcal{D}_{md}$ are generated by the generators $C_m$ and $C_2$ of a group $\mathcal{D}_{m}$ and by a reflection $s$.  In the case of $\mathcal{D}_{mh}$, the
mirror of $s$ is the plane that is perpendicular to the $C_m$ axis and
contains the origin (and hence contains the rotational axis of $C_2$), whereas in
the case of $\mathcal{D}_{md}$, the mirror of $s$ is a plane that contains the $C_m$ axis and forms an angle of $\frac{\pi}{m}$ with the $C_2$ axis.  Finally, $S_{2m}$ is a symmetry group which is generated by a $2m$-fold improper rotation (i.e., a rotation by $\frac{\pi}{m}$ followed by a reflection in the plane which is perpendicular to the rotational axis).

\subsection{The Sphere}
\label{subsec:sphere}

Although it has not previously been stated, by combining results of  \cite{jkt} and \cite{BSWWconing}, the following theorem is immediate.

\begin{thm}[Rotation, reflection or dihedral symmetry on the sphere]\label{thm:spherecyclic}
Let $S$ be the group $\mathcal{C}_m$ representing $m$-fold rotational symmetry or the group $\mathcal{C}_s$ representing reflectional symmetry about a plane through the origin. 
Let $(G,p)$ be an $S$-generic framework in $\mathscr{R}^{\S}_{(G,S,\theta)}$ with quotient $S$-gain graph $(G_0,\psi)$.
Then $(G,p)$ is $S$-isostatic if and only if $(G_0,\psi)$ is $(2,3,1)$-gain-tight.

 Moreover, if $S$ is a dihedral group $\mathcal{C}_{mv}$, where $m$ is odd, then $(G,p)$ is $S$-isostatic if and only if $G_0$ is `maximum $\mathcal{D}$-tight' (as defined in \cite[Def. 7.1]{jkt}), i.e., if $G_0$ satisfies conditions 1 and 2 in Theorem~\ref{thm:maxwellnon-abelian} .
\end{thm}

\begin{proof} Let $S$ be one of the groups listed above. We may think of an $S$-symmetric framework $(G,p)$ supported on the unit sphere as the `coned framework' $(G*0,p^*)$, where $(G*0,p^*)$ is the framework obtained from $(G,p)$ by adding a new joint $p_0$ at the origin (i.e., the centre of the sphere) which is linked to every joint of $(G,p)$ by a bar. We may now invert  vertex orbits (under the symmetry group $S$ of $(G*0,p^*)$) so that we obtain a framework on the upper half-sphere, and then project (gnomonically) the resulting framework from the origin to the plane $z=1$. Note that this yields a framework $(G,q)$ in the plane which also has symmetry $S$. Moreover, as shown in \cite{BSWWorbit}, the $S$-symmetric infinitesimal rigidity properties of $(G*0,p^*)$ and $(G,q)$ are the same. Therefore, the result follows directly from \cite[Theorems 6.3 and 8.2]{jkt} and \cite[Theorems 3.7 and 6.2]{BSWWconing}.
\end{proof}

We note that in the cases of rotation and reflection symmetry groups, the proof techniques we employ below can easily be adapted to give direct inductive proofs of these results (see also \cite{schtan}). We leave the details to the reader. We also point out that for these non-dihedral groups, we could also prove the theorem using \cite[Theorems 3 and 4]{mt1}.

Note that for dihedral groups of the form $\mathcal{C}_{mv}$, where $m$ is even, there does not exist a combinatorial characterisation of symmetry-generic symmetry-forced rigid frameworks in the plane. For example, it was shown in \cite{jkt} that Bottema's mechanism (a realisation of the complete bipartite graph $K_{4,4}$ with $\mathcal{C}_{2v}$ symmetry in the plane) is falsely predicted to be $\mathcal{C}_{2v}$-symmetric infinitesimally rigid by the sparsity counts for the orbit rigidity matrix. Thus, the corresponding results for the sphere also remain open.

In general, for any point group $S$ of a framework on the sphere, except for the groups $\mathcal{C}_m$, $\mathcal{C}_s$, $\mathcal{C}_i$, $\mathcal{C}_{mh}$ and $\mathcal{S}_{2m}$ there are no tangential isometries (i.e., no rotations) which are $S$-symmetric. Thus, for those groups, we need to cope with the $(2,3,0)$-gain-sparsity count to establish characterisations for symmetry-forced rigidity on the sphere. This is a significant obstacle, as it was recently observed that this gain-sparsity count is in general not matroidal \cite{jkt}.

For the group $\mathcal{C}_i$, we will prove the following theorem in the subsequent sections.

\begin{thm}[Inversion symmetry on the sphere]\label{thm:sphereinversion} \label{thm:sphereinv}
Let $S$ be the group $\mathcal{C}_i$ representing  inversion symmetry. 
Let $(G,p)$ be an $S$-generic framework in $\mathscr{R}^{\S}_{(G,S,\theta)}$ with quotient $S$-gain graph $(G_0,\psi)$.
Then $(G,p)$ is $S$-isostatic if and only if $(G_0,\psi)$ is $(2,3,3)$-gain-tight.
\end{thm}

 Note that the covering graph of a  $(2,3,3)$-gain-tight graph is not $(2,3)$-tight. For example, for the group $\mathcal{C}_i$, a triangle whose edges all have trivial gains lifts to the disjoint union of two triangles. By Theorem~\ref{thm:sphereinv}, a $\mathcal{C}_i$-generic realisation of this graph on the sphere is  $\mathcal{C}_i$-isostatic.

For the groups $\mathcal{C}_{mh}$, where $m$ is odd, and $S_{2m}$, where $m$ is even, we will prove the following theorem in the subsequent sections.

\begin{thm}[Improper rotational symmetry on the sphere]\label{thm:spheredihedral} \label{thm:spheredih}
Let $S$ be the \break group $\mathcal{C}_{mh}$, where $m$ is odd, or $S_{2m}$, where $m$ is even. 
Let $(G,p)$ be an $S$-generic framework in $\mathscr{R}^{\S}_{(G,S,\theta)}$ with quotient $S$-gain graph $(G_0,\psi)$.
Then $(G,p)$ is $S$-isostatic if and only if $(G_0,\psi)$ is $(2,3,1)$-gain-tight.
\end{thm}

For the remaining groups $\mathcal{C}_{mh}$, where $m$ is even, and $\mathcal{S}_{2m}$, where $m$ is odd, we will provide the corresponding conjectures in Section~\ref{sec:furtherwork}. (See also Table 1.) Note that these groups contain the group $\mathcal{C}_i$ as a subgroup (whereas $\mathcal{C}_{mh}$, where $m$ is odd, and $\mathcal{S}_{2m}$, where $m$ is even, do not). 
Thus, we need to consider gain-sparsity counts which depend on the groups $\langle F\rangle_v$ induced by edge subsets $F$ of the gain graph $G_0$, since there is only a $1$-dimensional space of rotations which is symmetric with respect to the subgroups $\mathcal{C}_{s}$, $\mathcal{C}_{m}$, $\mathcal{C}_{m'h}$, or $\mathcal{S}_{2m'}$, but there is a $3$-dimensional space of rotations which is symmetric with respect to the inversion subgroup $\mathcal{C}_i$.


\subsection{The cylinder}

We now consider the case of the surface being a cylinder. Note that the cylinder has point group symmetry $\mathcal{D}_{\infty h}$, and hence the possible point groups of frameworks on the cylinder are $\mathcal{C}_s$, $\mathcal{C}_m$ (with the rotational axis being the axis of the cylinder for $m>2$), $\mathcal{C}_i$, $\mathcal{C}_{mv}$,  $\mathcal{C}_{mh}$, $\mathcal{D}_{m}$, $\mathcal{D}_{mh}$, $\mathcal{D}_{md}$ and $S_{2m}$. 

We will prove the following two main theorems for the cylinder in the subsequent sections.

\begin{thm}[Rotation symmetry on the cylinder] \label{thm:cylindercyclic}
Let $S$ be the symmetry group $\mathcal{C}_m$ representing $m$-fold rotational symmetry around the $z$-axis. 
Let $(G,p)$ be an $S$-generic framework in $\mathscr{R}^{\Y}_{(G,S,\theta)}$ with quotient $S$-gain graph $(G_0,\psi)$.
Then $(G,p)$ is $S$-isostatic if and only if $(G_0,\psi)$ is $(2,2,2)$-gain-tight.
\end{thm}

\begin{thm}[Reflection or inversion symmetry on the cylinder]\label{thm:cylinderreflection}
Let $S$ be the group $\mathcal{C}_s$ \break (where the mirror plane of the reflection either contains the $z$-axis or is equal to the plane $z=0$) or the inversion group $\mathcal{C}_i$. 
Let $(G,p)$ be an $S$-generic framework in $\mathscr{R}^{\Y}_{(G,S,\theta)}$ with quotient $S$-gain graph $(G_0,\psi)$.
Then $(G,p)$ is  $S$-isostatic if and only if $G_0$ is $(2,2,1)$-gain-tight.
\end{thm}

We will discuss the remaining groups and provide some conjectures in Section~\ref{sec:furtherwork}. (See Table 2.)


\subsection{The cone} Note that the cone defined by the polynomial $x^2+y^2=z^2$ has the same point group symmetry $\mathcal{D}_{\infty h}$ as the cylinder, and hence the possible point groups of frameworks on the cone are the same as the ones for the cylinder. 

We will prove the following main theorem for the cone in the subsequent sections.

\begin{thm}[Reflection, rotation, inversion or improper rotational symmetry on the cone]\label{thm:conecyclic}
Let $S$ be the group $\mathcal{C}_m$ representing $m$-fold rotation around the $z$-axis, or the group $\mathcal{C}_s$ (where the mirror plane of the reflection is perpendicular to the $z$-axis), or the inversion group $\mathcal{C}_i$, or the group $\mathcal{C}_{mh}$, or the improper rotational group $S_{2m}$. 
Let $(G,p)$ be an $S$-generic framework in $\mathscr{R}^{\C}_{(G,S,\theta)}$ with quotient $S$-gain graph $(G_0,\psi)$.
Then $(G,p)$ is $S$-isostatic if and only if $(G_0,\psi)$ is $(2,1,1)$-gain-tight.
\end{thm}

Again we will return to the remaining groups in Section~\ref{sec:furtherwork}. (See Table 3.)

\subsection{A note on  $(2,2,1)$-gain-tight graphs}

We finish this section by observing a corollary to Theorem~\ref{thm:cylinderreflection} which points out that for certain types of graphs, groups and surfaces, generic rigidity (without symmetry) is equivalent to symmetry-generic symmetry-forced rigidity. We will need the following result.

\begin{prop}\label{prop:22tight} Let $S$ be the group $\mathbb{Z}_2=\{0,1\}$, and let $G$ be a simple $S$-symmetric graph, where the action $\theta:S\to \rm{Aut}(G)$ is free on both $V(G)$ and $E(G)$. Then $G$ is $(2,2)$-tight if and only if its quotient $S$-gain graph $G_0$ is $(2,2,1)$-gain-tight. 
\end{prop}

\begin{proof} Suppose first that $G$ is $(2,2)$-tight, and $G_0$ is not $(2,2,1)$-gain-tight. If $|E(G_0)|\neq 2|V(G_0)|-1$, then $|E(G)|\neq 2|V(G)|-2$, a contradiction. If $|F|> 2|V(F)|-1$ for some $F\subseteq E(G_0)$, then  $|c^{-1}(F)|>2|c^{-1}(V(F))|-2$, where $c:G\to G_0$ is the covering map. This is a again a contradiction to $G$ being $(2,2)$-tight. Similarly, if $|F|> 2|V(F)|-2$ for some balanced $F\subseteq E(G_0)$, then, by Proposition 2, we may switch the vertices in $V(F)$ so that each edge  of $F$ has trivial gain. But then $E(G)$ contains a copy of $F$ as a subset, which again violates the $(2,2)$-tightness of $G$. 

Conversely, suppose that $G_0$ is $(2,2,1)$-gain-tight, but $G$ is not $(2,2)$-tight. Clearly, if  $|E(G)|\neq 2|V(G)|-2$, then $|E(G_0)|\neq 2|V(G_0)|-1$, a contradiction. Thus,  there exists a subgraph $H$ of $G$ with $|E(H)|> 2|V(H)|-2$.  Note that $H$ is not $S$-symmetric, for otherwise $c(H)$ violates the $(2,2,1)$-gain-sparsity count.  Clearly, we may choose $H$ so that it has a minimal number of vertices, and we may then remove edges if necessary so that $|E(H)|= 2|V(H)|-1$. Let $H'$ be the symmetric copy of $H$, i.e., $H'$ has the vertex set $\{\theta(1)(v) |\,v\in V(H)\}$ and edge set $\{\theta(1)(e) |\,e\in E(H)\}$. By symmetry, we have $|E(H')|= 2|V(H')|-1$.

\begin{claim} If $V(H\cap H')=\emptyset$, then $c(H\cup H')$ is balanced.
\end{claim}
\begin{proof} Suppose $V(H \cap H')=\emptyset$ and $c(H\cup H')$ is unbalanced. Then there exists an unbalanced cycle $C$ in $c(H\cup H')$ and we may switch the vertices of $C$ so that all edges of $C$ have gain $0$, except for a single edge, say $(v,w)$. Then in $H\cup H'$ we have the edges $\{v, \theta(1)(w)\}$ and $\{\theta(1)(v),w\}$, as well as paths from $v$ to $w$ and  from $\theta(1)(v)$ to $\theta(1)(w)$. If $\{v, \theta(1)(w)\}$ and $\{\theta(1)(v),w\}$ are both in $E(H)$, then they are also both in $E(H')$, a contradiction. Analogously, we obtain a contradiction if $\{v, \theta(1)(w)\},\{\theta(1)(v),w\}\in E(H')$.

Thus, wlog we may assume that $\{v, \theta(1)(w)\}\in E(H)\setminus E(H')$ and $\{\theta(1)(v),w\}\in E(H')\setminus E(H)$. Consider a path $v, v_1, v_2, \ldots, v_k, w$ from $v$ to $w$ in $H\cup H'$. Clearly, $\{v,v_1\}\in E(H)\setminus E(H')$ for otherwise $v\in  V(H\cap H')$. Continuing in this fashion, we deduce that $\{v_k,w\}\in E(H)\setminus E(H')$. Thus,  $w\in V(H)$. However, $w$ is also a vertex of  $H'$ since $\{\theta(1)(v),w\}\in E(H')\setminus E(H)$. This contradicts $V(H \cap H')=\emptyset$, and the proof of the claim is complete. 
\end{proof}

Now, suppose first that $V(H\cap H')=\emptyset$. Then  we have $|E(H \cup H')|= 2|V(H)|-1 +2|V(H' )|-1=  2|V(H \cup H')|-2$, and hence $|c(E(H\cup H'))|=2|c(V(H\cup H'))|-1$. Thus,  $c(H\cup H')$ violates the $(2,2,1)$-gain-sparsity count since $c(H\cup H')$ is balanced (by the claim).

So it remains to consider the case  $V(H\cap H')\neq \emptyset$.  Since $H$ is not $S$-symmetric, we have $H\neq H'$. Thus, by the minimality of $H$, we have  $|E(H\cap H')|\leq 2|V(H\cap H')|-2$. It follows that $|E(H \cup H')|\geq  2|V(H)|-1 +2|V(H' )|-1- (2|V(H\cap H')|-2)=  2|V(H \cup H')|$. Thus, $|c(E(H\cup H'))|\geq 2|c(V(H\cup H'))|$, and hence $c(H\cup H')$ violates the $(2,2,1)$-gain-sparsity count. 

\end{proof}

Thus, by combining Theorem \ref{thm:surf3} with Theorem~\ref{thm:cylinderreflection}, it follows that for a graph $G$, satisfying the conditions of Proposition~\ref{prop:22tight}, a generic realisation of $G$ on the cylinder $\Y$ is isostatic if and only if  an $S$-generic realisation of $G$ on $\Y$ is $S$-isostatic, where $S=\C_s$ or $\C_i$ (see also Conjecture~\ref{con:incsym}).



\section{Inductive Constructions} \label{sec:recchar}

To prove Theorems \ref{thm:sphereinversion}--\ref{thm:conecyclic} we will make use of a celebrated proof technique in rigidity theory: an inductive proof using Henneberg-type graph constructions \cite{jkt,nop,N&R,Ross,schtan,W1}. This comes in two steps. First we prove a characterisation of $(2,k,k_S)$-gain-tight graphs, for the relevant choices of $k$ and $k_S$, showing that all such graphs can be generated from the smallest such graph, or graphs, by simple operations. Then we apply these operations to frameworks and show that they preserve the dimension of the nullspace of the orbit-surface rigidity matrix. 

\subsection{Admissible Operations}

Let $(H,\psi)$ be an $S$-gain graph. The \emph{Henneberg 1 move} is the addition of a new vertex $v$ and two edges $e_1$ and $e_2$ to $H$ such that the new edges are incident with $v$ and are not both loops at $v$. The three possible ways this can be done are illustrated in Figure \ref{fig:hen1}.  Note that if $e_1$ and $e_2$ are not loops and not parallel edges, then their labels can be arbitrary. This move is called H1a and is depicted  in Figure \ref{fig:hen1}(a). If they are non-loops, but parallel edges, then the labels are assigned so that $\psi(e_1)\neq \psi(e_2)$, assuming that $e_1$ and $e_2$ are directed to $v$. This move is called H1b and is depicted in Figure \ref{fig:hen1}(b). Finally, if one of the edges, say $e_1$, is a loop at $v$, then we set $\psi(e_1)\neq id$. This move is called H1c and is depicted in Figure \ref{fig:hen1}(c).

\begin{center}
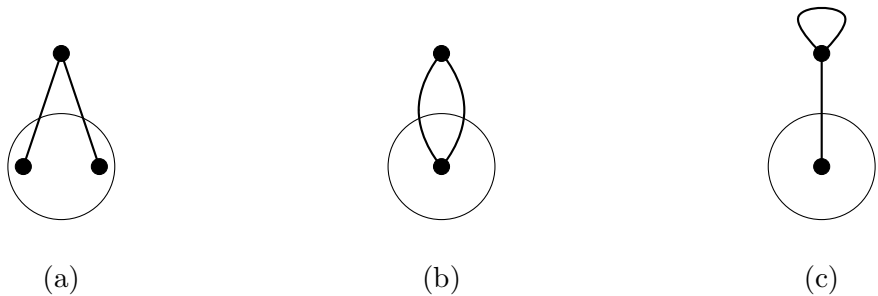
\begin{figure}[ht]
\centering
\begin{tikzpicture}\tikzstyle{every node}=[circle, draw=black, fill=black, inner sep=0pt, minimum width=5pt];
 \draw (-.3,0.75) circle (0.1pt){};
 \draw (4.7,2) circle (0.1pt) {};
 \draw (0,0) circle (20pt);
 \draw (5,0) circle (20pt);
 \draw (-5,0) circle (20pt);
 \filldraw (-5.5,0) circle (3pt);
  \filldraw (-4.5,0) circle (3pt);
   \filldraw (0,0) circle (3pt);
    \filldraw (5,0) circle (3pt);
     \filldraw (-5,1.5) circle (3pt);
     \filldraw (0,1.5) circle (3pt);
     \filldraw (5,1.5) circle (3pt);
 \draw[black,thick]
(-5.5,0) -- (-5,1.5) -- (-4.5,0);
 \draw[black,thick]
(5,0) -- (5,1.5);
\draw[thick] plot[smooth, tension=1] coordinates{(0,0) (-.3,0.75) (0,1.5)};
\draw[thick] plot[smooth, tension=1] coordinates{(0,0) (.3,0.75) (0,1.5)};
\draw[thick] plot[smooth, tension=1] coordinates{(5,1.5) (4.7,2) (5.3, 2) (5,1.5)};
\node [draw=white, fill=white] (b) at (-5,-1.5) {(a)};
\node [draw=white, fill=white] (b) at (0,-1.5) {(b)};
\node [draw=white, fill=white] (b) at (5,-1.5) {(c)};
\end{tikzpicture}
\caption{H1a, H1b and H1c moves. Directions and gain labels of the edges are omitted in the figure.  }
\label{fig:hen1}
\end{figure}
\end{center}

The \emph{Henneberg 2 move} deletes an edge of $(H,\psi)$ and adds a new vertex and three new edges to $(H,\psi)$.
First, one chooses an edge $e$ of $H$ (which will be deleted) and a vertex $z$ of $H$ which may be an end-vertex of $e$. Then one subdivides $e$, with a new vertex $v$ and new edges $e_1$ and $e_2$, such that the tail of $e_1$ is the tail of $e$ and the tail of $e_2$ is the head of $e$. The gains of the new edges are assigned so that $\psi(e_1)\cdot \psi(e_2)^{-1}=\psi(e)$. Finally, we add a third new edge, $e_3$, to $H$. This edge is oriented from $z$ to $v$ and its gain is such that every 2-cycle $e_i e_j$, if it exists, is unbalanced. 
There are four possible ways this can be done, as illustrated in Figure \ref{fig:hen2}.

Suppose first that the edge $e$ is not a loop. If none of the edges $e_i$ is a loop or a parallel edge, then the move is called H2a (see Figure~\ref{fig:hen2}(a)). If none of the edges $e_i$ is a loop, but exactly two of the edges are parallel edges (i.e., $z$ is an end-vertex of $e$), then the move is called H2b (see Figure~\ref{fig:hen2}(b)). If the edge $e$ is a loop, then the moves corresponding to H2a and H2b are called H2c and H2d, respectively 
(see Figures~\ref{fig:hen2}(c) and (d)).

\begin{center}
\begin{figure}[ht]
\centering
\begin{tikzpicture}
 \draw (3.4,0.75) circle (0.01pt) node[anchor=east]{};
  \draw (3.4,0.75) circle (0.01pt) node[anchor=west]{};
   \draw (4.4,0.75) circle (0.01pt) node[anchor=west]{};
   \draw (-0.5,0) circle (0.01pt) node[anchor=north]{};
 \draw (4,0) circle (0.01pt) node[anchor=north]{};
  \draw (0,.4) circle (0.01pt) node[anchor=east]{};
 \draw (-.31,0.84) circle (0.01pt) node[anchor=east]{};
 \draw (0,0) circle (27pt);
 \draw (4,0) circle (27pt);
   \filldraw (-0.7,0) circle (3pt);
   \filldraw (0.7,0) circle (3pt);
    \filldraw (0,-.3) circle (3pt);
    \filldraw (3.3,0.1) circle (3pt);
    \filldraw (4.7,0.1) circle (3pt);
     \filldraw (0,1.5) circle (3pt);
     \filldraw (4,1.5) circle (3pt);
\draw[black,dashed]
(3.3,0.1) -- (4.7,0.1);
\draw[black,dashed]
(-.7,0.1) -- (0,-.3);
\draw[black,thick]
(0,-.3) -- (0,1.5) -- (-.7,0.1);
\draw[black,thick]
(4,1.5) -- (4.7,0.1);
\draw[black,thick]
(0,1.5) -- (0.7,0);
\draw[thick] plot[smooth, tension=1] coordinates{(3.3,0) (3.4,0.75) (4,1.5)};
\draw[thick] plot[smooth, tension=1] coordinates{(3.3,0) (3.9,0.75) (4,1.5)};

 \draw (7.6,0.75) circle (0.01pt) node[anchor=east]{};
  \draw (7.6,0.75) circle (0.01pt) node[anchor=west]{};
   \draw (8.2,0.75) circle (0.01pt) node[anchor=west]{};
 
 \draw (7.8,0) circle (0.01pt) node[anchor=north]{};
  \draw (7.8,.4) circle (0.01pt) node[anchor=east]{};
 \draw (7.8,0) circle (27pt);

    \filldraw (7.5,0.1) circle (3pt);
    \filldraw (8.5,0.1) circle (3pt);

     \filldraw (7.8,1.5) circle (3pt);
\draw[dashed] plot[smooth, tension=1] coordinates{(7.5,0.1) (7.2,-0.4) (7.8, -0.4) (7.5,0.1)};
\draw[black,thick]
(7.8,1.5) -- (8.5,0.1);
\draw[thick] plot[smooth, tension=1] coordinates{(7.5,0) (7.4,0.75) (7.8,1.5)};
\draw[thick] plot[smooth, tension=1] coordinates{(7.5,0) (7.9,0.75) (7.8,1.5)};

 \draw (11.6,0.75) circle (0.01pt) node[anchor=east]{};
  \draw (11.6,0.75) circle (0.01pt) node[anchor=west]{};
 
 \draw (11.8,0) circle (0.01pt) node[anchor=north]{};
  \draw (11.8,.4) circle (0.01pt) node[anchor=east]{};
 \draw (11.8,0) circle (27pt);

    \filldraw (11.5,0.1) circle (3pt);

     \filldraw (11.8,1.5) circle (3pt);
\draw[dashed] plot[smooth, tension=1] coordinates{(11.5,0.1) (11.2,-0.4) (11.8, -0.4) (11.5,0.1)};
\draw[black,thick]
(11.5,0) -- (11.8,1.5);
\draw[thick] plot[smooth, tension=1] coordinates{(11.5,0) (11.4,0.75) (11.8,1.5)};
\draw[thick] plot[smooth, tension=1] coordinates{(11.5,0) (11.9,0.75) (11.8,1.5)};

\node [draw=white, fill=white] (b) at (0,-1.5) {(a)};
\node [draw=white, fill=white] (b) at (4,-1.5) {(b)};
\node [draw=white, fill=white] (b) at (8,-1.5) {(c)};
\node [draw=white, fill=white] (b) at (12,-1.5) {(d)};
\end{tikzpicture}
\caption{H2a, H2b, H2c and H2d moves. Directions and gain labels of the edges are omitted in the figure.}
\label{fig:hen2}
\end{figure}
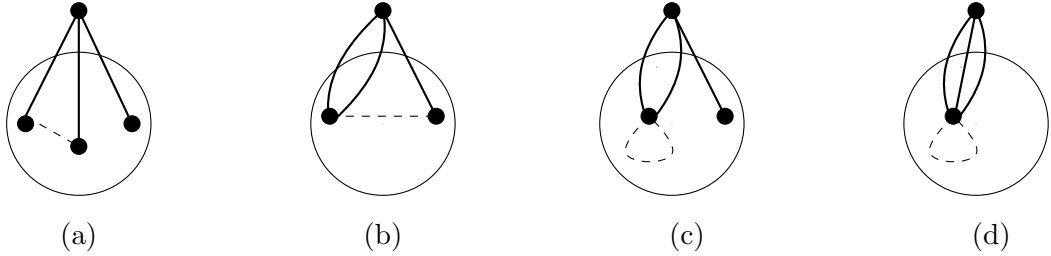
\end{center}

For a $(2,k,k_S)$-gain-tight $S$-gain graph $(H,\psi)$, an inverse Henneberg 1 or 2 move on $v\in V(G_0)$ is \emph{admissible} if the resulting $S$-gain graph $H'$ is $(2,k,k_S)$-gain-tight and the
covering graph of $H'$ is simple.

A \emph{vertex-to-$K_4$} operation on an  $S$-gain graph $(H,\psi)$ removes a vertex $v$ (of arbitrary degree) and 
all the edges incident with $v$, and adds in a copy of $K_4$ with only trivial gains (see also Figure~\ref{fig:vk4}). Without loss of generality we may assume that all edges incident with $v$ are directed to $v$, i.e., are of the form $(x,v)$ for some $x\in V(H)$.
Each removed edge $(x,v)$ is replaced by an edge $(x,y)$ for some $y$ in the new $K_4$, where the gain is preserved, that is, $\psi((x,v))=\psi((x,y))$. 

For a $(2,k,k_S)$-gain-tight $S$-gain graph $(H,\psi)$, the inverse move, a \emph{$K_4$-contraction} on a copy of $K_4$ 
with only trivial gains
 is admissible if the resulting graph $H'$ is $(2,k,k_S)$-gain-tight and the covering graph
is simple. Note that we only apply the $K_4$-contraction when there are no additional edges induced by the vertices of the $K_4$.

\begin{center}
\begin{figure}[ht]
\centering
\begin{tikzpicture}
\draw (0,0) circle (27pt);
\draw (5,0) circle (27pt);
 
\filldraw (0,1.5) circle (3pt);
\filldraw (.2,0.5) circle (3pt);
\filldraw (-.2,0.5) circle (3pt);
\filldraw (0.6,0.35) circle (3pt);
\filldraw (-0.6,0.35) circle (3pt);

\filldraw (4.65,1.35) circle (3pt);
\filldraw (5.35,1.35) circle (3pt);
\filldraw (4.65,1.8) circle (3pt);
\filldraw (5.35,1.8) circle (3pt);

\filldraw (5.2,0.5) circle (3pt);
\filldraw (4.8,0.5) circle (3pt);
\filldraw (5.6,0.35) circle (3pt);
\filldraw (4.4,0.35) circle (3pt);

\draw[black]
(-.6,0.35) -- (0,1.5)  -- (-.2,.5) -- (0,1.5) -- (.2,.5);
\draw[black]
(.6,.35) -- (0,1.5);

\draw[black]
(2,0) -- (2.85,0);

\draw[black]
(2.85,0.15) -- (2.85,-.15) -- (3,0) -- (2.85,.15); 

\draw[black]
(4.4,0.35) -- (4.65,1.35) -- (5.35,1.35) -- (4.65,1.8) -- (5.35,1.8) -- (5.35,1.35) -- (5.6,.35);

\draw[black]
(5.2,0.5) -- (4.65,1.8) -- (4.65,1.35) -- (4.8,.5);

\draw[black]
(4.65,1.35) -- (5.35,1.8);

\end{tikzpicture}
\caption{The vertex-to-$K_4$ operation (in this case expanding a degree 4 vertex which is not incident to any loop). Directions and gain labels of the edges are omitted.}
\label{fig:vk4}
\end{figure}
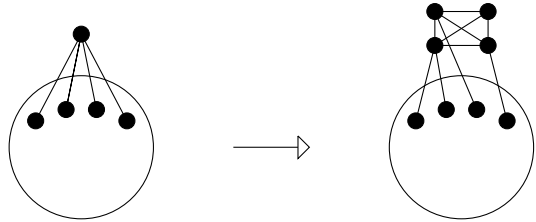
\end{center}

 In the following, an edge $e=(x,v)$ with gain $\psi(e)$ in $(H,\psi)$ will be denoted by $(x,v)_{\psi(e)}$. 
A \emph{vertex-to-4-cycle} operation on an $S$-gain graph $(H,\psi)$ removes a vertex $v$ and all the edges incident with $v$, adds in two new vertices $v_1,v_2$, and chooses two neighbours $a,b$ of $v$ (w.l.o.g. with edges $(a,v)_\alpha$ and $(b,v)_\beta$) and creates a 4-cycle $a,v_1,b,v_2$ with edges $(a,v_1)_\alpha, (a,v_2)_\alpha, (b,v_1)_\beta, (b,v_2)_\beta$ (see also Figure~\ref{fig:4cycle}). Each of the removed edges $(x,v)$, $x\neq a,b$, is replaced by an edge $(x,v_i)$ for some $i=1,2$, where the gain is preserved, that is, $\psi((x,v))=\psi((x,v_i))$. If the deleted vertex $v$ is incident to a loop $(v,v)$ in $H$ (this may be the case if $H$ is $(2,k,k_S)$-gain tight for $k_S=1$, for example), then this loop is replaced by a loop $(v_i,v_i)$ (with the same gain) for some $i=1,2$.

The inverse operation, a \emph{4-cycle contraction}, on a $(2,k,k_S)$-gain-tight $S$-gain graph is admissible if the resulting graph $H'$ is $(2,k,k_S)$-gain-tight and the covering graph of $H'$ is simple. It will suffice for our purposes to contract only 4-cycles in which each edge has trivial gain. Thus, we may restrict to $\alpha$ and $\beta$ above both equaling the identity, so it is easy to see that if $H$ is $(2,k,k_S)$-gain-tight, then applying a vertex-to-4-cycle operation to $H$ results in a $(2,k,k_S)$-gain-tight graph.

\begin{center}
\begin{figure}[ht]
\centering
\begin{tikzpicture}
\draw (0,0) circle (27pt);
\draw (5,0) circle (27pt);

\filldraw (0,0) circle (3pt);
\filldraw (-.7,-.1) circle (3pt);
\filldraw (.7,.1) circle (3pt);

\filldraw (5,.5) circle (3pt);
\filldraw (5,-.5) circle (3pt);
\filldraw (4.3,-.1) circle (3pt);
\filldraw (5.7,.1) circle (3pt);

\draw[black]
(-.7,-.1) -- (0,0) -- (.7,.1);

\draw[black]
(-.2, .2) -- (0,0) -- (.2,.2);

\draw[black]
(-.2,- .2) -- (0,0) -- (.2,-.2);

\draw[black]
(4.3, -.1) -- (5,.5) -- (5.7,.1) -- (5,-.5) -- (4.3,-.1);

\draw[black]
(4.8, .7) -- (5,.5) -- (5.2,.7);

\draw[black]
(4.8, -.7) -- (5,-.5) -- (5.2,-.7);

\draw[black]
(2,0) -- (2.85,0);

\draw[black]
(2.85,0.15) -- (2.85,-.15) -- (3,0) -- (2.85,.15); 
\end{tikzpicture}
\caption{The vertex-to-4-cycle operation. Directions and gain labels of the edges are omitted.}
\label{fig:4cycle}
\end{figure}
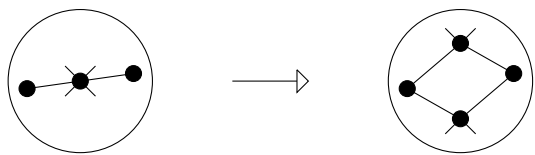
\end{center}

An \emph{edge joining} operation takes two $(2,k,1)$-gain-tight $S$-gain graphs $H_0,H_1$ and creates the graph $H_0 \oplus H_1$ which has vertex set the disjoint union of $V(H_0)$ and $V(H_1)$ and edge set $E(H_0)\cup E(H_1) \cup \{(a,b)\}$ where $a\in V(H_0)$ and $b\in V(H_1)$ is an edge with arbitrary gain. It is clear that $H_0 \oplus H_1$ is $(2,k,1)$-gain-tight if and only if $H_0$ and $H_1$ are $(2,k,1)$-gain-tight.

Note that in the covering graph, each of the above operations is a graph operation that preserves the underlying symmetry.
Some of them can be recognised as performing standard - non-symmetric - Henneberg operations~\cite{BSlaman,W1}
simultaneously.


\subsection{Recursive Characterisations}

We now derive inductive constructions for $(2,i,i)$-gain-tight graphs, where $i=1,2,3$, and for $(2,2,1)$-gain-tight graphs.
Inductive constructions of $(2,3,1)$-gain tight graphs were established in \cite[Theorem 4.4]{jkt} (using Henneberg 1 and Henneberg 2 moves only).
Note that as the balanced and unbalanced subgraph conditions are the same for $(2,i,i)$-gain-tight graphs,  $i=1,2,3$, we do not need to worry about preserving cycles with non-trivial gain in the first three theorems.

For a vertex $v$ of a directed multi-graph $G_0$, we will denote the set of vertices which are adjacent to $v$ (in the underlying undirected multi-graph) by $N(v)$. Each of the vertices in $N(v)$ is called a \emph{neighbour} of $v$ in $G_0$. 

In the following 4 theorems it is easy to establish that the operations under consideration preserve $(2,\ell,m)$-gain-sparsity. We concentrate on the converse where in all 4 cases it is clear that the minimum degree is 2 or 3.

\begin{thm}\label{thm:233construction}
Let $S$ be a group, $G$ be a simple graph and $(G_0,\psi)$ be its quotient $S$-gain graph. Then $(G_0,\psi)$ is $(2,3,3)$-gain-tight if and only if $(G_0,\psi)$ can be constructed sequentially from $K_2$ by H1a and H2a operations.
\end{thm}

\begin{proof} Note that if $G_0$ is $(2,3,3)$-gain-tight, then it cannot contain any loop or multiple edges. It is easy to see that if there exists a vertex $v\in V(G_0)$ of degree $2$, then the removal of $v$ yields another $(2,3,3)$-gain-tight graph. Thus, we may assume that $G_0$ has no vertex of degree $2$.
Since $G_0$ is $(2,3,3)$-gain-tight, every vertex of degree 3 must have
three distinct neighbors.
 But now, standard arguments can be used to show that an inverse H2a operation (performed on one of the degree 3 vertices and one pair of its neighbors)  preserves $(2,3,3)$-gain-tightness \cite{Lamanbib,BSlaman}.
\end{proof}

The next two theorems are reminiscent of \cite[Theorems 3.1 and 1.2]{nop}. The new difficulty being the existence of loops and multiple edges with certain gains. We will need the following lemma.

\begin{lem}\label{lemma:deg3}
Let $G_0$ be a $(2,i,i)$-gain-tight gain graph whose covering graph is simple, for $i=1,2$, with a vertex $v$ with $N(v)=\{a,b,c\}$ and no self-loop attached to $v$. Then either $v$ is contained in a copy of $K_4$ whose gain labelling is equivalent to the trivial gain labelling or there is an inverse H2a move on $v$ that results in a $(2,i,i)$-gain-tight gain graph $G_0'$ whose covering graph is simple.
\end{lem}

\begin{proof}
Let $\alpha,\beta,\gamma$ be the gains on $(v,a),(v,b),(v,c)$ respectively. If the edges $(a,b),(a,c),(b,c)$ are all present with gains $\beta-\alpha,\gamma-\alpha$ and $\gamma-\beta$ respectively, then any inverse H2a move will give a gain graph whose covering graph is not simple. In this case, though, the $K_4$ defined by these  vertices and edges is balanced so Propositions~\ref{prop:lem1} and \ref{prop:lem2} show the labelling is equivalent to the trivial gain labelling.

Thus we may suppose that not all 3 of these edges (with the specified gains) are present.
Now we can forget the gains and argue exactly as in \cite[Lemma 3.1]{no}. That is, by pure counting arguments, we deduce that, if $(a,b)_{\beta-\alpha}$ is not present then either there is an inverse H2a move adding $(a,b)_{\beta-\alpha}$ or another of the edges, say $(a,c)_{\gamma-\alpha}$, is not present and either we can add that edge or the third edge $(b,c)_{\gamma-\beta}$ is not present and we can add at least one of the edges. 
\end{proof}

\begin{thm}\label{thm:222construction}
Let $S$ be a  group, $G$ be a simple graph and $(G_0,\psi)$ be its quotient $S$-gain graph. Then $(G_0,\psi)$ is $(2,2,2)$-gain-tight if and only if $(G_0,\psi)$ can be constructed sequentially from $K_1$ by H1a, H1b, H2a, H2b, vertex-to-$K_4$ and vertex-to-4-cycle operations.
\end{thm}

\begin{proof} Note that if $G_0$ is $(2,2,2)$-gain-tight, then it cannot contain any loop. 
If there is no inverse $H1a$ or $H1b$ move then $G_0$ has minimum degree 3. It is elementary to show that a degree 3 vertex with exactly two neighbours can always be reduced using an inverse H2b move. So we may suppose that every degree 3 vertex $v$ has 3 distinct neighbours. By Lemma~\ref{lemma:deg3}, we may assume that $v$ is contained in a  $K_4$ whose gain labelling is equivalent to the trivial gain labelling.
The vertices of this $K_4$ can induce no additional edges since $G_0$ is $(2,2,2)$-gain-tight.
Denote this copy of $K_4$ as $K$.
$K$ is admissible for a $K_4$-to-vertex contraction unless there is a vertex $x \notin K$ and edges $(x,a),(x,b)$ for $a,b\in K$ with equal gains. Since the final vertex $c\in K$ is not adjacent to $x$, simple counting \cite{nop} shows there is a 4-cycle contraction merging $v$ and $x$ which results in a $(2,2,2)$-gain-tight graph contrary to our assumption. 
\end{proof}

\begin{thm}\label{thm:211construction}
Let $S$ be a group, $G$ be a simple graph, and $(G_0,\psi)$ be its quotient  $S$-gain graph. Then $(G_0,\psi)$ is $(2,1,1)$-gain-tight if and only if $(G_0,\psi)$ can be constructed sequentially from either a single vertex with an unbalanced loop, $K_4+e$ (with trivial gains on the $K_4$ and a non-trivial gain on $e$) or $K_5-e$ (with trivial gain on every edge) by H1a, H1b, H1c, H2a, H2b, H2c, H2d, vertex-to-$K_4$, vertex-to-4-cycle and edge joining operations.
\end{thm}

In the proof we use the terminology $2G$ to refer to the graph with vertex set $V(G)$ and 2 copies of each edge in $E(G)$. Of course we will be referring to gain graphs, so some appropriate gain assignment is assumed.  It will also be convenient to define $f(G):=2|V(G)|-|E(G)|$.

\begin{proof} 
Note that $G_0$ may contain loops. 
If there is no inverse H1a, H1b or H1c move then the minimum degree in $G_0$ is 3 and any such vertex has no self-loop. If there is no inverse H2d move then any such vertex $v$ has at least two neighbours.

Suppose $N(v)=\{a,b\}$ and there are two edges from $a$ to $v$. If there is at most one edge between $a$ and $b$ then any $H_0 \subset G_0$ containing $a$ and $b$ has $f(H_0)\geq 2$. Since the two edges from $a$ to $v$ have different gains (since $G$ is simple) it is easy to see that we can choose a gain for a new edge $ab$ that makes $H_0+ ab$ unbalanced. Thus we may assume $v$ is contained in a copy of $2K_3-e$. However, now there is an inverse H2c move.

Therefore, we may suppose that $v$ has 3 distinct neighbours.  By Lemma~\ref{lemma:deg3}, we may assume that $v$ is contained in a  $K_4$ whose gain labelling is equivalent to the trivial gain labelling.
 Denote this copy as $K$ and suppose, for now, that the vertices of $K$ induce no additional edges.

$K$ admits an admissible $K_4$-to-vertex contraction unless there is a vertex $x \notin K$ and edges $(x,a),(x,b)$ for $a,b\in K$ with equal gains. In such case there is a 4-cycle contraction merging $v$ and $x$ which results in a $(2,1,1)$-gain-tight graph unless for the final vertex $c\in K$ the edge $(x,c)\in E(G_0)$ with $\psi((x,c))=id$. (Note that if  $\psi((x,c))=id$, then $\psi((x,c))=\psi((v,c))=id$ so that a $4$-cycle-contraction would not result in a simple covering graph.)

Now, the graph induced by $v,a,b,c$ and $x$ is a copy of $K_5-e$. Repeat the whole process for every degree 3 vertex. This gives us copies of $K_5-e$ or, if the vertices of $K$ did induce additional edges, copies of $K_4+e$ (with non-zero gain on $e$).

We now argue as in \cite[Lemma 4.10]{no}. Let $Y=\{Y_1,\dots ,Y_n\}$ be the subgraphs which are copies of $K_5- e$ or $K_4+e$. They are necessarily vertex disjoint since $f(Y_i\cup Y_j)=2-f(Y_i\cap Y_j)$ and every proper subgraph $X$ of $K_5\backslash e$ has $f(X) \geq 2$.
Let $V_0$ and $E_0$ be the sets of vertices and edges of $G_0$ which are in none of the $Y_i$. Then
\[
f(G_0)=\sum_{i=1}^nf(Y_i)+2|V_0|-|E_0|
\]
so $|E_0|=2|V_0|+n-1.$ Each vertex in $V_0$ is incident to at least $4$ edges.
If every $Y_i$ is incident to at least $2$ edges in $E_0$, then there are at least $4|V_0|+2n$ edge/vertex incidences in $E_0$. This
implies $|E_0|\geq 2|V_0|+n$, a contradiction. Thus,
either there is a copy $Y_i$ with no incidences, which would imply $G_0=Y_i=K_5-e$ or $G_0=Y_i=K_4+e$, since $G_0$ is connected, or there is a copy with one incidence, i.e., $G_0$ contains a bridge and there is an edge separation move on this bridge contrary to our assumption.
\end{proof}

For the following we have to be more careful to preserve the gain-sparsity of subgraphs.

\begin{thm}\label{thm:221construction}
Let $S$ be a group of order 2, $G$ be a simple graph and $(G_0,\psi)$ be its quotient  $S$-gain graph. Then $(G_0,\psi)$ is $(2,2,1)$-gain-tight if and only if $(G_0,\psi)$ can be constructed sequentially from a single vertex with an unbalanced loop or from $K_4+e$ (where the $K_4$ has trivial gains and $e$ has a non-trivial gain) by H1a, H1b, H1c, H2a, H2b, H2c, vertex-to-$K_4$, vertex-to-4-cycle and edge joining operations.
\end{thm}

\begin{proof} We will think of $S$ as the group $\mathbb{Z}_2=\{0,1\}$ with addition as the
group operation.
If there is no inverse H1a, H1b or H1c move then the minimum degree in $G$ is 3 and any such vertex has no self-loop. 
Since $S$ has order 2, any such vertex $v$ has at least two neighbours.

First, suppose $N(v)=\{a,b\}$.
%
Suppose that $v$ is not contained in a copy of $2K_3-e$.
It suffices to check the case when $a$ and $b$ are not joined by an edge in $G_0$. If there is no admissible H2b move, then it is straightforward to deduce that there are distinct subgraphs $H_1,H_2$ of $G_0-v$ with $a,b\in V(H_i)$, $f(H_i)=2$ for $i=1,2$ and all paths in $H_i$ from $a$ to $b$ have gain $\alpha_i$ where $\alpha_1\neq \alpha_2$. (The gain of a path in a gain graph
 is defined analogously to the gain of a closed walk (recall Section~\ref{sec:balgaingr}).) Then $H_1\cap H_2$ is connected since $f(H_1\cap H_2)=2$, which implies that all paths from $a$ to $b$ in $H_1\cap H_2$ have 2 distinct gains, a contradiction. Thus, there is a  choice of gain for the edge $(a,b)$ so that the corresponding inverse H2b move  is admissible.
%
Thus $v$ is contained in a copy of $2K_3-e$ (with appropriate gains).
However, now there is an inverse H2c move.

Now let $N(v)=\{a,b,c\}$. 

\begin{claim}\label{clm:221k4part2}
$v$ is contained in a copy of $K_4$ in which every edge has gain $0$.
\end{claim}

\begin{proof}
We will show that if $v$ is not contained in a copy of $K_4$ with gain 0 on every edge, then there is an admissible H2a move. 

Let $\alpha_{av}, \alpha_{bv}, \alpha_{cv}$ be the gains on the edges $(a,v),(b,v),(c,v)$. Suppose first that $(a,b)_{\alpha_{av}-\alpha_{bv}}$ and $ (b,c)_{\alpha_{bv}-\alpha_{cv}} \notin E(G_0)$ (the edges $(a,b)$ and $(b,c)$ with other gains, or the edge 
$(a,c)$ with any gain may or may not be in $E(G_0)$). See Figure \ref{fig:proof1}.
If there is no admissible H2a move, then there must exist balanced subgraphs $H_{ab}$ of $G_0-v$ such that $a,b \in V(H_{ab})$ and $f(H_{ab})=2$ and $H_{bc}$ such that $b,c \in V(H_{bc})$ and $f(H_{bc})=2$.
If $c \in V(H_{ab})$ then let $H=H_{ab}$ and adding $v$ and its 3 incident edges to $H$ gives a graph $H^*$ with $f(H^*)=1$. Similarly if $a \in V(H_{bc})$ then we can let $H=H_{bc}$. If $c\notin V(H_{ab})$ and $a\notin V(H_{bc})$ then we still must have $f(H_{ab} \cup H_{bc})=2$. Thus we can let $H=H_{ab}\cup H_{bc}$ and adding $v$ and its 3 incident edges to $H$ gives a graph $H^*$ with $f(H^*)=1$. $H_{ab}\cap H_{bc}$ is connected, so Lemma~\ref{lem:balance} implies $H$ is balanced. 

Since $H$ is balanced we may apply gain switches to make every edge have identity gain.
Now, consider the three edges $va, vb$ and $vc$. If they all have 0 gains, then $H^*$ is balanced, contrary to the fact that $f(H^*)=1$. If they all have gain 1, then we can switch $v$ using Proposition \ref{prop:lem1} and again $H^*$ is balanced, a contradiction. 

By switching $v$ now, we may assume w.l.o.g. that $va$ has gain 1 and $vb$ and $vc$ have gain 0. Since $H_{ab}$ is connected, there is a path $P$ (with only gain 0 edges) from $a$ to $b$. Now, $P$ together with $va$ and $vb$ is an unbalanced cycle. Thus, if we do an inverse H2a move (adding the edge $ab$ with gain 1), then $H_{ab} \cup ab$ is unbalanced with the correct count so the move was valid.

\begin{center}
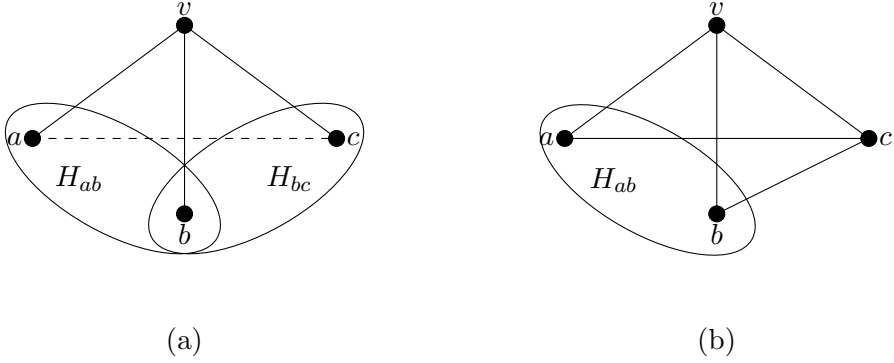
\begin{figure}[ht]
\centering
\begin{tikzpicture}
\filldraw (0,5) circle (3pt)node[anchor=south]{$v$};
\filldraw (0,2.5) circle (3pt)node[anchor=north]{$b$};
\filldraw (-2,3.5) circle (3pt)node[anchor=east]{$a$};
\filldraw (2,3.5) circle (3pt)node[anchor=west]{$c$};

\draw[rotate=330] (-2.3,2.1) ellipse (45pt and 20pt)node[anchor=east]{$H_{ab}$};
\draw[rotate=30] (2.3,2.1) ellipse (45pt and 20pt)node[anchor=west]{$H_{bc}$};

\draw[black]
(-2,3.5) -- (0,5) -- (2,3.5);
\draw[black]
(0,2.5) -- (0,5);

\draw[dashed]
(-2,3.5) -- (2,3.5);

\filldraw (7,5) circle (3pt)node[anchor=south]{$v$};
\filldraw (7,2.5) circle (3pt)node[anchor=north]{$b$};
\filldraw (5,3.5) circle (3pt)node[anchor=east]{$a$};
\filldraw (9,3.5) circle (3pt)node[anchor=west]{$c$};

\draw[rotate=330] (3.8,5.6) ellipse (45pt and 20pt)node[anchor=east]{$H_{ab}$};

\draw[black]
(7,2.5) -- (9,3.5) -- (5,3.5) -- (7,5) -- (9,3.5);
\draw[black]
(7,2.5) -- (7,5);

\node [draw=white, fill=white] (b) at (0,.8) {(a)};
\node [draw=white, fill=white] (b) at (7,.8) {(b)};
\end{tikzpicture}
\caption{The first two cases in the proof of the second claim. Directions and labels omitted.}
\label{fig:proof1}
\end{figure}
\end{center}

Now suppose that $(a,b)_{\alpha_{av}-\alpha_{bv}} \notin E(G_0)$ but $ (a,c)_{\alpha_{av}-\alpha_{cv}}$ and $(b,c)_{\alpha_{bv}-\alpha_{cv}} \in E(G_0)$. By Proposition \ref{prop:lem1}, we may then assume that the gains of $(v,a),(v,b),(v,c),(a,c),(b,c)$ are all 0. See Figure \ref{fig:proof1}.
Now if $H_{ab}$ is a subgraph of $G_0-v$ containing $a,b$ but not $c$, then $f(H_{ab})\geq 2$. We are done unless equality holds and $H_{ab}$ is balanced. Now consider paths from $a$ to $b$ in $H_{ab}$. If all such paths have 0 gain then all cycles in $H_{ab}\cup v \cup c$ (along with the relevant edges) have gain 0, but $f(H_{ab}\cup v \cup c)=1$. Thus, there must be a path from $a$ to $b$ with gain 1 and adding $(a,b)$ with gain 0 during the inverse H2a move gives an unbalanced subgraph, as required.

Finally, suppose  $(a,b)_{\alpha_{av}-\alpha_{bv}},  (a,c)_{\alpha_{av}-\alpha_{cv}}$ and $(b,c)_{\alpha_{bv}-\alpha_{cv}} \in E(G_0)$. By Proposition \ref{prop:lem1}, we may switch vertices to make a copy of $K_4$ with gain 0 on each edge.
\end{proof}

Let $K$ denote the copy of $K_4$ containing $v$. 
As in the previous proof, we suppose first that $v$ belongs to a $K_4$ whose vertices induce no additional edge. Since we cannot apply a $K_4$-contraction, there are vertices $a,b \in K$ and a vertex $x \notin K$ such that $(a,x)_{\alpha},(b,x)_{\alpha} \in E(G_0)$.
To see this, note that any unbalanced subgraph containing $K$ is still unbalanced as a subgraph of the contracted graph.
Now let $c$ be the final vertex in $K$.

\begin{claim}\label{clm:2214cycle}
$(c,x) \in E(G_0)$.
\end{claim}

\begin{proof}
Suppose $(c,x) \notin E(G_0)$. By applying Proposition \ref{prop:lem1} to $x$ (if necessary) we may assume that the 4-cycle $C$ induced by $v,a,b,x$ has label 0 on each edge. Apply 4-cycle contraction to $C$ merging $v$ and $x$. It is routine, as in the previous proofs, to check that the counts hold for unbalanced subgraphs. The sparsity conditions for balanced subgraphs hold since every edge of $C$ (and $K$) has gain 0.
\end{proof}

We have shown that the subgraph induced by $K$ and $x$ is $(2,1)$-tight. When the vertices of $K$ induce an additional edge we have a $K_4+e$.
Finally, we may show, exactly as in the proof of the previous theorem, that $G_0$ contains a bridge. Note that this time the copies of $K_5-e$ must have edges with non-zero gain since $G_0$ is $(2,2,1)$-gain-tight so the case when $G_0=Y_i=K_5-e$ cannot happen (an inverse H2a move would have been possible by the first claim).
\end{proof}

We will briefly discuss extensions to $(k,l,m)$-gain-tight graphs for other triples in the final section.


\section{Operations on Frameworks Supported on Surfaces} \label{sec:geom}

We now consider the geometric question of how these inductive operations behave as operations on frameworks rather than on graphs. We pursue this by a combination of limiting arguments and matrix techniques using the  orbit-surface rigidity matrix and making extensive use of results in \cite{nop,nop1,W2}.

\subsection{Henneberg moves}

Our first lemma is simple linear algebra.

\begin{lem}\label{lem:hen1SM}
Let $\M\in \{\S,\Y,\C\}$ and let $S$ be any possible point group. Let $(G,p)$ be an $S$-regular $S$-isostatic framework in $\mathscr{R}^{\M}_{(G,S,\theta)}$ with quotient $S$-gain graph $(G_0,\psi)$. Let $(G_0',\psi')$ be formed from $(G_0,\psi)$ by a Henneberg 1 move and let $G'$ be the corresponding covering graph. Then any $S$-regular realisation of $G'$ is $S$-isostatic.
\end{lem}

\begin{proof}
Let $(G',p')$ be an $S$-regular realisation of $G'$.
Then simply note that by the block structure of $O_\M(G',p',S)$ and the regularity of $(G',p')$, we have
 $$\rank O_\M(G',p',S)= \rank O_\M(G,p,S)+3.$$
\end{proof}

Our second lemma is more involved but by utilising the proof technique of \cite[Lemma $4.2$]{nop} we can still argue for any group and surface simultaneously.

\begin{lem}\label{lem:hen2SM}
Let $\M\in \{\S,\Y,\C\}$ and let $S$ be any possible point group. Let
 $(G,p)$ be an $S$-generic $S$-isostatic framework in $\mathscr{R}^{\M}_{(G,S,\theta)}$ with quotient $S$-gain graph $(G_0,\psi)$. Let $(G_0',\psi')$ be formed from $(G_0,\psi)$ by a Henneberg 2 move and let $G'$ be the corresponding covering graph. Then any $S$-generic realisation of $G'$ is $S$-isostatic.
\end{lem}

\begin{proof}
Let $S=\{x_1=id,x_2,\ldots, x_{|S|}\}$. For an edge $e_0=(1,2)$ (with gain $\alpha \in S$) in $G_0$ we apply the argument in \cite[Lemma $4.2$]{nop} simultaneously to each edge in the edge orbit $c^{-1}(e_0)$, where $c:G\to G_0$ is the covering map. (Note that if $e_0$ is a loop, then the proof follows analogously.) Suppose $V(G_0')=V(G_0)\cup\{0\}$ and let $p'=(p_{x_1(0)}, p_{x_2(0)},\ldots, p_{x_{|S|}(0)},p)$, where $(G',p')$ is $S$-generic. 
Suppose that $(G',p')$ is not $S$-symmetric infinitesimally rigid. Then it follows that every specialised framework in $\mathscr{R}^{\M}_{(G',S,\theta)}$ is $S$-symmetric infinitesimally flexible. 
Let $x_i(a)$ and $x_i(b)$ be orthogonal tangent vectors at $p_{x_i\circ\alpha(2)}$ where $x_i(b)$ is orthogonal to $p_{x_i\circ\alpha(2)}-p_{x_i(1)}$. 
Consider a sequence of specialisations $(G', p^k)$ in which only the joints $p_{x_1(0)},\ldots, p_{x_{|S|}(0)}$ are specialised to the joints $p_{x_1(0)}^k,\ldots, p_{x_{|S|}(0)}^k$, respectively,  
 and for each $i=1,\ldots, |S|$, $p_{x_i(0)}^k$ tends to $p_{x_i\circ\alpha(2)}$ in the direction $x_i(a)$.
 See Figure \ref{fig:contract}.
 More precisely, the normalised vector $(p_{x_i\circ\alpha(2)}-p_{x_i(0)}^k)/\|p_{x_i\circ\alpha(2)}-p_{x_i(0)}^k\|$ converges to $x_i(a)$, as $k\to \infty$.

  Each of the $S$-symmetric frameworks  $(G', p^k)$ has a unit norm non-trivial $S$-symmetric infinitesimal motion $u^k$ which is orthogonal to the space of $S$-symmetric trivial infinitesimal motions of its framework, $(G',p^k)$. By the Bolzano-Weierstrass theorem there is a subsequence of the sequence $u^k$ which converges to a vector, $u^\infty$ say, of unit norm. Discarding framework points and relabeling we may assume this holds for the original sequence. The $S$-symmetric limit motion of the degenerate $S$-symmetric framework $(G',p^\infty)$ is denoted by
 $u^\infty=(u_{x_1(0)}^\infty, u_{x_2(0)}^\infty,\ldots, u_{x_{|S|}(0)}^\infty,u)$. Also, we have $p^\infty =(p_{x_1\circ \alpha(2)}, p_{x_2\circ \alpha(2)},\ldots, p_{x_{|S|}\circ \alpha(2)},p)$.

We claim that the velocities $u_{x_i(1)}, u_{x_i\circ \alpha(2)}$ give an $S$-symmetric infinitesimal motion of the framework on $\M$ consisting of the bars joining $p_{x_i(1)}$ and $p_{x_i\circ\alpha(2)}$, $i=1,\ldots, |S|$. To see this note that in view of the well-behaved convergence of  $p_{x_i(0)}^k$ to $p_{x_i\circ\alpha(2)}$ (in the $x_i(a)$ direction) it follows that the velocities
 $u_{x_i\circ\alpha(2)}$ and $u_{x_i(0)}^\infty$ have the same component in the $x_i(a)$ direction, and so $(u_{x_i\circ\alpha(2)} -u_{x_i(0)}^\infty)\cdot x_i(a)=0$. Since $u_{x_i\circ\alpha(2)} -u_{x_i(0)}^\infty$ is tangential to $\M$ it follows from the choice of $x_i(a)$ that $u_{x_i\circ\alpha(2)} -u_{x_i(0)}^\infty$ is orthogonal to $p_{x_i\circ\alpha(2)}-p_{x_i(1)}$. On the other hand $u_{x_i(1)}-u_{x_i(0)}^\infty$ is orthogonal to $p_{x_i\circ\alpha(2)}-p_{x_i(1)}$ and so taking differences $u_{x_i\circ\alpha(2)} -u_{x_i(1)}$  is orthogonal to $p_{x_2\circ \alpha(2)}-p_{x_i(1)}$, as desired.

It now follows, by the symmetry-forced rigidity of $(G,p)$ and (hence) the symmetry-forced rigidity of the degenerate framework $(G',p^\infty)$, that the restriction motion $u^\infty_{\rm res}=u$,
and hence $u^\infty$ itself, is a trivial $S$-symmetric infinitesimal motion. This is a contradiction since the motion has unit norm and is orthogonal to the space of trivial $S$-symmetric infinitesimal motions.
\end{proof}

\begin{center}
\begin{figure}[ht]
\centering
\begin{tikzpicture}
\filldraw (0,2) circle (.5pt);

\filldraw (-3.7,4) circle (3pt)node[anchor=east]{$p_{x_1(0)}$};
\filldraw (3.7,4) circle (3pt)node[anchor=south]{$p_{x_3(0)}$};
\filldraw (0,-1.4) circle (3pt)node[anchor=west]{$p_{x_2(0)}$};

\filldraw (-4.5,4.5) circle  (3pt) node[anchor=east]{$p_{x_1(1)}$};
\filldraw (4.5,4.5) circle(3pt) node[anchor=west]{$p_{x_3(1)}$};
\filldraw (0,-2.3) circle (3pt) node[anchor=north]{$p_{x_2(1)}$};

\filldraw (-2,4.5) circle (3pt);
\filldraw (4.2,2.3) circle (3pt);
\filldraw (-2,-1.2) circle (3pt);

\filldraw (-2.2,2) circle (3pt)node[anchor=north]{$p_{x_1\circ\alpha(2)}$};
\filldraw (1.5,3.7) circle (3pt)node[anchor=east]{$p_{x_3\circ\alpha(2)}$};
\filldraw (1.3,.6) circle (3pt)node[anchor=west]{$p_{x_2\circ\alpha(2)}$};

\draw[black,thick]
(-2,4.5) -- (-3.7,4) -- (-4.5,4.5);

\draw[black,thick]
(-3.7,4) -- (-2.2,2);

\draw[black,thick]
(4.5,4.5) -- (3.7,4) -- (4.2,2.3);

\draw[black,thick]
(3.7,4) -- (1.5,3.7);

\draw[black,thick]
(0,-2.3) -- (0,-1.4) -- (-2,-1.2);

\draw[black,thick]
(0,-1.4) -- (1.3,.6);

\draw[dashed] plot[smooth, tension=1] coordinates{(-3.7,4) (-3.5,3) (-2.2,2)};
\draw[dashed] plot[smooth, tension=1] coordinates{(3.7,4) (2.5,4.3) (1.5,3.7)};
\draw[dashed] plot[smooth, tension=1] coordinates{(0,-1.4) (1,-.6) (1.3,.6)};

\draw[thick]
(-3.8,3.1) -- (-3.5,3) -- (-3.4,3.3);

\draw[thick]
(2.7,4.1) -- (2.5,4.3) -- (2.7,4.5);

\draw[thick]
(1.1,-.85) -- (1,-.6) -- (.7,-.7);
\end{tikzpicture}
\caption{An illustration of the specialisation in the case $S=\C_3$.}
\label{fig:contract}
\end{figure}
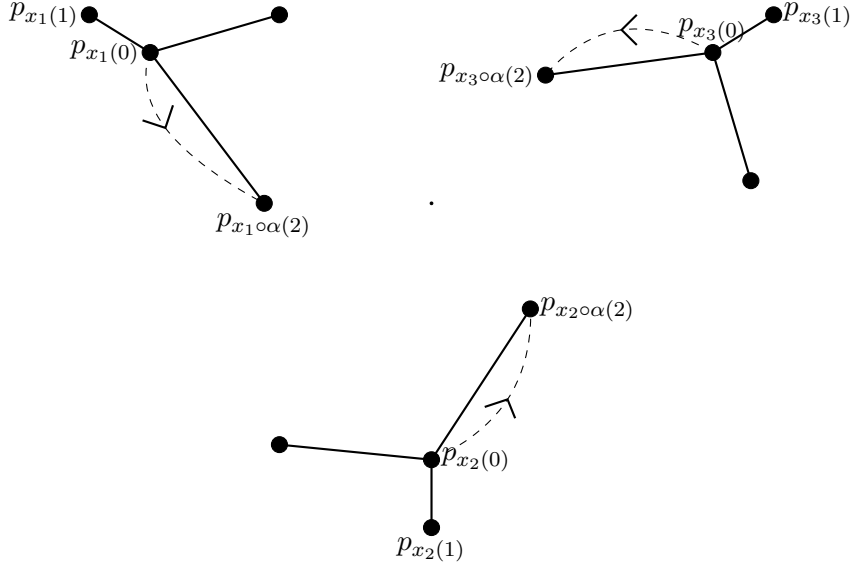
\end{center}

\subsection{Vertex Surgery moves}

We now consider the vertex-to-$K_4$ move. For rotational symmetry on the cylinder we present a direct matrix argument. 

\begin{lem}\label{lem:vtok4cylinderrotation}
Let  $S$ be $\C_m$ with the $z$-axis as the rotational axis.
Let $(G,p)$ be an $S$-generic $S$-isostatic framework in $\mathscr{R}^{\Y}_{(G,S,\theta)}$ with quotient $S$-gain graph $(G_0,\psi)$ and let $(G_0',\psi')$ be formed from $(G_0,\psi)$ by a vertex-to-$K_4$ move. Then any $S$-generic realisation of the covering graph $G'$ of $(G_0',\psi')$ is $S$-isostatic.
\end{lem}

\begin{proof}
The proof is similar to \cite[Lemma $5.2$]{nop1} with $K_4$ replacing $H$.
Let $n$ = $|V(G_0')|$. Let $v_*$ be a fixed vertex of $K_4$.
Let $(G',p')$ be $S$-generic.
Consider the orbit-surface rigidity matrix $O_{\Y}(G',p',S)$ with column triples in the order
of  $v_1, v_2 ,v_{3}, v_*, v_{5}, \dots , v_n$
where $v_1, v_2,v_3 ,v_4=v_*$ are the vertices of $K_4$.
Order the rows of $O_{\Y}(G',p',S)$ in the order of the edges $e_1, \dots ,e_6$
for $K_4$
followed by the $n$ rows of the block diagonal matrix whose
diagonal entries are the respective normal vectors
 to $\Y$ at $p'_1,\ldots, p'_n$,
followed by the remaining rows for the edges of $E(G_0')\setminus E(K_4)$.
Note that the submatrix formed by the first  $10$ rows  is the
$1$ by $2$ block matrix 
\[ \begin{bmatrix}
O_{\Y}(c^{-1}(K_4),p',S) &0 \end{bmatrix},\]
where $c:G'\to G_0'$ is the covering map.
Suppose, by way of contradiction, that $G_0'$
is not $S$-isostatic. Then, there is a vector $u$ in the kernel
of $O_{\Y}(G',p',S)$ which corresponds to an $S$-symmetric non-trivial infinitesimal motion.

\begin{claim}\label{claim:k4} W.l.o.g. we may assume that $u_{K_4}=0$.
\end{claim}

\begin{proof} We will identify an $S$-symmetric infinitesimal motion of a framework $(G,p)$ with its restriction to the  vertices of the quotient $S$-gain graph $G_0$ of $G$. 

By adding to $u$ some trivial $S$-symmetric infinitesimal motion, we may assume that $u_4=0$.
Write $u=(u_{K_4}, u_{G_0'\setminus K_4})$ where $u_{K_4} = (u_1^x,u_1^y, u_1^z, \dots ,u_4^x,u_4^y, u_4^z)$.
The matrix $O_{\Y}(G',p',S)$ has the block form
\[
O_{\Y}(G',p',S)=\begin{bmatrix}
O_{\Y}(c^{-1}(K_4),p',S) & 0\\X_1 & X_2
\end{bmatrix}
\]
where $X =[X_1\, X_2]$ is the matrix formed by the last $|E(G')|- 6+n-4$
rows. Since $K_4$ is $S$-isostatic on $\Y$ and $O_{\Y}(c^{-1}(K_4),p',S)u_{K_4}=0$ it follows that
$u_{K_4}$ is an $S$-symmetric trivial infinitesimal motion. But  $u_4=0$ and so  $u_{K_4} =0$.

\end{proof}

Consider now the framework vector $\tilde{p}=(p_4,p_4,p_4,p_4,p_{5},\dots,p_n)$
in which the first $4$ framework joints are specialised to $p_4$ and let $p_*=(p_4,p_{5},\dots,p_n)$
be the restricted vector with associated $S$-generic framework $(G,p_*)$. By the hypotheses, this framework is $S$-symmetric infinitesimally rigid.

Note that the matrix $X_2=X_2(p')$ is square. Moreover, it has the vector $u_{G'\setminus K_4}$ in the kernel, and $u_{G'\setminus K_4}$ is clearly non-zero (since $u$ is non-zero and $u_{K_4}=0$ by the claim  above). 

Thus the determinant as a polynomial in the coordinates of the $p_i'$ vanishes identically because of genericity. It follows that $\det X_2(\tilde{p})$ vanishes identically and that there is a nonzero vector, $w_{G_0'\setminus K_4}$ say, in the kernel. But now we obtain the contradiction
\[ \begin{bmatrix}O_{\Y}(G,p_*,S) \end{bmatrix}\begin{bmatrix} \mathbf{0} \\ w_{G_0'\setminus K_4}\end{bmatrix}=\begin{bmatrix} a & \mathbf{0}\\ *& 
 X_2(\tilde{p})
\end{bmatrix}\begin{bmatrix} \mathbf{0}\\ w_{G_0'\setminus K_4} \end{bmatrix}=\mathbf{0},  \]
where $a=(0,1,0)$. 
\end{proof}




For  combinations of surfaces and symmetry groups with $k_S=1$ the above argument breaks down. Instead we adapt a limiting argument from \cite[Section 5.2]{nop}.

Let $\M\in\{\Y,\C\}$ and let $S$ be chosen so that $k_S=1$.
Take the complete graph on four vertices with identity gain on each edge as an $S$-gain graph and denote the covering graph by $H$. For $j=1,2,\dots,|S|$ and $i=1,2,3,4$, let $p_{x_j(i)}$  be the joints for some $S$-symmetric realisation of $H$. Further, let $p|_j=(p_{x_j(1)},p_{x_j(2)},p_{x_j(3)}, p_{x_{j}(4)})$ for $j=1,\dots, |S|$, and let $(H,p)=(H,(p|_1,p|_2,\dots, p|_{|S|}))$.

Since $k<3$, $\M$ has distinct principal curvatures $\kappa_{s_j},\kappa_{t_j}$ with associated orthonormal vectors $\hat{s}_j,\hat{t}_j$ in the tangent plane at $p_{x_j(1)}$ for $j=1,\dots ,|S|$. Recall that, in a neighbourhood $U$ of $p_{x_j(1)}$ on $\M$, we can use Taylor's theorem to express $p|_j$ as a function of $s_j$ and $t_j$: 
\[ p|_j(s_j,t_j)=p_{x_j(1)}+(s_j\hat{s}_j+t_j\hat{t}_j)+\frac{1}{2}(\kappa_{s_j}s_j^2+\kappa_{t_j}t_j^2)\hat{n}_j+r(s_j,t_j) \]
where $\hat{n}_j$ is the unit normal at $p_{x_j(1)}$, the vectors $\hat{s}_j,\hat{t}_j,\hat{n}_j$ form a right-handed orthonormal triple and $r(s_j,t_j)$ denotes the higher-order terms.    Let $p_{x_j(i)}=p|_j(s_{x_j(i)},t_{x_j(i)})$, so $p_{x_j(2)},p_{x_j(3)},p_{x_j(4)}\in U$, and let $p_{x_j(i)}^{\epsilon_k}=p|_j(\epsilon_ks_{x_j(i)},\epsilon_kt_{x_j(i)})$ for $i=2,3,4$. 
We say that a sequence of frameworks $(H,p^{j,\epsilon_k})$ with 
$$p^{j,\epsilon_k}=(p|_1, \dots, p|_{j-1}, p_{x_j(1)},p_{x_j(2)}^{\epsilon_k},p_{x_j(3)}^{\epsilon_k}, p_{x_j(4)}^{\epsilon_k}, p|_{j+1}, \dots, p|_{|S|})$$
where $\epsilon_k\rightarrow 0$ as $k\rightarrow \infty$ is a \emph{well-behaved $K_4$ contraction} if the local coordinates $s_{x_j(2)},t_{x_j(2)},s_{x_j(3)}, \break t_{x_j(3)},  s_{x_j(4)},t_{x_j(4)}$ satisfy
\[ det \begin{pmatrix}s_{x_j(2)}&t_{x_j(2)}&s_{x_j(2)}t_{x_j(2)}\\ s_{x_j(3)} & t_{x_j(3)} & s_{x_j(3)}t_{x_j(3)}\\ s_{x_j(4)} & t_{x_j(4)} & s_{x_j(4)}t_{x_j(4)}\end{pmatrix}\neq 0. \]
Note that, by symmetry, $(H,p^{j,\epsilon_k})$  is a well-behaved $K_4$ contraction for some $j\in\{1,\dots |S|\}$ if and only if it is for all $j=1,\ldots,|S|$.

\begin{lem}\label{lem:NOP}
Let $\M\in\{\Y,\C\}$ and let $S$ be chosen so that $k_S=1$. Take the complete graph on four vertices with identity gain on each edge as an $S$-gain graph and denote the covering graph by $H$. Let $p_{x_j(i)}$, for $j=1,2,\dots,|S|$ and for $i=1,2,3,4$, be the joints for some $S$-symmetric realisation of $H$. Let $$p^{\epsilon_k} = (p_{x_1(1)},p_{x_1(2)}^{\epsilon_k},p_{x_1(3)}^{\epsilon_k}, p_{x_{|1|}(4)}^{\epsilon_k},p_{x_2(1)},p_{x_2(2)}^{\epsilon_k},p_{x_2(3)}^{\epsilon_k}, p_{x_2(4)}^{\epsilon_k}, \dots, p_{x_{|S|}(1)},p_{x_{|S|}(2)}^{\epsilon_k},p_{x_{|S|}(3)}^{\epsilon_k}, p_{x_{|S|}(4)}^{\epsilon_k}),$$
and let $(H,p^{\epsilon_k})$,
 for $k=1,2,\dots$, be a well-behaved contraction of $S$-symmetric frameworks on $\M$. Further, let $u^k, k=1,2,\dots$, be an associated sequence of $S$-symmetric infinitesimal motions which forms a convergent sequence in $\mathbb{R}^{12|S|}$. Then the limit vector has the form $$(u_{{x_1}(1)},u_{x_1(1)},u_{x_1(1)},u_{x_1(1)},u_{x_2(1)},u_{x_2(1)},u_{x_2(1)},u_{x_2(1)},\dots, u_{x_{|S|}(1)},u_{x_{|S|}(1)},u_{x_{|S|}(1)},u_{x_{|S|}(1)}).$$
\end{lem}

\begin{proof}
Denote the copies of $K_4$ in $H$ as $K^1, \ldots , K^{|S|}$, and let $(K^j, p|_j)$ be the corresponding realisations induced by $(H,p)$. Let $u=u^1$ and
denote the $S$-symmetric infinitesimal motion $$u=(u_{{x_1}(1)},u_{x_1(2)},u_{x_1(3)},u_{x_1(4)},\dots, u_{x_{|S|}(1)},u_{x_{|S|}(2)},u_{x_{|S|}(3)},u_{x_{|S|}(4)}).$$

By focusing on the restricted motion $u|_{K^j}$ and the subsequence $u^k|_{K^j}$, associated with $(K^j, p|_j)$, we can follow the proof of \cite[Lemma 5.4]{nop} with only trivial modifications, to establish that 
the well-behaved contraction sequence  $(H, p^{\epsilon_k})$   takes the infinitesimal motions $u^k|_{K^j}$ to $(u_{x_j(1)},u_{x_j(1)},u_{x_j(1)},u_{x_j(1)})$. 
Since that Lemma was stated for surfaces of type $k=1$ and our lemma includes the cylinder, with symmetry groups of type $k_S=1$, we include the details in the appendix.
%
Since $u|_{K^{\ell}}$ is the image of $u|_{K^j}$ under some $x_t\in S$, we may simultaneously use the argument from the appendix to see that $u^k|_{K^\ell}$ is taken to $(u_{x_\ell(1)},u_{x_\ell(1)},u_{x_\ell(1)},u_{x_\ell(1)})$ for each $\ell$, where $u_{x_\ell(1)}$ is the image of $u_{x_j(1)}$ under $x_t\in S$.  The contracted framework is again an $S$-symmetric framework on $\M$ and the limit vector of the sequence of $S$-symmetric infinitesimal motions has  the required form.   
\end{proof}

%

%
%

\begin{lem}\label{lem:vtok4type1}
Let $G_0'\rightarrow G_0$ be a vertex-to-$K_4$ move applied to an $S$-gain graph $G_0'$ with the new $K_4$ having all identity gains and $V(G_0)=\{v_1,v_2,\dots,v_n\}$ where $v_1,v_2,v_3,v_4$ are the vertices of the new $K_4$. Let $G'$ and $G$ be the covering graphs corresponding to $G_0'$ and $G_0$, and let $(G',p')$ and $(G,p)$ be the corresponding  $S$-generic frameworks on $\M$ with $\M\in\{\Y,\C\}$ and $S$ chosen so that $\M$ has symmetric type $k_S=1$. Suppose that $(G',p')$ is $S$-isostatic. Then $(G,p)$ is also $S$-isostatic.
\end{lem}

\begin{proof}
Let $p_i=p(v_i)$ for $i=1,2,3,4$ and let $p_{x_j(i)}$ for $j=1,2,\dots,|S|$ be the joints of $(G,p)$ corresponding to the orbit of $p_i$ (and so for each $j$, $v_{x_j(i)}$ is the corresponding vertex of $G$). Suppose that $(G,p)$ is not $S$-isostatic. Since $(G',p')$ is $S$-isostatic, $G_0'$ is $(2,k,k_S)$-gain-tight so $G_0$ is also $(2,k,k_S)$-gain-tight. It follows that $(G,p)$ is $S$-dependent with a non-trivial $S$-symmetric infinitesimal motion. Let $p^{\epsilon_k}$ be a sequence of $S$-generic framework vectors for $(G,p)$ with $p_{x_j(i)}^{\epsilon_k}\rightarrow p_{x_j(1)}$, as $k\rightarrow \infty$, for $i=2,3,4$ and for $1\leq j \leq |S|$, where the $K_4$ contraction is well-behaved (for all $j=1,\ldots, |S|$). Since $(G,p^{\epsilon_k})$ is $S$-generic  and $\M$ has symmetric type $k_S=1$, there is exactly 1 trivial $S$-symmetric infinitesimal motion of $(G,p^{\epsilon_k})$ for each $k$. Moreover, since $(G,p)$ is $S$-dependent there is a non-trivial $S$-symmetric infinitesimal motion of $(G,p^{\epsilon_k})$ for each $k$. 

Hence, for all $k$, there exists an $S$-symmetric infinitesimal motion $u^k=(u_{x_{1}(1)}^k,\dots,u_n^k)$ of $(G,p^{\epsilon_k})$ such that $u^k$ is in the tangent space of $p^{\epsilon_k}$, $u^k$ has norm 1 and $u^k$ is orthogonal to the trivial $S$-symmetric infinitesimal motion. By passing to a subsequence, if required, $u^k$ converges to an $S$-symmetric infinitesimal motion $u$ of the limit framework $(G,q)$ as $k\rightarrow \infty$. It now follows from Lemma \ref{lem:NOP}  that $u$ has the form
 $$(u_{x_1(1)},u_{x_1(1)},u_{x_1(1)},u_{x_1(1)},u_{x_2(1)},u_{x_2(1)},u_{x_2(1)},u_{x_2(1)},\dots,u_{x_{|S|}(1)},u_{x_{|S|}(1)},u_{x_{|S|}(1)},u_{x_{|S|}(1)},u_r,\dots,u_n).$$ Moreover, $u$ is in the tangent space of $q$, $u$ has norm 1 and $u$  is orthogonal to the trivial $S$-symmetric infinitesimal motion. 
This implies that the $S$-generic framework $(G',q')$ (where $q'$ arises from $q$ by restricting $V(G)$ to $V(G')$)  has an $S$-symmetric infinitesimal motion $$u^-=(u_{x_{1}(1)},u_{x_{2}(1)},\dots,u_{x_{|S|}(1)},u_r,\dots,u_n)$$ which is orthogonal to the unique trivial $S$-symmetric infinitesimal motion of $(G',q')$, contradicting the hypothesis that $(G',p')$ was $S$-isostatic.

\end{proof}

\begin{lem}\label{lem:vto4cycleCC}
Let $\M \in \{\Y,\C\}$ and 
\begin{itemize}
\item if $\M=\Y$ let $S$ be $\C_m$ (with the $z$-axis as the rotational axis), $\C_s$ (with mirror orthogonal to the $z$-axis or a plane containing the $z$-axis) or $C_i$ and
\item if $\M=\C$ let $S$ be $\C_m$ (with the $z$-axis as the rotational axis), $\C_s$ (with mirror orthogonal to the $z$-axis), $\C_i$, $\C_{mh}$ (with the $z$-axis as the rotational axis) or $\S_{2m}$ (with the $z$-axis as the rotational axis).
\end{itemize}
Let $(G,p)$ be an $S$-generic $S$-isostatic framework in $\mathscr{R}^{\M}_{(G,S,\theta)}$ with quotient $S$-gain graph $(G_0,\psi)$ and let $(G_0',\psi')$ be formed from $(G_0,\psi)$ by a vertex-to-4-cycle move. Then any $S$-generic realisation of the covering graph $G'$ of $(G_0',\psi')$ is $S$-isostatic.
\end{lem}

The proof is similar to \cite[Proposition 1]{W2} with minor modifications due to the different definition of a stress in the surface context (recall Definition \ref{defn:stress}) and to using the orbit-surface rigidity matrix rather than the rigidity matrix of a $3$-frame.

\begin{proof}
Choose a vertex $v_1$ of $G_0$ to be split and  suppose $v_2,v_3,\dots,v_k$ are the neighbors of $v_1$. Without loss of generality, we may assume that all edges joining $v_1$ and $v_i$, $i=1,\ldots, k$, are directed away from $v_1$ and that the edge $(v_1,v_i)$ has gain $\alpha_i$. (If the edge $(v_1,v_i)$ appears $l>1$ times, then we denote the corresponding edge gains by $\alpha_i=\alpha_{i1}, \ldots,\alpha_{il}$, for a fixed numbering of the edges; in this case, the gains $\alpha_{i1},\ldots, \alpha_{il}$ are of course all distinct.) Let $G_0^*$  be the $S$-gain graph obtained from $G_0$ by adding
a new vertex $v_0$ and two edges $(v_0,v_2)$ and $(v_0,v_3)$ with respective gains $\alpha_2$ and $\alpha_3$. The covering graph of $G_0^*$ is denoted by $G^*$. Further, we let $(G^*,p^*)$ be the framework obtained from $(G,p)$ by setting 
 $p^*(v_0)=p(v_1)$ and $p^*(v_i)=p(v_i)$ for all other vertices $v_i$, $i\neq 0$, of $G^*$.

Consider the $S$-gain graph $G_0'$ which is obtained from $G_0^*$ by swapping some number of edges $(v_1,v_j)$ for $j\in \{4,5,\dots\}$ to edges $(v_0,v_j)$ (keeping the same gains). (If there are multiple edges joining $v_1$ with $v_2$ and $v_3$, then we may also swap edges of the form $(v_1,v_2)$ and $(v_1,v_3)$ to edges $(v_0,v_2)$ and $(v_0,v_3)$, provided that their gains are not equal to $\alpha_{2}$ or $\alpha_{3}$.) Let $G'$ be the covering graph of $G_0'$.   

Let $(\omega,\lambda)$ be an $S$-symmetric self-stress  on $(G',p^*)$. In the following, we will identify an $S$-symmetric self-stress of a framework with its restriction to the edges and vertices of the corresponding quotient $S$-gain graph. 
Let $\omega_{ij}=\omega(\{v_i,\alpha_j(v_j)\})$ for $i=0,1$ and $j=2,3$, and let $\lambda_i=\lambda(v_i)$. Then note that if we  set $\tilde{\omega}_{12}=\omega_{02}+\omega_{12}$, $\tilde{\omega}_{13}=\omega_{03}+\omega_{13}$,  $\tilde{\lambda}_1=\lambda_0+\lambda_1$, and $\tilde{\omega}(e)=\omega(e)$ for all other edges $e$ and $\tilde{\lambda}_i=\lambda_i$ for all other vertices $v_i$, then $(\tilde{\omega}, \tilde{\lambda})$ is an $S$-symmetric self-stress of $(G,p)$.  Thus, since $(G,p)$ is $S$-isostatic, we must have $\tilde{\omega}_{1j}=0$ for each $j$ and $\tilde{\lambda}_i=0$ for each $i$.

It follows that around $p_1$ we have $\omega_{12}(p_1-p_{\alpha_2(2)})+\omega_{13}(p_1-p_{\alpha_3(3)})+\lambda_1 s_1=0$. Since $(G,p)$ is $S$-generic, and $p_1-p_{\alpha_2(2)}$, $p_1-p_{\alpha_3(3)}$ and the normal $s_1$ to $\M$ at $p_1$ are not coplanar, we have $\omega_{12}=\omega_{13}=\lambda_1=0$. Similarly, we deduce that $\omega_{02}=\omega_{03}=\lambda_0=0$. Thus, the rows of $O_{\M}(G',p^*,S)$  are linearly independent. Now we perturb $(G',p^*)$ within a neighbourhood $B(p^*,\epsilon) \cap \M$, for sufficiently small $\epsilon$ to find an $S$-generic position $(G',p')$ which is guaranteed to be $S$-isostatic since $(G',p^*)$ is.
\end{proof}

\begin{lem}\label{lem:joinCC}
Let $\M \in \{\Y,\C\}$ and 
\begin{itemize}
\item if $\M=\Y$ let $S$ be $\C_m$ (with the $z$-axis as the rotational axis), $\C_s$ (with mirror orthogonal to the $z$-axis or a plane containing the $z$-axis) or $C_i$ and
\item if $\M=\C$ let $S$ be $\C_m$ (with the $z$-axis as the rotational axis), $\C_s$ (with mirror orthogonal to the $z$-axis), $\C_i$, $\C_{mh}$ (with the $z$-axis as the rotational axis) or $\S_{2m}$ (with the $z$-axis as the rotational axis).
\end{itemize}
Let $(G^1,p^1)$ and $(G^2,p^2)$ be two $S$-generic $S$-isostatic frameworks in $\mathscr{R}^{\Y}_{(G,S,\theta)}$ with quotient $S$-gain graphs $(G^1_0,\psi^1)$ and $(G^2_0,\psi^2)$.
Then an $S$-generic framework $(G^1\oplus G^2, p)$ corresponding to the edge join of $G^1_0$ and $G^2_0$ with joining edge given arbitrary gain is $S$-isostatic on $\M$.
\end{lem}

\begin{proof}
The orbit-surface matrices $O_{\M}(G^1,p^1,S)$ and $O_{\M}(G^2,p^2,S)$ have maximal rank. Hence the block matrix
\[ \begin{bmatrix}O_{\M}(G^1,p^1,S) & \mathbf{0}\\ \mathbf{0} & O_{\M}(G^2,p^2,S) \end{bmatrix} \]
has a 2-dimensional nullspace. The final joining edge, $S$-generically, eliminates the additional $S$-symmetric infinitesimal motion.
\end{proof}

Note that the final lemma clearly fails for settings where there are more or less than one isometry. 


\section{Laman type theorems} \label{sec:thms}

We have now put together enough results to prove our main theorems.
For convenience let us say that, for an $S$-gain graph $G_0$, $V(G_0)=\{1,\dots, n\}$ and $p_i=(x_i,y_i,z_i)$ for each $i$.

\begin{proof}[of Theorem \ref{thm:sphereinv}]
Theorem \ref{thm:maxwell} proves the necessity.

For the sufficiency we use induction and Theorem \ref{thm:233construction}. Let the edge of $K_2$ be $e=(1,2)$ with gain $\alpha\in \C_i$, and let $p_{\alpha(i)}=(x_i',y_i',z_i')$, $i=1,2$.
Then note that $O_{\S}(K_2,p,\C_i)$ is the $3\times 6$ matrix 
\[ \begin{bmatrix} x_1-x_2' & y_1-y_2'  & z_1-z_2' & x_2-x_1' & y_2-y_1' & z_2-z_1' \\ x_1 & y_1 & z_1 & 0 & 0 & 0\\ 0 & 0 & 0 & x_2 & y_2 & z_2 \end{bmatrix}\]
which is easily checked to have rank 3.
Theorem \ref{thm:233construction} gives us a short list of operations that generate all $(2,3,3)$-gain-tight graphs. For the inductive step suppose $(G,p)$ is $\C_i$-isostatic and suppose $G'$ is formed from $G$ by any one of these operations. Then Lemmas \ref{lem:hen1SM} and \ref{lem:hen2SM} confirm that any $\C_i$-generic realisation of $G'$ is $\C_i$-isostatic.
\end{proof}

\begin{proof}[of Theorem \ref{thm:spheredih}]
Theorem \ref{thm:maxwell} proves the necessity.

For the sufficiency we use induction and \cite[Theorem 4.4]{jkt}. 
Let $S\in \{\C_{mh},S_{2m}\}$ and let $K_1^*$ denote a loop at vertex $1$ with non-trivial gain $\alpha$ (in particular, note that the gain $\alpha$ cannot be an inversion). Further, let $p_{\alpha(1)}=(x_1',y_1',z_1')$ and $p_{\alpha^{-1}(1)}=(x_1'',y_1'',z_1'')$
Then note that $O_{\S}(K_1^*,p,S)$ is the $2\times 3$ matrix
\[ \begin{bmatrix} 2x_1-x_1'-x_1''& 2y_1-y_1'-y_1'' & 2z_1-z_1'-z_1''\\ x_1 & y_1 & z_1 \end{bmatrix}\]
which is easily checked to have rank 2 in each case.
\cite[Theorem 4.4]{jkt} gives us a short list of operations that generate all $(2,3,1)$-gain-tight graphs. For the inductive step suppose $(G,p)$ is $S$-isostatic and suppose $G'$ is formed from $G$ by any one of these operations. Then Lemmas \ref{lem:hen1SM} and \ref{lem:hen2SM} confirm that any $S$-generic realisation of $G'$ is $S$-isostatic.
\end{proof}

\begin{proof}[of Theorem \ref{thm:cylindercyclic}]
Theorem \ref{thm:maxwell} proves the necessity.

For the sufficiency we use induction and Theorem \ref{thm:222construction}. 
First note that $O_{\Y}(K_1,p,\C_m)$, where $K_1$ is the vertex $1$, is the $1\times 3$ matrix 
\[ \begin{bmatrix} x_1 & y_1 & 0 \end{bmatrix}\]
with rank 1.
Theorem \ref{thm:222construction} gives us a short list of operations that generate all $(2,2,2)$-gain-tight graphs. For the inductive step suppose $(G,p)$ is $\C_m$-isostatic and suppose $G'$ is formed from $G$ by any one of these operations. Then Lemmas \ref{lem:hen1SM}, \ref{lem:hen2SM}, \ref{lem:vtok4cylinderrotation} and \ref{lem:vto4cycleCC} confirm that any $\C_m$-generic realisation of $G'$ is $\C_m$-isostatic.
\end{proof}

\begin{proof}[of Theorem \ref{thm:cylinderreflection}]
Theorem \ref{thm:maxwell} proves the necessity.

For the sufficiency we use induction and Theorem \ref{thm:221construction}. 
Let $S\in \{\C_s,C_i\}$ and $K_1^*$ denote a loop at vertex $1$ with non-trivial gain $\alpha$. Further, let $p_{\alpha(1)}=(x_1',y_1',z_1')$. Clearly, $p_{\alpha(1)}=p_{\alpha^{-1}(1)}$, since $S$ is a group of order $2$.
Note that $O_{\Y}(K_1^*,p,S)$ is the $2\times 3$ matrix 
\[ \begin{bmatrix} 2(x_1-x_1') & 2(y_1-y_1') &2(z_1-z_1') \\ x_1 & y_1 & 0 \end{bmatrix}\]
which has rank 2 in each case.
Also $O_{\Y}(K_4+e,p,S)$ is a $11\times 12$ matrix 
which can easily be checked to have rank 11 for each choice of $S$.\footnote{Let $V(K_4+e)=\{v_1,v_2,v_3,v_4\}$, let $e=v_1v_2$ or let $e$ be a loop on $v_1$ and let $q(v_1)=(\frac{1}{\sqrt{2}},\frac{1}{\sqrt{2}},4),q(v_2)=(1,0,2),q(v_3)=(-1,0,-1),q(v_4)=\frac{1}{\sqrt{2}},\frac{1}{\sqrt{2}},-1)$. It is elementary to check that $\rank O_{\Y}(K_4+e,q,S)=11$ for each group (in the case of the vertical mirror we used the plane through the point $(0,1,0)$). It follows that the generic rank is also 11.}
Theorem \ref{thm:221construction} gives us a short list of operations that generate all $(2,2,1)$-gain-tight graphs. For the inductive step suppose $(G,p)$ is $S$-isostatic and suppose $G'$ is formed from $G$ by any one of these operations. Then Lemmas \ref{lem:hen1SM}, \ref{lem:hen2SM}, \ref{lem:vtok4type1}, \ref{lem:vto4cycleCC} and \ref{lem:joinCC} confirm that any $S$-generic realisation of $G'$ is $S$-isostatic.
\end{proof}

\begin{proof}[of Theorem \ref{thm:conecyclic}]
Theorem \ref{thm:maxwell} proves the necessity.

For the sufficiency we use induction and Theorem \ref{thm:211construction}. 
Let $S\in \{\C_m,\C_s, \C_i,\C_{mh},S_{2m}\}$ and $K_1^*$ denote a loop at vertex $1$ with non-trivial gain $\alpha$. Further, let $p_{\alpha(1)}=(x_1',y_1',z_1')$ and $p_{\alpha^{-1}(1)}=(x_1'',y_1'',z_1'')$.
Note that $O_{\C}(K_1^*,p,S)$ is a $2\times 3$ matrix 
\[ \begin{bmatrix} 2x_1-x_1'-x_1''& 2y_1-y_1'-y_1'' & 2z_1-z_1'-z_1'' \\ x_1 & y_1 & -z_1 \end{bmatrix}\]
which has rank 2 for each choice of $S$.
Also  $O_{\C}(K_4+f,p,S)$ is a $11 \times 12$ matrix and $O_{\C}(K_5-e,p,S)$ is a $14\times 15$ matrix 
which can easily be checked to have rank 11 and 14 for each respective choice of $S$.\footnote{Let $q(v_1)=(1,0,1),q(v_2)=(3,0,-3),q(v_3)=(1,1,\sqrt{2}),q(v_4)=(-2,2,2\sqrt{2}),q(v_5)=(\sqrt{2},-\sqrt{2},-2)$. Let $V(K_4+f)=\{v_1,v_2,v_3,v_4\}$ with $f=v_1v_3$ or with $f$ being a loop on $v_1$ and let $V(K_5-e)=\{v_1,v_2,v_3,v_4,v_5\}$ with $e=v_4v_5$. Then $\rank O_{\C}(K_4+f,q,S)=11$ and $\rank O_{\C}(K_5-e,q,S)=14$ for each group. Since $p$ is generic, it follows that $\rank O_{\C}(K_4+f,p,S)=11$ and $\rank O_{\C}(K_5-e,p,S)=14$. }

Theorem \ref{thm:211construction} gives us a short list of operations that generate all $(2,1,1)$-gain-tight graphs. For the inductive step suppose $(G,p)$ is $S$-isostatic and suppose $G'$ is formed from $G$ by any one of these operations. Then Lemmas \ref{lem:hen1SM}, \ref{lem:hen2SM}, \ref{lem:vtok4type1}, \ref{lem:vto4cycleCC} and \ref{lem:joinCC} confirm that any $S$-generic realisation of $G'$ is $S$-isostatic.
\end{proof}


\section{Further Work} \label{sec:furtherwork}

We finish by outlining a number of avenues of further developments. We start with a slight diversion into matroid theory.

\subsection{Matroids and inductive constructions} \label{subsec:mic}

Let $\mathcal{I}_{k,\ell,m}$ be the family of $(k,\ell,m)$-gain-sparse edge sets in $(H,\psi)$.
Then, as noted in \cite{jkt,schtan}, $\mathcal{I}_{k,\ell,m}$ forms the family of independent sets
of a matroid on $E(H)$ for certain $(k,\ell, m)$. Let $M(k,\ell,m):=(E(H),\I_{k,\ell,m})$ whether or not the triple $k,\ell,m$ induces a matroid.
Using 3 basic matroids as `building blocks' we can use standard matroid techniques to see that $M(k,\ell,m)$ is a matroid for a large range of triples $k,l,m \in \bN$.

First note that when $\ell=m$, $(k,\ell,m)$-gain-sparsity is exactly $(k,\ell)$-sparsity (on a multi-graph) and $M(k,\ell)$ is known to be a matroid for all $0\leq \ell <2k$ \cite{W1}. We also assume in this paper that $m \leq \ell$.

Our 3 basic matroids are the \emph{frame matroid} $M(1,1,0)$ \cite{zas1}, the \emph{cycle matroid} $M(1,1,1)$ and the \emph{bicircular matroid} $M(1,0,0)$. Then the previous sentence tells us that each possible option for $M(1,\ell,m)$ is a matroid. For $k>1$, using matroid union and Dilworth truncation we know the following:

\begin{enumerate}
\item $M(k,k+t,m+t)$ is a matroid for $0\leq t <k$ and $m\leq k$ (take $m$ copies of the cycle matroid and $k-m$ copies of the frame matroid and then apply $t$ Dilworth truncations),
\item $M(k,k-m+n+t,n+t)$ with $0 \leq n \leq m \leq k$ and $n+t<k+m$ is a matroid (take $n$ copies of the cycle matroid, $m-n$ copies of the bicircular matroid and $k-m$ copies of the frame matroid and then apply $t$ Dilworth truncations).
\end{enumerate}

By the above remarks we know that: $M(2,3,3)$ is a matroid (see also Theorem \ref{thm:sphereinv}); $M(2,3,2)$ is a matroid (this count does not appear in symmetry-forced rigidity analyses of frameworks in the plane or on surfaces, as the surface must be a sphere and there does not exist a symmetry group with two fully symmetric rotations. However it does occur for periodic frameworks \cite{Ross} or in `anti-symmetric' rigidity analyses (see Section~\ref{subsec:incsym})); $M(2,3,1)$ is a matroid (see also Theorem \ref{thm:spherecyclic}); $M(2,3,0)$ is not a matroid in general \cite{schtan}; $M(2,2,2)$ is a matroid (see also Theorem \ref{thm:cylindercyclic}); $M(2,2,1)$ is a matroid (see also Theorem \ref{thm:cylinderreflection}); $M(2,2,0)$ is a matroid (see also Conjecture \ref{con:cylinderX} below); $M(2,1,1)$ is a matroid (see also Theorem \ref{thm:conecyclic}); $M(2,1,0)$ is a matroid (see also Conjecture \ref{con:coneX} below); $M(2,0,0)$ is a matroid (this count appears for frameworks on surfaces of type $k=0$ which we do not consider in this paper (recall Section~\ref{sec:fwsurf})). 

For $k\geq 3$ these observations still give us a lot of information; however they do not tell us anything about the case when $\ell - m >k$. We do not know if $M(2,3,0)$ is typical or atypical for such triples.

Returning to the subject of the paper, we note that combining the results of Section \ref{sec:recchar} with \cite[Theorem $4.4$]{jkt} and \cite[Theorem $4.7$]{Ross} gives inductive constructions for $(2,\ell,m)$-gain-tight graphs (sometimes only for particular groups) for all $1\leq m\leq \ell <4$. However, the case when $m=0$ is completely open. One indication of the potential difficulty to overcome here is that the minimum vertex degree in the graph may be 4. The analogue of the Henneberg moves for degree 4 vertices are known as X and V-replacement \cite{GSS,jkt,N&R,T&W}. V-replacement is known to not preserve $(2,\ell)$-sparsity. X-replacement has been used to some effect in \cite{jkt} so it is plausible that it could be used for the cases in question here. However, while X-replacement (as an operation on frameworks) is easy to understand in the plane (the generic argument is based on the simple fact that, generically, two lines intersect (see also \cite{jkt})), the question whether the corresponding operation  in 3-dimensions preserves generic rigidity is still open \cite{GSS,N&R} (two generic lines in 3D need not intersect!). This difficulty also arises for X-replacements on frameworks supported on surfaces since two lines will typically not intersect in a point on the surface. However, the X-replacement operation on frameworks supported on surfaces may still be more accessible than  the X-replacement operation in the general  3-dimensional case.

\subsection{The sphere}

As indicated in Section~\ref{sec:combchar}, it seems difficult to establish characterisations for symmetry-forced rigidity on the sphere for groups other than $\mathcal{C}_m$, $\mathcal{C}_s$, $\mathcal{C}_i$, $\mathcal{C}_{mh}$ and $\mathcal{S}_{2m}$, as there are no tangential isometries (i.e., rotations) which are symmetric with respect to these groups. 

An exception are the groups $\mathcal{C}_{mv}$, where $m$ is odd, as for these groups, we may combine results in \cite{jkt} and \cite{BSWWconing} to obtain a characterisation for symmetry-forced rigidity on the sphere even though there are no rotations which are symmetric with respect to $\mathcal{C}_{mv}$. Theorems~\ref{thm:spherecyclic}, \ref{thm:sphereinversion} and \ref{thm:spheredih}  provide characterisations for the groups $\mathcal{C}_m$, $\mathcal{C}_s$, $\mathcal{C}_i$, $\mathcal{C}_{mv}$ (with $m$ odd), $\mathcal{C}_{mh}$ (with $m$ odd) and $\mathcal{S}_{2m}$ (with $m$ even). The obstacles for  $\mathcal{C}_{mv}$, where $m$ is even, were described in Section~\ref{sec:combchar} (see also \cite{jkt}). This leaves the groups $\mathcal{C}_{mh}$, where $m$ is even, and $\mathcal{S}_{2m}$, where $m$ is odd. 

\begin{defn} \label{def:gainsparsitysph}
Let $(H,\psi)$ be a $\mathcal{C}_{mh}$-gain graph, where $m$ is even, or a  $\mathcal{S}_{2m}$-gain graph, where $m$ is odd. Then
$(H,\psi)$ is called \emph{$(2,3,1)^i$-gain-sparse} if 
\begin{itemize}
\item $|F|\leq 2|V(F)|-3$ for any nonempty $F\subseteq E(H)$ and $v\in V(F)$ with $\langle F \rangle_v =\mathcal{C}_1$ or $\langle F \rangle_v =\mathcal{C}_i$;
\item $|F|\leq 2|V(F)|-1$ otherwise.
\end{itemize}
A $(2,3,1)^i$-gain-sparse graph $(H,\psi)$ satisfying $|F|= 2|V(F)|-1$ is called \emph{$(2,3,1)^i$-gain-tight}. 
\end{defn}

\begin{con}\label{con:spherecmh}
Let $S$ be the group $\mathcal{C}_{mh}$, where $m$ is even, or the  group $\mathcal{S}_{2m}$, where $m$ is odd. Let $(G,p)$ be an $S$-generic realisation on $\S$ and let $(G_0,\psi)$ be the quotient $S$-gain graph of $G$. Then $(G,p)$ is $S$-isostatic if and only if $(G_0,\psi)$ is $(2,3,1)^i$-gain-tight.
\end{con}

\begin{table}
\caption{Summary of counts for the various symmetry groups on the sphere $\S$.} 
\begin{center}
\begin{tabular}{|c|c|c|c|c|}
\hline\noalign{\smallskip}
Group & Necessary count & Sufficient? \\
\noalign{\smallskip}\hline\noalign{\smallskip}
$\C_s$ & $(2,3,1)$-gain-tight & Theorem \ref{thm:spherecyclic}\\
$\C_m$ &  $(2,3,1)$-gain-tight & Theorem \ref{thm:spherecyclic}\\
$\C_i$ &  $(2,3,3)$-gain-tight & Theorem \ref{thm:sphereinv}\\
$\C_{mv}$, $m$ odd &  maximum $\mathscr{D}$-tight & Theorem \ref{thm:spherecyclic}\\
$\C_{mv}$, $m$ even &  maximum $\mathscr{D}$-tight & No, see \cite{jkt}\\
$\C_{mh}$, $m$ odd &  $(2,3,1)$-gain-tight & Theorem \ref{thm:spheredih}\\
$\C_{mh}$, $m$ even & $(2,3,1)^i$-gain-tight & Conjecture \ref{con:spherecmh}\\
$\S_{2m}$, $m$ odd &  $(2,3,1)^i$-gain-tight &  Conjecture \ref{con:spherecmh}\\
$\S_{2m}$, $m$ even &  $(2,3,1)$-gain-tight & Theorem \ref{thm:spheredih}\\
$\D_{m}$ &  $(2,3,0)^r$-gain-tight & ?\\
$\D_{mh}$ & Theorem~\ref{thm:maxwellnon-abelian} & ?\\
$\D_{md}$ &  Theorem~\ref{thm:maxwellnon-abelian} & ?\\
$\mathcal{T}, \mathcal{T}_h, \mathcal{T}_d, \mathcal{O}, \mathcal{O}_h, \mathcal{I},\mathcal{I}_h $, &  Theorem~\ref{thm:maxwellnon-abelian} & ?\\
\noalign{\smallskip}\hline
\end{tabular}
\end{center}
\end{table}

\subsection{The cylinder} For the cylinder, we offer the following conjectures.

\begin{con}\label{con:cylinderX}
Let $S$ be the cyclic group $\mathcal{C}_2$ representing $2$-fold rotation around an axis which is orthogonal to the $z$-axis. Let $(G,p)$ be an $S$-generic realisation on $\Y$ and let $(G_0,\psi)$ be the quotient $S$-gain graph of $G$. Then $(G,p)$ is $S$-isostatic if and only if $(G_0,\psi)$ is $(2,2,0)$-gain-tight.
\end{con}

This conjecture is of particular interest because it implies that there is a symmetry-preserving motion in a framework that counts to be generically minimally rigid without symmetry. We illustrate such a motion in the following example.

\begin{example}
Let $G_0$ be the gain graph consisting of a 0-gain $K_4$ on vertices $a,b,c,d$, together with an additional edge $(c,d)$ with gain 1. Then, with symmetry group $\mathcal{C}_2$ as defined in Conjecture~\ref{con:cylinderX}, the covering graph $G$ of $G_0$ consists of two vertex disjoint copies of $K_4$ joined by two edges. Theorem \ref{thm:surf3} implies that generic realisations of $G$ (without symmetry) are rigid on the cylinder $\Y$. However, the quotient $\C_2$-gain graph $G_0$ of $G$ satisfies $|E(G_0)|=7<8=2|V(G_0)|-0$. Thus, embedded $\mathcal{C}_2$-generically on $\Y$, as in Figure \ref{fig:cylindermotion}, Theorem \ref{thm:maxwell} implies the existence of a non-trivial continuous motion on $\Y$.
\end{example}

\begin{figure}
\begin{center}
\begin{tikzpicture}
\draw[black,thick] (0,4) ellipse (1.5 and 0.75);
\draw[black,thick] (0,0) ellipse (1.5 and 0.75);
 \draw[black,thick]
  (-1.5,0) -- (-1.5,4);
\draw[black,thick]
(1.5,0) -- (1.5,4);


\draw[dashed] (-.5,1.2) -- (.5,2.6);

\filldraw (1,3) circle (3pt);
\filldraw (1,.8) circle (3pt);
\filldraw[gray] (0,3.5) circle (3pt);
\filldraw[gray] (0,1.3) circle (3pt);

\filldraw (-1,.8) circle (3pt);
\filldraw (-1,3) circle (3pt);
\filldraw (0,.3) circle (3pt);
\filldraw (0,2.5) circle (3pt);

 \draw[black,thick]
  (1,3) -- (0,3.5) -- (-1,3) -- (0,2.5) -- (1,3) -- (-1,3);

 \draw[black,thick]
  (0,3.5) -- (0,2.5);

 \draw[black,thick]
  (1,.8) -- (0,1.3) -- (-1,.8) -- (0,.3) -- (1,.8) -- (-1,.8);

 \draw[black,thick]
  (0,1.3) -- (0,.3);

\draw[black,thick]
(-1,3) -- (-1,.8);

\draw[black,thick]
(1,3) -- (1,.8);

\end{tikzpicture}
\end{center}
\caption{A $\C_2$-symmetric framework $(G,p)$ on the cylinder $\Y$ which has a non-trivial symmetry-preserving motion, but whose underlying graph $G$ is generically isostatic on $\Y$ (without symmetry).  The grey joints are at the `back' of the cylinder.}
\label{fig:cylindermotion}
\end{figure}
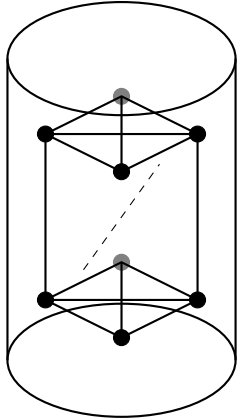

\begin{defn} \label{def:gainsparsitycyl}
Let $(H,\psi)$ be a $\mathcal{C}_{mh}$-gain graph, a  $\mathcal{C}_{mv}$-gain graph or a $\mathcal{S}_{2m}$-gain graph. Then
$(H,\psi)$ is called \emph{$(2,2,1)^r$-gain-sparse} if 
\begin{itemize}
\item $|F|\leq 2|V(F)|-2$ for any nonempty $F\subseteq E(H)$  and $v\in V(F)$ with $\langle F \rangle_v =\mathcal{C}_1$ or $\langle F \rangle_v =\mathcal{C}_{m'}$, $m'\leq m$;
\item $|F|\leq 2|V(F)|-1$ otherwise.
\end{itemize}
A $(2,2,1)^r$-gain-sparse graph $(H,\psi)$ satisfying $|F|= 2|V(F)|-1$ is called \emph{$(2,2,1)^r$-gain-tight}. 
\end{defn}

\begin{con}\label{thm:cylinderdihedral}
Let $S$ be the group  $\mathcal{C}_{mh}$, $\mathcal{C}_{mv}$ or $\mathcal{S}_{2m}$, where the rotational axis is the cylinder axis. Let $(G,p)$ be an $S$-generic framework in $\mathscr{R}^{\Y}_{(G,S,\theta)}$ with quotient $S$-gain graph $(G_0,\psi)$. Then $(G,p)$ is $S$-isostatic if and only if $(G_0,\psi)$ is $(2,2,1)^r$-gain tight.
\end{con}

We remark that $(2,2,1)^r$-gain-tight graphs are a special class of $(2,2,1)$-gain-tight graphs.

For the groups $\mathcal{D}_m$, $\mathcal{D}_{mh}$ and $\mathcal{D}_{md}$ it is plausible that the necessary conditions given in Theorem \ref{thm:maxwellnon-abelian} are also sufficient, but difficult to prove due to the $(2,2,0)$-gain-sparsity count (recall Section~\ref{subsec:mic}). Moreover, the remarks in Section \ref{subsec:sphere} give a warning that unexpected behaviour may arise, and hence we do not provide explicit conjectures for these groups.


\begin{table}
\caption{Summary of counts for the various symmetry groups on the cylinder $\Y$. We use `containing' as short hand for a plane containing the $z$-axis and `perpendicular' for a line perpendicular to the $z$-axis.}
\begin{center}
\begin{tabular}{|c|c|c|c|c|}
\hline\noalign{\smallskip}
Group & rotation axis & reflection plane & Necessary count & Sufficient? \\
\noalign{\smallskip}\hline\noalign{\smallskip}
$\C_s$ & - & containing & $(2,2,1)$-gain-tight & Theorem \ref{thm:cylinderreflection}\\
$\C_s$ & - & $z=0$ & $(2,2,1)$-gain-tight & Theorem \ref{thm:cylinderreflection}\\
$\C_m$ & $z$-axis & - & $(2,2,2)$-gain-tight & Theorem \ref{thm:cylindercyclic}\\
$\C_2$ & perpendicular & - & $(2,2,0)$-gain-tight & Conjecture \ref{con:cylinderX}\\
$\C_i$ & - & - & $(2,2,1)$-gain-tight & Theorem \ref{thm:cylinderreflection}\\
$\C_{mv}$ & $z$-axis & containing & $(2,2,1)^r$-gain-tight & Conjecture \ref{thm:cylinderdihedral}\\
$\C_{mh}$ & $z$-axis & $z=0$ & $(2,2,1)^r$-gain-tight & Conjecture \ref{thm:cylinderdihedral}\\
$\S_{2m}$ & $z$-axis & $z=0$ & $(2,2,1)^r$-gain-tight & Conjecture \ref{thm:cylinderdihedral}\\
$\D_{m}$ & $z$-axis & $z=0$ & Theorem~\ref{thm:maxwellnon-abelian} & ?\\
$\D_{mh}$ & $z$-axis & $z=0$ & Theorem~\ref{thm:maxwellnon-abelian} & ?\\
$\D_{md}$ & $z$-axis & $z=0$ & Theorem~\ref{thm:maxwellnon-abelian} & ?\\\noalign{\smallskip}\hline
\end{tabular}
\end{center}
\end{table}

\subsection{The cone}

For the cone, note that there is no symmetry group which turns a generically rigid framework with a free action on the vertex set on the cone into a flexible one.
However we do suggest the following conjecture. 

\begin{con}\label{con:coneX}
Let $S$ be the group $\mathcal{C}_2$ representing $2$-fold rotation about an axis perpendicular to the $z$-axis (i.e., perpendicular to the axis of the cone) or the group $\mathcal{C}_s$, where the mirror plane of the reflection  contains the $z$-axis. Let $(G,p)$ be an $S$-generic framework in $\mathscr{R}^{\C}_{(G,S,\theta)}$ with quotient $S$-gain graph $(G_0,\psi)$. Then $(G,p)$ is $S$-isostatic if and only if $(G_0,\psi)$ is $(2,1,0)$-gain-tight.
\end{con}

It is also plausible that the remaining groups ($\mathcal{C}_{mv}$, $\mathcal{D}_m$, $\mathcal{D}_{mh}$ and $\mathcal{D}_{md}$) can be understood similarly. For $\mathcal{C}_{mv}$ and $\mathcal{D}_m$, for example, the conjecture would be that $(2,1,0)^r$-gain tightness characterises the symmetry-forced isostatic frameworks (where $(2,1,0)^r$-gain-tightness is defined analogously to Def.~\ref{def:gainsparsitycyl}), since $k_{\langle F\rangle_v}=1$  if $\langle F\rangle_v$ is a purely rotational group (with the rotational axis being the axis of the cone) and $k_{\langle F\rangle_v}=0$ otherwise.   

\begin{table}
\caption{Summary of counts for the various symmetry groups on the cone $\C$. We use `containing' as short hand for a plane containing the $z$-axis and `perpendicular' for a line perpendicular to the $z$-axis.}
\begin{center}
\begin{tabular}{|c|c|c|c|c|}
\hline\noalign{\smallskip}
Group & rotation axis & reflection plane & Necessary count & Sufficient? \\
\noalign{\smallskip}\hline\noalign{\smallskip}
$\C_s$ & - & containing & $(2,1,0)$-gain-tight & Conjecture \ref{con:coneX}\\
$\C_s$ & - & $z=0$ & $(2,1,1)$-gain-tight & Theorem \ref{thm:conecyclic}\\
$\C_m$ & $z$-axis & - & $(2,1,1)$-gain-tight & Theorem \ref{thm:conecyclic}\\
$\C_2$ & perpendicular & - & $(2,1,0)$-gain-tight & Conjecture \ref{con:coneX}\\
$\C_i$ & - & - & $(2,1,1)$-gain-tight & Theorem \ref{thm:conecyclic}\\
$\C_{mv}$ & $z$-axis & containing & $(2,1,0)^r$-gain-tight & ?\\
$\C_{mh}$ & $z$-axis & $z=0$ & $(2,1,1)$-gain-tight & Theorem \ref{thm:conecyclic}\\
$\S_{2m}$ & $z$-axis & $z=0$ & $(2,1,1)$-gain-tight & Theorem \ref{thm:conecyclic}\\
$\D_{m}$ & $z$-axis & $z=0$ & $(2,1,0)^r$-gain-tight & ?\\
$\D_{mh}$ & $z$-axis & $z=0$ & Theorem~\ref{thm:maxwellnon-abelian} & ?\\
$\D_{md}$ & $z$-axis & $z=0$ & Theorem~\ref{thm:maxwellnon-abelian} & ?\\\noalign{\smallskip}\hline
\end{tabular}
\end{center}
\end{table}

\subsection{Non-free actions}

For the cylinder, rotational symmetry about the $z$-axis is necessarily a free group action. However, for reflection symmetry $s$ about a plane containing the $z$-axis, for example, this plane intersects the cylinder in 2 disjoint lines. If we allow symmetry-generic realisations to include joints lying on these lines then we must adapt the counts and the orbit-surface rigidity matrix accordingly. Any such  joint which is `fixed' by the reflection $s$ will have only 1 degree of freedom, as it has to stay on the reflection plane of $s$ and on the cylinder. While  we do not expect any new complication to arise in this more general situation, the proofs will become significantly more messy due to the reduced number of columns in the orbit-surface matrix for fixed vertices \cite{BSWWorbit}. 
Similar observations apply to the cone.

\subsection{Surfaces with 1 isometry}

In \cite{nop} Laman type theorems were developed for any surface with exactly 1 isometry. Here we have concentrated on the cone. We expect that our methods are adaptable to any other surface with one isometry (such as tori, hyperboloids and paraboloids).

However, we highlight that the same group acting on two different surfaces (that a priori have the same number of trivial motions) can give different numbers of symmetric trivial motions and hence different combinatorial counts with an example theorem.

\begin{thm}[Reflection symmetry on the elliptical cylinder]
Let $\F$ be an elliptical cylinder about the $z$-axis and let $\C_s$ be generated by a reflection whose mirror plane contains the $z$-axis. Let $(G,p)$ be a framework in $\mathscr{R}^{\F}_{(G,\C_s,\theta)}$ with quotient $\C_s$-gain graph $(G_0,\psi)$.
Then $(G,p)$ is $\C_s$-isostatic if and only if $(G_0,\psi)$ is $(2,1,1)$-gain-tight.
\end{thm}

This is in contrast with the same group for the cone, where Theorem \ref{thm:maxwell} implies that the $\C_s$-gain graph must be $(2,1,0)$-gain-tight.

\subsection{Incidental symmetry} \label{subsec:incsym}

In this paper, we focused on the \emph{symmetry-forced} rigidity of symmetric frameworks on surfaces. More generally, one may ask when an $S$-symmetric framework $(G,p)$ on a surface $\M$ is not only $S$-symmetric infinitesimally rigid (i.e., $(G,p)$ has no non-trivial $S$-symmetric infinitesimal motion), but also infinitesimally rigid (i.e., $(G,p)$ has no non-trivial infinitesimal motion at all).

 A fundamental result in the rigidity analysis of (`incidentally') symmetric frameworks in Euclidean $d$-space is that the rigidity matrix $R(G,p)$ of an $S$-symmetric framework $(G,p)$ can be transformed into a block-decomposed form, where each block $R_i(G,p)$ corresponds to an irreducible representation $\rho_i$ of the group $S$ \cite{KG2,BS2}. This breaks up the rigidity analysis of $(G,p)$ into a number of independent subproblems. In fact, the symmetry-forced rigidity properties of $(G,p)$ are described by the block matrix $R_1(G,p)$ corresponding to the trivial irreducible representation $\rho_1$ of $S$. In \cite{BSWWorbit} the orbit rigidity matrix was derived to simplify the symmetry-forced rigidity analysis of Euclidean frameworks (the orbit rigidity matrix $O(G,p,S)$ is equivalent to the block matrix $R_1(G,p)$, but it can be constructed without using any methods from group representation theory).  Moreover, in the very recent paper \cite{schtan}, an `orbit rigidity matrix' was established for each of the blocks  $R_i(G,p)$, and these new tools were successfully used to characterise  $S$-generic infinitesimally rigid frameworks for a number of point groups $S$.
 
These methods can clearly be extended to analyse the infinitesimal rigidity of $S$-generic frameworks on surfaces. However, note that for each surface $\M$ and each symmetry group $S$ considered in this paper (except for the groups $\C_s$ and $\C_i$ on the cylinder $\Y$), there always exists an irreducible representation $\rho_i$ of $S$ with the property that there is no trivial $\rho_i$-symmetric infinitesimal motion (i.e., there is no trivial motion in the kernel of the corresponding block matrix $R_i(G,p)$). Therefore, we need to deal with a $(2,k,0)$-gain-sparsity count in each of these cases (more precisely, $(2,3,0)$-gain-sparsity for the sphere, $(2,2,0)$-gain-sparsity for the cylinder, and $(2,1,0)$-gain-sparsity for the cone), which gives rise to the difficulties outlined in Sections~\ref{subsec:sphere} and \ref{subsec:mic}.

 For the groups $\C_s$ and $\C_i$ on the cylinder $\Y$, however, the block-decomposed orbit-surface rigidity matrix consists of two blocks, one corresponding to the trivial representation $\rho_1$ of the group (this block is equivalent to the orbit-surface rigidity matrix) and one corresponding to  the other irreducible representation $\rho_2$ (this block is equivalent to an `anti-symmetric' orbit-surface  rigidity matrix), and for both of these blocks, one needs to consider the same $(2,2,1)$-gain-sparsity count to test whether the block has maximal rank. Therefore, we propose the following  conjecture.

\begin{con}\label{con:incsym}
Let $S$ be the group $\C_s$ or the group $\C_i$, and let $(G,p)$ be an $S$-generic framework on the cylinder  $\Y$. 
Then $(G,p)$ is isostatic if and only if $(G,p)$ is $S$-isostatic.
\end{con}

\subsection{Algorithmic implications}

We expect that $(k,\ell,m)$-gain-sparsity can be checked deterministically in polynomial time whenever $(k,\ell,m)$-gain-sparsity is a matroidal property. This would confirm that our theorems provide efficient combinatorial descriptions of symmetry-forced rigidity. We leave the exact details to the reader, but remark that: the quotient gain graph of a given symmetric simple graph can clearly be obtained in polynomial time;  $(2,\ell,\ell)$-gain-sparsity for $\ell=0,1,2,3$ is known to be polynomial time computable \cite{L&S}; 
the case corresponding to Theorem \ref{thm:spherecyclic} has been considered \cite[Section $10$]{jkt}; and the remaining cases considered in this paper can be checked using similar arguments.

\section*{Acknowledgements}

We would like to thank the anonymous referee for a very careful reading and numerous helpful suggestions.

%




\section*{Appendix}

We adopt the set-up of Lemma~\ref{lem:NOP}. We claimed that
 the sequence of well-behaved frameworks in Lemma~\ref{lem:NOP}  takes $u^k$ to $$(u_{{x_1}(1)},u_{x_1(1)},u_{x_1(1)},u_{x_1(1)},u_{x_2(1)},u_{x_2(1)},u_{x_2(1)},u_{x_2(1)},\dots, u_{x_{|S|}(1)},u_{x_{|S|}(1)},u_{x_{|S|}(1)},u_{x_{|S|}(1)}).$$   

Consider an $S$-symmetric infinitesimal motion $u$ of $(H,p)$.
By definition, $u_{{x_j}(i)}\cdot\N_{{x_j}(i)} = 0$, where $\N_{{x_j}(i)}$ is the unit normal at $p_{{x_j}(i)}$ and $(p_{x_j(i)}-p_{x_j(i')})\cdot (u_{x_j(i)}-u_{x_j(i')}) = 0$ for $1 \leq i,i'  \leq 4$, $1 \leq j  \leq |S|$.
Since, for each $j$, $(K^j, p|_j)$  is infinitesimally rigid in $\mathbb{R}^3$, the restricted motion $u|_{K^j}$ is equal to $u^{(j)}_a + u^{(j)}_b$, where $u^{(j)}_b$ is
determined by translation by the vector $b^j$ and where $u^{(j)}_a$  corresponds to an infinitesimal
rotation about a line through $p_{x_j(1)}$ with direction vector $a^j$. 

We will show that  the subsequence  $u^k|_{K^j}$ converges to the limit $(u_{x_j(1)},u_{x_j(1)}, u_{x_j(1)},u_{x_j(1)})$ for some $j\in \{1,\dots, |S|\}$. By symmetry, this suffices to prove the claim. The argument that follows is a direct adaptation of Lemma~\cite[Lemma 5.4]{nop}.

We have $u_{x_j(1)}=b^j$ and we may choose $a^j$ so that $u_{x_j(i)}-u_{x_j(1)}=(p_{x_j(i)}-p_{x_j(1)}) \times a^j$ for $i=2,3,4$. By elementary manipulation of the cross product this gives us the equations \begin{equation} \label{eqn:app1} a^j \cdot (\N_{x_j(i)} \times (p_{x_j(i)}-p_{x_j(1)}))+b^j\cdot \N_{x_j(i)}=0 \quad \textrm{ for } i=2,3,4. \tag{A.1}\end{equation}
We have $$\frac{d(p|_j)}{d(s_{x_j(i)})}= \hat{s}_j+\kappa_{s_j} s_{x_j(i)}\hat{n}_j + r_{s_{x_j(i)}} \,  \textrm { and } \, \frac{d(p|_j)}{d(t_{x_j(i)})}= \hat{t}_j+\kappa_{t_j} t_{x_j(i)}\hat{n}_j + r_{t_{x_j(i)}}$$ where $\|r_{s_{x_j(i)}}\|$
and $\|r_{t_{x_j(i)}}\|$ are of order  $\|(s_{x_j(i)},t_{x_j(i)})\|^2$, and also
the normal vectors $$\N(s_{x_j(i)},t_{x_j(i)})=\frac{d(p|_j)}{d(s_{x_j(i)})}(s_{x_j(i)},t_{x_j(i)}) \times \frac{d(p|_j)}{d(t_{x_j(i)})}(s_{x_j(i)},t_{x_j(i)}).$$
This gives us, in a neighbourhood of $p_{x_j(1)}$, a continuous choice of  \begin{equation} \label{eqn:app2} \N(s_{x_j(i)},t_{x_j(i)})= \hat{s}_j \times \hat{t}_j + \kappa_{t_j} t_{x_j(i)} \hat{s}_j \times \hat{n}_j + \kappa_{s_j} s_{x_j(i)} \hat{n}_j \times \hat{t}_j + \overline{r}_j   = \hat{n}_j - (\kappa_{t_j} t_{x_j(i)} \hat{t}_j + \kappa_{s_j} s_{x_j(i)} \hat{s}_j )  + \overline{r}_j, \tag{A.2} \end{equation}
where $\|\overline{r}_j\|$ is of order $\|(s_{x_j(i)},t_{x_j(i)})\|^2$.
At the point $p_{x_j(i)}^\epsilon=p|_j(\epsilon s_{x_j(i)},\epsilon t_{x_j(i)})$, by (\ref{eqn:app2}), these normals take the form $$ \N_{x_j(i)}^\epsilon=\N(\epsilon s_{x_j(i)},\epsilon t_{x_j(i)})= \hat{n}_j - \epsilon(\kappa_{t_j} t_{x_j(i)}\hat{t}_j,+\kappa_{s_j} s_{x_j(i)}\hat{s}_j)+\overline{r}_{x_j(i)}^\epsilon,$$
where $\| \overline{r}_{x_j(i)}^\epsilon\|=O(\epsilon^2)$. By (\ref{eqn:app1}), an infinitesimal motion $u^{\epsilon}$ of $(K^j,p|_j^\epsilon)$ on $\M$ has associated equations  \begin{equation} \label{eqn:app3} a^{j,\epsilon} \cdot (\N_{x_j(i)}^\epsilon \times (p_{x_j(i)}^\epsilon-p_{x_j(1)}))+b^{j,\epsilon}\cdot \N_{x_j(i)}^\epsilon=0 \quad \textrm{ for } i=2,3,4. \tag{A.3} \end{equation} We can now identify the cross product $ \N_{x_j(i)}^\epsilon \times (p_{x_j(i)}^\epsilon-p_{x_j(1)})$ as
$$\hat{n}_j - \epsilon(\kappa_{t_j} t_{x_j(i)}\hat{t}_j,+\kappa_{s_j} s_{x_j(i)}\hat{s}_j) \times (\epsilon(s_{x_j(i)}\hat{s}_j+t_{x_j(i)}\hat{t}_j)+\frac{1}{2}\epsilon^2(\kappa_{s_j} s_{x_j(i)}^2+\kappa_{t_j} t_{x_j(i)}^2)\hat{n}_j)+R_{x_j(i)}^\epsilon $$ with $\| R_{x_j(i)}^\epsilon\|= O(\epsilon^3)$. We may assume, by passing to a subsequence, that $\epsilon$ runs through a sequence $\epsilon_k$ tending to zero and that the associated unit norm motions $u^\epsilon$ converge to a limit motion $u^0$ of the degenerate framework
 $(K^j,(p_{x_j(1)},p_{x_j(1)},p_{x_j(1)}, p_{x_j(1)}))$ on $\M$. 

Let $b^{j,0}=u_{x_j(1)}^0$ and let $b^{j,\epsilon}$ and $a^{j,\epsilon}$ be the associated vectors. 
While $b^{j,\epsilon}=u_{x_j(1)}^\epsilon$ converges to $b^{j,0}$, as $\epsilon=\epsilon_k\rightarrow 0$, the sequence $(a^{j,\epsilon_k})$ may be unbounded.
However, in view of the three equations $u_{x_j(i)}^\epsilon-u_{x_j(1)}^\epsilon=(p_{x_j(i)}^\epsilon-p_{x_j(1)})\times a^{j,\epsilon}$ and the definition of $p_{x_j(i)}^\epsilon$ it follows that $\|a^{j,\epsilon_k}\|$ is at worst of order $1/\epsilon_k$. We shall show that $\|a^{j,\epsilon_k}\|$ is in fact bounded and so, from the equation above, the desired conclusion follows.
Using (\ref{eqn:app3}) we see that
$$ a^{j,\epsilon} \cdot (\epsilon s_{x_j(i)} \hat{t}_j - \epsilon t_{x_j(i)} \hat{s}_j - \epsilon^2 s_{x_j(i)}t_{x_j(i)}(\kappa_{s_j}-\kappa_{t_j})\hat{n}_j +R_{x_j(i)}^\epsilon)-\kappa_{s_j}(b^{j,\epsilon} \cdot \hat{s}_j)\epsilon s_{x_j(i)} - \kappa_{t_j}(b^{j,\epsilon} \cdot \hat{t}_j)\epsilon t_{x_j(i)} + R_{x_j(i)}^\epsilon $$ is equal to zero, where $R_{x_j(i)}^\epsilon=\|b^{j,\epsilon} \cdot R_{x_j(i)}^\epsilon\|=O(\epsilon^2)$.
Note that $\|a^{j,\epsilon}\cdot R_{x_j(i)}^\epsilon\|=O(\epsilon^2)$ and introduce coordinates $a_s^{j,\epsilon}, a_t^{j,\epsilon}, a_n^{j,\epsilon}$ for $a^{j,\epsilon}$. Inputting these coordinates into the equation above and canceling a factor of $\epsilon$, it follows that
$$ (a_s^{j,\epsilon}\hat{s}_j+a_t^{j,\epsilon} \hat{t}_j+a_n^{j,\epsilon}\hat{n}_j)\cdot (s_{x_j(i)}\hat{t}_j-t_{x_j(i)}\hat{s}_j-\epsilon s_{x_j(i)}t_{x_j(i)}(\kappa_{s_j}-\kappa_{t_j})\hat{n}_j)-\kappa_{s_j}(b^{j,\epsilon}\cdot \hat{s}_j)s_{x_j(i)}-\kappa_{t_j}(b^{j,\epsilon}\cdot \hat{t}_j)t_{x_j(i)}$$ is of order $O(\epsilon)$ for $i=2,3,4$. Thus 
$$ -a_s^{j,\epsilon} t_{x_j(i)} + a_t^{j,\epsilon} s_{x_j(i)} - a_n^{j,\epsilon} \epsilon s_{x_j(i)}t_{x_j(i)}(\kappa_{s_j}-\kappa_{t_j})=d_{x_j(i)}^{j,\epsilon} $$ for $i=2,3,4,$ where $$d_{x_j(i)}^{j,\epsilon}=b^{j,\epsilon} \cdot (\kappa_{s_j} s_{x_j(i)}\hat{s}_j+\kappa_{t_j} t_{x_j(i)} \hat{t}_j)+X_{x_j(i)}^{j,\epsilon}$$ with $X_{x_j(i)}^{j,\epsilon}=O(\epsilon)$. Let $\eta_j=\epsilon(\kappa_{s_j}-\kappa_{t_j})$ for $i=2,3,4$, let $A_{j,\epsilon}$ be the matrix
\[ \begin{pmatrix}-t_{x_j(2)} & s_{x_j(2)} & -s_{x_j(2)}t_{x_j(2)}\eta_j\\ -t_{x_j(3)} & s_{x_j(3)} & -s_{x_j(3)}t_{x_j(3)}\eta_j \\ -t_{x_j(4)} & s_{x_j(4)} & -s_{x_j(4)}t_{x_j(4)}\eta_j \end{pmatrix} \] and note that det$A_{j,\epsilon}=C\epsilon$ for some nonzero constant $C$. By Cramer's rule we have \[ a_n^{j,\epsilon} = (\mbox{det }A_{j,\epsilon})^{-1}\mbox{det} \begin{pmatrix}-t_{x_j(2)} & s_{x_j(2)} & d_{x_j(2)}^\epsilon\\ -t_{x_j(3)} & s_{x_j(3)} & d_{x_j(3)}^\epsilon\\-t_{x_j(4)} & s_{x_j(4)}& d_{x_j(4)}^\epsilon \end{pmatrix} =(\mbox{det }A_{j,\epsilon})^{-1}\mbox{det}\begin{pmatrix} -t_{x_j(2)}& s_{x_j(2)} & X_{x_j(2)}^\epsilon\\ -t_{x_j(3)} & s_{x_j(3)} & X_{x_j(3)}^\epsilon\\ -t_{x_j(4)} & s_{x_j(4)} & X_{x_j(4)}^\epsilon\end{pmatrix}\] since the column for $d_{x_j(i)}^{j,\epsilon}-X_{x_j(i)}^{j,\epsilon}$ is a linear combination of the first two columns. It follows that the sequence $a_n^{j,\epsilon_k}$ is bounded. The boundedness of $(a_s^{j,\epsilon_k})$, and similarly $(a_t^{j,\epsilon_k})$, follows more readily, since \[ a_s^{j,\epsilon}=(\mbox{det }A_{j,\epsilon})^{-1}\mbox{det}\begin{pmatrix} d_{x_j(2)}^\epsilon & s_{x_j(2)} & -s_{x_j(2)}t_{x_j(2)}\eta_j\\ d_{x_j(3)}^\epsilon & s_{x_j(3)} & -s_{x_j(3)}t_{x_j(3)}\eta_j\\ d_{x_j(4)}^\epsilon & s_{x_j(4)} & -s_{x_j(4)}t_{x_j(4)} \eta_j \end{pmatrix} \] and the $\epsilon$ factors cancel. Thus, the sequence of vectors $a^{j,\epsilon_k}$ is bounded, as desired.

\end{document}